\documentclass[11pt,twoside,reqno,centertags,draft]{amsart}
\usepackage{amsmath,amsthm,amsfonts,amssymb,epsf,cite}

\makeatletter

\@addtoreset{figure}{section}

\theoremstyle{plain}
\newtheorem{theorem}{Theorem}
\@addtoreset{theorem}{section}

\newtheorem{lemma}{Lemma}
\@addtoreset{lemma}{section}

\newtheorem{corollary}{Corollary}
\@addtoreset{corollary}{section}

\newtheorem{condition}{Condition}
\@addtoreset{condition}{section}

\newtheorem*{condition_a}{Condition~\ref{condK1}$'$}

\theoremstyle{definition}

\newtheorem{definition}{Definition}
\@addtoreset{definition}{section}

\@addtoreset{example}{section}

\newtheorem{remark}{Remark}
\@addtoreset{remark}{section}

\makeatother

\renewcommand{\Im}{{\rm Im\,}}
\renewcommand{\Re}{{\rm Re\,}}

\newcommand{\supp}{{\rm supp\,}}

\newcommand{\Orb}{{\rm Orb}}

\newcommand{\const}{{\rm const}}
\newcommand{\dist}{{\rm dist}}

\begin{document}
\title{The consistency conditions and the smoothness of
generalized solutions of nonlocal elliptic problems }
\thanks{Accepted for publication: September 2005.}
\date{}
\maketitle
\vspace{ -1\baselineskip}

{\small
\begin{center}
{\sc Pavel Gurevich}
\end{center}
}

\numberwithin{equation}{section}
\allowdisplaybreaks

\begin{quote}
\footnotesize
{\bf Abstract.}
We study smoothness of generalized solutions of nonlocal
elliptic problems in plane bounded
domains with piecewise smooth boundary. The case where the
support of nonlocal terms can
intersect the boundary is considered. We find conditions that
are necessary and sufficient for
any generalized solution to possess an appropriate smoothness
(in terms of Sobolev spaces).
Both homogeneous and nonhomogeneous nonlocal
boundary-value conditions are studied.

\end{quote}

\section{Introduction}

Nonlocal elliptic problems arise in various areas such as plasma
theory~\cite{Sam}, biophysics, theory of diffusion
processes~\cite{Feller,Ventsel,SatoUeno,Taira,GalSkub}, control
theory~\cite{BensLions,Amann}, and so on.

In the one-dimensional case, nonlocal problems were studied since
the beginning of the 20th century by Sommerfeld~\cite{Sommerfeld},
Picone~\cite{Picone}, Tamarkin~\cite{Tamarkin}, etc. In the
two-dimensional case, one of the first works was due to
Carleman~\cite{Carleman}, who treated the problem of finding a
harmonic function, in a plane bounded domain, satisfying a
nonlocal condition which connects the values of the unknown
function at different points of the boundary. Further
investigation of elliptic problems with transformations mapping a
boundary onto itself has been carried out by Vishik~\cite{Vishik},
Browder~\cite{Browder}, Beals~\cite{Beals},
Antonevich~\cite{Antonevich}, and others.

In 1969 Bitsadze and Samarskii~\cite{BitzSam} considered the
following nonlocal problem arising in plasma theory: to find a
function $u(y_1, y_2)$ harmonic on the rectangular
$G=\{y\in\mathbb R^2: -1<y_1<1,\ 0<y_2<1\}$, continuous on
$\overline{G}$, and satisfying the relations
 \begin{align*}
  u(y_1, 0)&=f_1(y_1),\quad u(y_1, 1)=f_2(y_1),\quad -1<y_1<1,\\
  u(-1, y_2)&=f_3(y_2),\quad u(1, y_2)=u(0, y_2),\quad 0<y_2<1,
 \end{align*}
where $f_1, f_2, f_3$ are given continuous functions. This problem
was solved in~\cite{BitzSam} by reducing it to a Fredholm integral
equation and by using the maximum principle. For arbitrary domains
and for general nonlocal conditions, such a problem was formulated
as an unsolved one (see also~\cite{krall}). Different
generalizations of nonlocal problems with transformations mapping
the boundary inside the closure of a domain were studied by many
authors~\cite{ZhEid,RSh,Kishk,GM}.

The most complete theory for elliptic equations of order $2m$ with
general nonlocal conditions was developed by Skubachevskii and his
pupils~\cite{SkMs83,SkMs86,SkDu90,SkDu91,SkJMAA,KovSk,SkBook,GurRJMP03,GurRJMP04}:
a classification with respect to types of nonlocal conditions was
suggested, the Fredholm solvability in the corresponding spaces
was investigated, and asymptotics of solutions near special
conjugation points was obtained. One can find other relevant
references and description of applications in~\cite{SkBook}.

In the present paper, we consider a little-studied question
concerning the smoothness of solutions for nonlocal elliptic
problems. For simplicity, we study nonlocal perturbations of the
Dirichlet problem for elliptic second-order equations. However,
the approach we are developing is also applicable to elliptic
equations of order $2m$ with general nonlocal conditions.

It appears that the most difficult situation is that where the
support of nonlocal terms can intersect the boundary of a
domain~\cite{SkMs86, SkRJMP}. In this case, solutions of nonlocal
problems can have power-law singularities near some points of the
boundary even if the right-hand side is infinitely differentiable
and the boundary is infinitely smooth. It follows from our results
that solutions of nonlocal problems can have power-law
singularities even if the support of nonlocal terms lies strictly
inside a domain. For this reason, we use special weighted spaces
to study nonlocal problems. These spaces were originally proposed
by Kondrat'ev~\cite{KondrTMMO67} to study elliptic boundary-value
problems in nonsmooth domains.

Note that smoothness of solutions for ``local'' elliptic problems
in nonsmooth domains is studied rather thoroughly
(see~\cite{KondrTMMO67,MP,Schulze,CD} and others); here principal
difficulties are related to the presence of special singular
points on the boundary of a domain. In the theory of nonlocal
problems, there appear principally different difficulties:
violation of smoothness of solutions is connected not only with
the fact that the boundary may be nonsmooth but also with the
presence of nonlocal terms in the boundary-value conditions.

Consider the following example. Let $\partial
G=\Gamma_1\cup\Gamma_2\cup\{g,h\}$, where $\Gamma_i$ are open (in
the topology of $\partial G$) $C^\infty$-curves; $g,h$ are the end
points of the curves $\overline{\Gamma_1}$ and
$\overline{\Gamma_2}$. Suppose that the domain $G$ is the plane
angle of opening $\pi$ in some neighborhood of each of the points
$g$ and $h$. We deliberately take a smooth domain in this example
to illustrate how the nonlocal terms can affect the smoothness of
solutions. Consider the following nonlocal problem in the domain
$G$:
\begin{equation}\label{eqIntroPinG}
 \Delta u=f_0(y)\quad (y\in G),
\end{equation}
\begin{equation}\label{eqIntroBinG}
\begin{aligned}
&u|_{\Gamma_1}+b_1(y)
u\bigl(\Omega_{1}(y)\bigr)\big|_{\Gamma_1}+a(y)
u\bigl(\Omega(y)\bigr)\big|_{\Gamma_1}=f_1(y)\quad
(y\in\Gamma_1),\\
&u|_{\Gamma_2}+b_2(y)
u\bigl(\Omega_{2}(y)\bigr)\big|_{\Gamma_2}=f_2(y)\quad
 (y\in\Gamma_2).
\end{aligned}
\end{equation}
Here $b_1$, $b_2$, and $a$ are real-valued $C^\infty$-functions;
$\Omega_i$ ($\Omega$) are $C^\infty$-diffeo\-mor\-phisms taking
some neighborhood ${\mathcal O}_i$ (${\mathcal O}_1$) of the curve
$\Gamma_i$ ($\Gamma_1$) onto the set $\Omega_i({\mathcal O}_i)$
($\Omega({\mathcal O}_1)$) in such a way that
$\Omega_i(\Gamma_i)\subset G$, $\Omega_i(g)=g$, $\Omega_i(h)=h$,
and the transformation $\Omega_i$, near the points $g, h$, is the
rotation of the boundary $\Gamma_i$ through the angle $\pi/2$
inwards the domain $G$ (respectively, $\Omega(\Gamma_1)\subset G$,
$\overline{\Omega(\Gamma_1)}\cap\{g,h\}=\varnothing$, and the
approach of the curve $\Omega(\overline{\Gamma_1})$ to the
boundary $\partial G$ can be arbitrary, cf.~\cite{SkMs86,
SkDu91}), see Fig.~\ref{figEx1}.
\begin{figure}[ht]
{ \hfill\epsfbox{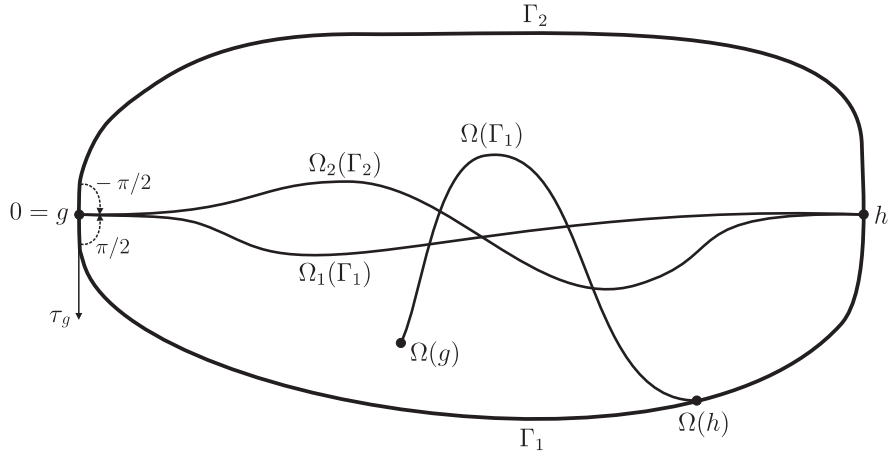}\hfill\ } \caption{Domain $G$ with
boundary $\partial G=\Gamma_1\cup\Gamma_2\cup\{g,h\}$.}
   \label{figEx1}
\end{figure}

We say that $g$ and $h$ are the {\em points of conjugation of
nonlocal conditions} because they divide the curves on which
different nonlocal conditions are set. The closure of the set
$$
\bigcup_{i=1,2}\{y\in\Omega_i(\Gamma_i):\
b_i(\Omega_i^{-1}(y))\ne0\}\cup\{y\in\Omega(\Gamma_1):\
a(\Omega^{-1}(y))\ne0\}
$$
is referred to as the {\em support of nonlocal terms}. It is clear
that, if $b_1(y)=a(y)=0$ for $y\in\Gamma_1$ and $b_2(y)=0$ for
$y\in\Gamma_2$, then the support of nonlocal terms is the empty
set. If, say, $b_1(y),a(y)\ne0$ for $y\in\Gamma_1$ and
$b_2(y)\ne0$ for $y\in\Gamma_2$, then the support of nonlocal
terms is the set
$\overline{\Omega_1(\Gamma_1)}\cup\overline{\Omega_2(\Gamma_2)}\cup\overline{\Omega(\Gamma_1)}$.

Denote by $W^k(G)=W^k_2(G)$ the Sobolev space. We say that a
function $u\in W^1(G)$ is a {\em generalized solution} of
problem~\eqref{eqIntroPinG}, \eqref{eqIntroBinG} with right-hand
side $f_0\in L_2(G)$, $f_i\in W^{1/2}(\Gamma_i)$ if $u$ satisfies
Eq.~\eqref{eqIntroPinG} in the sense of distributions and nonlocal
conditions~\eqref{eqIntroBinG} in the sense of traces. Using the
notation of problem~\eqref{eqIntroPinG}, \eqref{eqIntroBinG}, we
can formulate the main questions of our paper.
\begin{enumerate}
\item Find a condition on the
right-hand sides $f_0\in L_2(G)$, $f_i\in W^{3/2}(\Gamma_i)$ and
on the coefficients $b_1$, $b_2$, and $a$ which is necessary and
sufficient for {\em any} generalized solution of
problem~\eqref{eqIntroPinG}, \eqref{eqIntroBinG} to belong to the
space $W^2(G)$.
\item The same question for homogeneous nonlocal conditions, $\{f_i\}=0$.
\end{enumerate}

It is relatively easy to prove that any generalized solution of
problem~\eqref{eqIntroPinG}, \eqref{eqIntroBinG} belongs to the
space $W^2$ outside an arbitrarily small neighborhood of the
points $g$ and $h$ (see Sec.~\ref{sectNoEigen}). Clearly, the
behavior of solutions near the points $g$ and $h$ is affected by
the behavior of the coefficients $b_1$, $b_2$, and $a$ near these
points. However, the influence of the coefficients $b_i$ is
principally different from the influence of the coefficient $a$.
This phenomenon is explained by the fact that the coefficients
$b_i$ (for $y$ being in a small neighborhood of the points $g$ and
$h$) correspond to nonlocal terms supported {\em near} the set
$\{g,h\}$ (in the general case, such terms correspond to operators
$\mathbf B_i^1$), whereas the coefficient $a$ corresponds to a
nonlocal term supported {\em outside} some neighborhood of the set
$\{g,h\}$ (in the general case, such terms correspond to abstract
operators $\mathbf B_i^2$). What we give below is a scheme for the
investigation of smoothness of generalized solutions near the
point $g$ (this scheme is realized in
Secs.~\ref{sectStatement}--\ref{sectGener} for the general case
and in Sec.~\ref{sectExample} for the particular case of
problem~\eqref{eqIntroPinG}, \eqref{eqIntroBinG}).
\begin{description}
\item[Step~1] We construct a model nonlocal problem,
with a parameter, for ordinary differential equation corresponding
to the point $g$. The structure of nonlocal conditions in the
model problem depends only on the values of the coefficients
$b_i(g)$, $i=1,2$. (Section~\ref{sectStatement}.)
\item[Step~2] We consider the values $b_i(g)$ for which the band
$-1\le\Im\lambda<0$ contains no eigenvalues of the model problem.
In this case, any generalized solution belongs to  $W^2$ near the
point $g$. Note that, in this case, we impose no additional
restrictions on the right-hand side or the coefficients $b_1$,
$b_2$, and $a$. (Section~\ref{sectNoEigen} and
Theorem~\ref{thExuinW^2NoEigen}.)
\item[Step~3] We consider the values $b_i(g)$ for which the band
$-1\le\Im\lambda<0$ contains only the proper eigenvalue
$\lambda=-i$ of the model problem (see
Definition~\ref{defRegEigVal}). This is the most complicated
situation, which we call a ``border case.'' In this case, any
generalized solution belongs to $W^2$ near the point $g$ if and
only if the coefficients $b_1$, $b_2$, and $a$ satisfy a certain
consistency condition near the point $g$. The type of the
consistency condition depends on whether we consider homogeneous
or nonhomogeneous nonlocal conditions. In the latter case, the
consistency conditions must also be imposed on the right-hand side
$\{f_i\}$. (Section~\ref{sectProperEigen} and
Theorems~\ref{thExSmoothfne0}, \ref{thExSmoothf0}, and
Corollary~\ref{corExSmoothf0Homog}.)
\item[Step~4] We consider the values $b_i(g)$ for which the band
$-1\le\Im\lambda<0$ contains an improper eigenvalue of the model
problem (see Definition~\ref{defRegEigVal}). In this case, for any
coefficient $a$, one can find right-hand sides $f_0\in L_2(G)$,
$\{f_i\}=0$ ($f_0$ depends on the behavior of the coefficients
$b_i$ near the point $g$ and does not depend on the coefficient
$a$) and construct the corresponding generalized solution $u\in
W^1(G)$ such that $u$ does not belong to $W^2$ near the point $g$.
(Section~\ref{sectImproperEigen} and
Theorem~\ref{thExNonSmoothfne0}.)
\end{description}

It turns out that the smoothness of generalized solutions
preserves if $b_1(g)+b_2(g)\le-2$ or $b_1(g)+b_2(g)>0$ and can be
violated if $-2<b_1(g)+b_2(g)<0$. If $b_1(g)+b_2(g)=0$, we have
the border case. The necessary condition that the smoothness
preserve is the validity of a consistency condition imposed on the
right-hand side $\{f_i\}$ (see~\eqref{eqExConsistencyf}). Let us
show that the {\em presence of variable coefficients in nonlocal
conditions} may affect the smoothness of generalized solutions.
For simplicity, we assume that $a(y)\equiv0$. Let
condition~\eqref{eqExB2vB1CConsistency2} hold; in particular, let
$b_i(y)$ be constant near the point $g$. Then the smoothness of
generalized solutions preserves near the point $g$ whenever the
right-hand side $\{f_i\}$ satisfies the consistency
condition~\eqref{eqExConsistencyf}. However, if
condition~\eqref{eqExB2vB1CConsistency2} fails (e.g., if
$b_1(y)\equiv \beta_1y_2$, $b_2(y)\equiv \beta_2y_2$,
$\beta_1\ne\beta_2$, near the point $g=0$, axis $Oy_2$ being
tangent to $\partial G$ at $g=0$), then the smoothness of
generalized solutions can be violated even if the right-hand side
$\{f_i\}$ satisfies the consistency
condition~\eqref{eqExConsistencyf}. This follows from
Theorem~\ref{thExSmoothfne0}.

Now we illustrate another phenomenon arising in the border case.
Assume that $b_1(y)\equiv b_2(y)\equiv 0$. Let $a(y)=0$ in some
neighborhood of the point $h$ and $\Omega(g)\in G$. Then the {\em
support of nonlocal terms lies strictly inside the domain $G$.}
However, if $a(g)\ne0$ or $(\partial
a/\partial\tau_g)|_{y=g}\ne0$, where $\tau_g$ denotes the unit
vector tangent to $\partial G$ at the point $g$, then the
smoothness of generalized solutions of
problem~\eqref{eqIntroPinG}, \eqref{eqIntroBinG} (even with
homogeneous nonlocal conditions, $\{f_i\}=0$) can be violated.
This follows from Corollary~\ref{corSmoothf0B^2vInsideG} (see also
Sec.~\ref{subsecEx2}).

Note that the smoothness of generalized solutions for some
particular nonlocal elliptic problems was earlier studied by
Skubachevskii~\cite{SkMs86,SkRJMP}. In these papers, a nonlocal
perturbation of the Dirichlet problem for the Laplace operator is
treated; a condition which is necessary and sufficient for any
generalized solution of a problem with homogeneous nonlocal
conditions to belong to the space $W^2(G)$ has been found.
However, it was fundamental that the ``local'' Dirichlet
conditions are set on a part of the boundary and the coefficients
of nonlocal terms are constant.

In this paper, we suggest an approach for the study of smoothness,
based on the results concerning the solvability of model nonlocal
problems in plane angles in Sobolev spaces~\cite{GurRJMP03} and on
the asymptotic behavior of solutions of such problems in weighted
spaces~\cite{SkMs86,GurPetr03}. Our approach enables one to
investigate the smoothness of generalized solutions when different
nonlocal conditions are set on different parts of the boundary,
coefficients of nonlocal terms supported near the conjugation
points are variable, and nonlocal operators corresponding to
nonlocal terms supported outside the conjugation points are
abstract. Moreover, nonlocal boundary-value conditions can be both
homogeneous and nonhomogeneous.

\section{Setting of Nonlocal Problems in Bounded Domains}\label{sectStatement}

\subsection{Setting of the Problem}\label{subsectStatement}
Let $G\subset{\mathbb R}^2$ be a bounded domain with boundary
$\partial G$. Consider a set ${\mathcal K}\subset\partial G$
consisting of finitely many points. Let $\partial
G\setminus{\mathcal K}=\bigcup\limits_{i=1}^{N}\Gamma_i$, where
$\Gamma_i$ are open (in the topology of $\partial G$)
$C^\infty$-curves. Assume that the domain $G$ is a plane angle in
some neighborhood of each point $g\in{\mathcal K}$.

For an integer $k\ge0$, denote by $W^k(G)=W_2^k(G)$ the Sobolev
space with the norm
$$
\|u\|_{W^k(G)}=\left(\sum\limits_{|\alpha|\le k}\int\limits_G
|D^\alpha u(y)|^2\,dy\right)^{1/2}
$$
(set $W^0(G)=L_2(G)$ for $k=0$). For an integer $k\ge1$, we
introduce the space $W^{k-1/2}(\Gamma)$ of traces on a smooth
curve $\Gamma\subset\overline{ G}$ with the norm
\begin{equation}\label{eqTraceNormW}
\|\psi\|_{W^{k-1/2}(\Gamma)}=\inf\|u\|_{W^k(G)}\quad (u\in
W^k(G):\ u|_\Gamma=\psi).
\end{equation}

Along with Sobolev spaces, we will use weighted spaces (the
Kondrar'ev spaces). Let us introduce these spaces. Let
$Q=\{y\in{\mathbb R}^2:\ r>0,\ |\omega|<\omega_0\}$,
$Q=\{y\in{\mathbb R}^2:\ 0<r<d,\ |\omega|<\omega_0\}$,
$0<\omega_0<\pi$, $d>0$, or $Q=G$. We denote by $\mathcal M$ the
set $\{0\}$ in the first and second cases and the set $\mathcal K$
in the third case. Introduce the space $H_a^k(Q)=H_a^k(Q,\mathcal
M)$ as a completion of the set $C_0^\infty(\overline{ Q}\setminus
\mathcal M)$ with respect to the norm
$$
 \|u\|_{H_a^k(Q)}=\left(
    \sum_{|\alpha|\le k}\int\limits_Q \rho^{2(a-k+|\alpha|)} |D^\alpha u(y)|^2 dy
                                       \right)^{1/2},
$$
where $a\in \mathbb R$, $k\ge 0$ is an integer, and
$\rho=\rho(y)=\dist(y,\mathcal M)$. For an integer $k\ge1$, denote
by $H_a^{k-1/2}(\Gamma)$  the set of traces on a smooth curve
$\Gamma\subset\overline{ Q}$ with the norm
\begin{equation}\label{eqTraceNormH}
\|\psi\|_{H_a^{k-1/2}(\Gamma)}=\inf\|u\|_{H_a^k(Q)} \quad (u\in
H_a^k(Q):\ u|_\Gamma = \psi).
\end{equation}

For an integer $k\ge1$, we also set
$$
\mathcal W^{k-1/2}(\partial G)=\prod\limits_{i=1}^{N}
W^{k-1/2}(\Gamma_i),\qquad \mathcal H_a^{k-1/2}(\partial
G)=\prod\limits_{i=1}^{N} H_a^{k-1/2}(\Gamma_i).
$$

Consider the operator
$$
{\bf
P}u=\sum\limits_{i,k=1}^2p_{ik}(y)u_{y_iy_k}+\sum\limits_{k=1}^2p_k(y)u_{y_k}+p_0(y)u,
$$
where $p_{ik}$, $i,k=1,2$, and $p_k$, $k=0,1,2$, are
complex-valued $C^\infty$-coefficients. We assume throughout the
paper that the operator ${\bf P}$ is {\em properly elliptic} on
$\overline{G}$ (see, e.g.,~\cite[Chap.~2, \S~1]{LM}).

\smallskip

For any closed set $\mathcal M$, we denote its
$\varepsilon$-neighborhood by $\mathcal O_{\varepsilon}(\mathcal
M)$, i.e.,
$$
\mathcal O_{\varepsilon}(\mathcal M)=\{y\in \mathbb R^2:\ \dist(y,
\mathcal M)<\varepsilon\},\qquad {\varepsilon}>0.
$$

Now we introduce operators corresponding to nonlocal terms
supported near the set $\mathcal K$. Let $\Omega_{is}$ ($i=1,
\dots, N;$ $s=1, \dots, S_i$) be $C^\infty$-diffeomorphisms taking
some neighborhood ${\mathcal O}_i$ of the curve
$\overline{\Gamma_i\cap\mathcal O_{{\varepsilon}}(\mathcal K)}$ to
the set $\Omega_{is}({\mathcal O}_i)$ in such a way that
$\Omega_{is}(\Gamma_i\cap\mathcal O_{{\varepsilon}}(\mathcal
K))\subset G$ and
\begin{equation}\label{eqOmega}
\Omega_{is}(g)\in\mathcal K\quad\text{for}\quad
g\in\overline{\Gamma_i}\cap\mathcal K.
\end{equation}
Thus, the transformations $\Omega_{is}$ take the curves
$\Gamma_i\cap\mathcal O_{{\varepsilon}}(\mathcal K)$ strictly
inside the domain $G$ and the set of their end points
$\overline{\Gamma_i}\cap\mathcal K$ to itself.

Let us specify the structure of the transformations $\Omega_{is}$
near the set $\mathcal K$. Denote by $\Omega_{is}^{+1}$ the
transformation $\Omega_{is}:{\mathcal O}_i\to\Omega_{is}({\mathcal
O}_i)$ and by $\Omega_{is}^{-1}:\Omega_{is}({\mathcal
O}_i)\to{\mathcal O}_i$ the inverse transformation. The set of
points
$\Omega_{i_qs_q}^{\pm1}(\dots\Omega_{i_1s_1}^{\pm1}(g))\in{\mathcal
K}$ ($1\le s_j\le S_{i_j},\ j=1, \dots, q$) is said to be an {\em
orbit} of the point $g\in{\mathcal K}$ and denoted by $\Orb(g)$.
In other words, the orbit $\Orb(g)$ is formed by the points (of
the set $\mathcal K$) that can be obtained by consecutively
applying the transformations $\Omega_{i_js_j}^{\pm1}$ to the point
$g$.

It is clear that either $\Orb(g)=\Orb(g')$ or
$\Orb(g)\cap\Orb(g')=\varnothing$ for any $g, g'\in{\mathcal K}$.
In what follows, we assume that the set $\mathcal K$ consists of
one orbit only (the results we will obtain are easy to generalize
for the case in which $\mathcal K$ consists of finitely many
disjoint orbits, see Sec.~\ref{sectGener}). The set (orbit)
$\mathcal K$ consists of $N$ points. We denote these points by
$g_j$, $j=1, \dots, N$.

Take a sufficiently small number $\varepsilon$ (see
Remark~\ref{remSmallEps} below) such that there exist
neighborhoods $\mathcal O_{\varepsilon_1}(g_j)$, $ \mathcal
O_{\varepsilon_1}(g_j)\supset\mathcal O_{\varepsilon}(g_j) $,
satisfying the following conditions:
\begin{enumerate}
\item The domain $G$ is a plane angle in the neighborhood $\mathcal O_{\varepsilon_1}(g_j)$;
\item
$\overline{\mathcal O_{\varepsilon_1}(g_j)}\cap\overline{\mathcal
O_{\varepsilon_1}(g_k)}=\varnothing$ for any $g_j,g_k\in\mathcal
K$, $k\ne j$;
\item If $g_j\in\overline{\Gamma_i}$ and
$\Omega_{is}(g_j)=g_k,$ then ${\mathcal
O}_{\varepsilon}(g_j)\subset\mathcal
 O_i$ and
 $\Omega_{is}\big({\mathcal
O}_{\varepsilon}(g_j)\big)\subset{\mathcal
O}_{\varepsilon_1}(g_k).$
\end{enumerate}

For each point $g_j\in\overline{\Gamma_i}\cap\mathcal K$, we fix a
transformation $Y_j: y\mapsto y'(g_j)$ which is a composition of
the shift by the vector $-\overrightarrow{Og_j}$ and the rotation
through some angle so that
$$
Y_j({\mathcal O}_{\varepsilon_1}(g_j))={\mathcal
O}_{\varepsilon_1}(0),\qquad Y_j(G\cap{\mathcal
O}_{\varepsilon_1}(g_j))=K_j\cap{\mathcal O}_{\varepsilon_1}(0),
$$
$$
Y_j(\Gamma_i\cap{\mathcal
O}_{\varepsilon_1}(g_j))=\gamma_{j\sigma}\cap{\mathcal
O}_{\varepsilon_1}(0)\quad (\sigma=1\ \text{or}\ \sigma=2),
$$
where
$$
K_j=\{y\in{\mathbb R}^2:\ r>0,\ |\omega|<\omega_j\},\qquad
\gamma_{j\sigma}=\{y\in\mathbb R^2:\ r>0,\ \omega=(-1)^\sigma
\omega_j\}.
$$
Here $(\omega,r)$ are the polar coordinates and $0<\omega_j<\pi$.

Consider the following condition (see Fig.~\ref{figTransform}).
\begin{condition}\label{condK1}
Let $g_j\in\overline{\Gamma_i}\cap\mathcal K$ and
$\Omega_{is}(g_j)=g_k\in\mathcal K;$ then the transformation
$$
Y_k\circ\Omega_{is}\circ Y_j^{-1}:{\mathcal
O}_{\varepsilon}(0)\to{\mathcal O}_{\varepsilon_1}(0)
$$
is the composition of rotation and homothety.
\end{condition}
\begin{figure}[ht]
{ \hfill\epsfbox{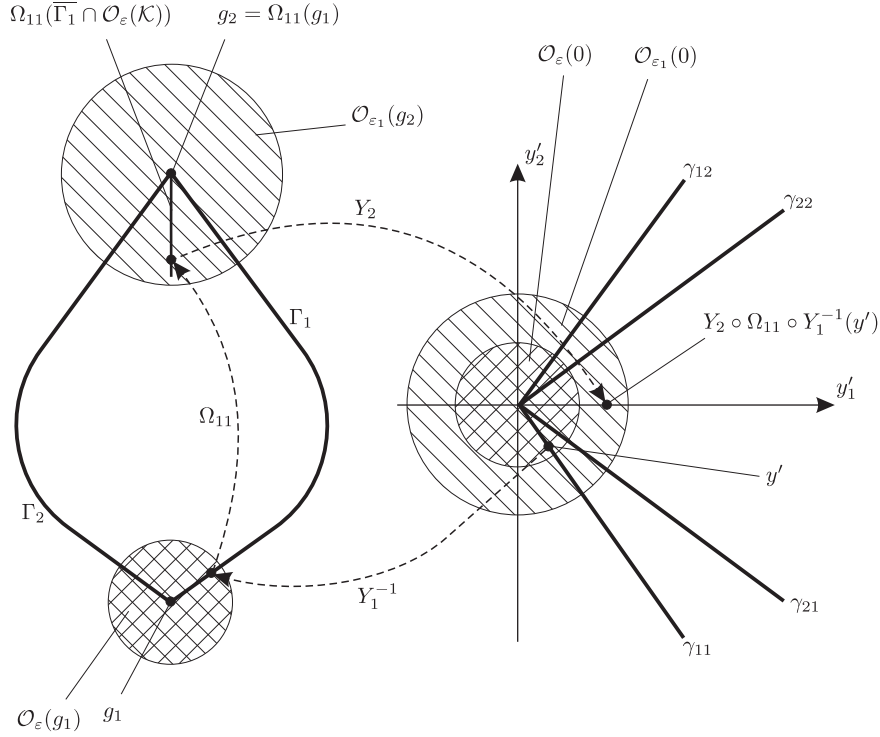}\hfill\ } \caption{The
transformation $ Y_2\circ\Omega_{11}\circ Y_1^{-1}:{\mathcal
O}_{\varepsilon}(0)\to{\mathcal O}_{\varepsilon_1}(0) $ is a
composition of rotation and homothety}
   \label{figTransform}
\end{figure}

\begin{remark}\label{remK1}
Condition~\ref{condK1}, together with the fact that
$\Omega_{is}(\Gamma_i)\subset G$, implies that, if
$g\in\Omega_{is}(\overline{\Gamma_i}\cap\mathcal
K)\cap\overline{\Gamma_j}\cap{\mathcal K}\ne\varnothing$, then the
curves $\Omega_{is}(\overline{\Gamma_i}\cap\mathcal
O_\varepsilon(\mathcal K))$ and $\overline{\Gamma_j}$ intersect at
nonzero angle at the point $g$.
\end{remark}

Introduce the nonlocal operators $\mathbf B_{i}^1$ by the formulas
$$
 \mathbf B_{i}^1u=\sum\limits_{s=1}^{S_i}
b_{is}(y) u(\Omega_{is}(y)),\quad
   y\in\Gamma_i\cap\mathcal O_{\varepsilon}(\mathcal K),\qquad
\mathbf B_{i}^1u=0,\quad y\in\Gamma_i\setminus\mathcal
O_{\varepsilon}(\mathcal K),
$$
where $b_{is}\in C^\infty(\mathbb R^2)$ and $\supp
b_{is}\subset\mathcal O_{{\varepsilon}}(\mathcal K)$. Since
$\mathbf B_{i}^1u=0$ whenever $\supp u\subset\overline{
G}\setminus\overline{\mathcal O_{\varepsilon_1}(\mathcal K)}$, we
say that the operators $\mathbf B_{i}^1$ \textit{correspond to
nonlocal terms supported near the set} $\mathcal K$.

\smallskip

Set $G_\rho=\{y\in G:\ \dist(y,
\partial G)>\rho\}$ for $\rho>0$. Consider operators $\mathbf
B_{i}^2$ satisfying the following condition
(cf.~\cite{SkMs86,SkJMAA,GurRJMP03}).
\begin{condition}\label{condSeparK23}
There exist numbers $\varkappa_1>\varkappa_2>0$ and $\rho>0$ such
that
\begin{equation}\label{eqSeparK23'}
  \|\mathbf B^2_{i}u\|_{W^{3/2}(\Gamma_i)}\le c_1
  \|u\|_{W^{2}(G\setminus\overline{\mathcal O_{\varkappa_1}(\mathcal
  K)})}\qquad\forall u\in W^{2}(G\setminus\overline{\mathcal
O_{\varkappa_1}(\mathcal
  K)}),
\end{equation}
\begin{equation}\label{eqSeparK23''}
  \|\mathbf B^2_{i}u\|_{W^{3/2}
   (\Gamma_i\setminus\overline{\mathcal O_{\varkappa_2}(\mathcal K)})}\le
  c_2 \|u\|_{W^{2}(G_\rho)}\qquad\forall u\in
  W^{2}(G_\rho),
\end{equation}
where $i=1, \dots, N$, whereas $c_1,c_2>0$ do not depend on $u$.
\end{condition}

In particular, inequality~\eqref{eqSeparK23'} implies that
$\mathbf B_{i}^2u=0$ whenever $\supp u\subset \mathcal
O_{\varkappa_1}(\mathcal K)$. For this reason, we say that the
operators $\mathbf B_{i}^2$ \textit{correspond to nonlocal terms
supported outside the set} $\mathcal K$.

\smallskip

We assume that Conditions~\ref{condK1} and~\ref{condSeparK23} are
fulfilled throughout
Secs.~\ref{sectStatement}--\ref{sectImproperEigen}.

\smallskip

We study the following nonlocal elliptic boundary-value problem:
\begin{align}
 {\bf P} u=f_0(y) \quad &(y\in G),\label{eqPinG}\\
     u|_{\Gamma_i}+\mathbf B_{i}^1 u+\mathbf B_{i}^2 u=
   f_{i}(y)\quad
    &(y\in \Gamma_i;\ i=1, \dots, N).\label{eqBinG}
\end{align}
Note that the points $g_j$ divide the curves on which different
nonlocal conditions are set; therefore, it is natural to say that
$g_j$, $j=1,\dots,N$, are the {\em points of conjugation of
nonlocal conditions}. Problem~\eqref{eqIntroPinG},
\eqref{eqIntroBinG} is an example of an elliptic problem with
nonlocal conditions~\eqref{eqBinG} (see also
Sec.~\ref{sectExample}).

 For any set $X\in \mathbb R^2$ having a nonempty interior, we
denote by $C_0^\infty(X)$ the set of functions infinitely
differentiable on $\overline{ X}$ and compactly supported on $X$.

\begin{definition}\label{defGenSol2}
A function $u\in W^{1}(G)$ is called a {\em generalized solution}
of problem~\eqref{eqPinG}, \eqref{eqBinG} with right-hand side
$\{f_0,f_i\}\in L_2(G)\times \mathcal W^{1/2}(\partial G)$ if $u$
satisfies nonlocal conditions~\eqref{eqBinG} in the sense of
traces and Eq.~\eqref{eqPinG} in the sense of distributions. The
latter is equivalent to the validity of the integral identity
$$
-\int\limits_G\sum\limits_{i,k=1}^2
u_{y_i}\overline{(p_{ik}w)_{y_k}}\,dy+\int\limits_G\Bigg(\sum\limits_{k=1}^2
p_ku_{y_k}+p_0u\Bigg)\overline{w}\,dy=\int\limits_G
    f_0\overline{w}\,dy
$$
for all $w\in C_0^\infty(G)$.
\end{definition}

\begin{remark}
Generalized solutions a priori belong to $W^1(G)$, whereas
Condition~\ref{condSeparK23} is formulated for functions from the
space $W^{2}$ inside the domain and near a smooth part of the
boundary. This formulation can be justified by the fact that any
generalized solution belongs to the space $W^{2}$ outside an
arbitrarily small neighborhood of the set $\mathcal K$ (see
Sec.~\ref{sectNoEigen}).
\end{remark}

\begin{remark}\label{remSmallEps}
We can assume that the number $\varepsilon$ occurring in the
definition of the operators $\mathbf B_i^1$ is sufficiently small
(while $\varkappa_1,\varkappa_2,\rho$ occurring in the definition
of the operators $\mathbf B_i^2$ can be arbitrary). Let us show
that this assumption leads to no loss of generality.

Take a number $\hat\varepsilon$, where
$0<\hat\varepsilon<\varepsilon$. Set $\hat{\mathbf
B}_{i}^1u=\sum\limits_{s=1}^{S_i}
   \big(\hat b_{i s}(y) u(\Omega_{is}(y))$ for
$y\in\Gamma_i\cap\mathcal O_{\hat\varepsilon}(\mathcal K)$ and $
\hat{\mathbf B}_{i}^1u=0$ for $y\in\Gamma_i\setminus\mathcal
O_{\hat\varepsilon}(\mathcal K)$, where ${\hat b}_{is}\in
C^\infty(\mathbb R^2)$,  $\supp b_{is}\subset\mathcal
O_{{\hat\varepsilon}}(\mathcal K)$, and ${\hat
b}_{is}(y)=b_{is}(y)$ for $y\in\Gamma_i\cap\mathcal
O_{\hat\varepsilon/2}(\mathcal K)$. It is clear that
$$
\mathbf B_{i}^1+\mathbf B_{i}^2=\hat{\mathbf B}_{i}^1+\hat{\mathbf
B}_{i}^2,
$$
where $\hat{\mathbf B}_{i}^2=\mathbf B_{i}^1-\hat{\mathbf
B}_{i}^1+\mathbf B_{i}^2$. Since $\mathbf B_{i}^1u-\hat{\mathbf
B}_{i}^1u=0$ near the set $\mathcal K$, it follows that the
operator $\mathbf B_{i}^1-\hat{\mathbf B}_{i}^1$ satisfy
Condition~\ref{condSeparK23} for appropriate
$\varkappa_1,\varkappa_2,\rho$ (see~\cite[\S~1]{GurRJMP03} for
details). Thus, we see that $\varepsilon$ can be taken as small as
needed. However, one must remember that the operator $\mathbf
B_{i}^2$ and the values of $\varkappa_1,\varkappa_2,\rho$ may
change if we change the value of $\varepsilon$.
\end{remark}

\subsection{Model Problems}\label{subsectStatementNearK}

When studying problem~\eqref{eqPinG}, \eqref{eqBinG}, particular
attention must be paid to the behavior of solutions near the set
${\mathcal K}$ of conjugation points. In this subsection, we
consider corresponding model problems.

Denote by $u_j(y)$ the function $u(y)$ for $y\in{\mathcal
O}_{\varepsilon_1}(g_j)$. If $g_j\in\overline{\Gamma_i},$
$y\in{\mathcal O}_{\varepsilon}(g_j),$ $\Omega_{is}(y)\in{\mathcal
O}_{\varepsilon_1}(g_k),$ then denote by $u_k(\Omega_{is}(y))$ the
function $u(\Omega_{is}(y))$. In this case, nonlocal
problem~\eqref{eqPinG}, \eqref{eqBinG} acquires the following form
in the neighborhood of the set (orbit) $\mathcal K$:
\begin{gather*}
 {\bf P} u_j=f_0(y) \quad (y\in\mathcal O_\varepsilon(g_j)\cap
 G),\\
\begin{aligned}
u_j(y)|_{\mathcal O_\varepsilon(g_j)\cap\Gamma_i}+
\sum\limits_{s=1}^{S_i} b_{is}(y) u_k(\Omega_{is}(y))|_{\mathcal
O_\varepsilon(g_j)\cap\Gamma_i}
=\psi_i(y) \\
\big(y\in \mathcal O_\varepsilon(g_j)\cap\Gamma_i;\ i\in\{1\le
i\le N:\ g_j\in\overline{\Gamma_i}\};\ j=1, \dots, N\big),
\end{aligned}
\end{gather*}
where
$$
\psi_i=f_i-\mathbf B_{i}^2u.
$$

Let $y\mapsto y'(g_j)$ be the change of variables described in
Sec.~\ref{subsectStatement}. Set
$$
K_j^\varepsilon=K_j\cap\mathcal O_\varepsilon(0),\qquad
\gamma_{j\sigma}^\varepsilon=\gamma_{j\sigma}\cap\mathcal
O_\varepsilon(0)
$$
and introduce the functions
\begin{equation}\label{eqytoy'}
\begin{gathered}
U_{j}(y')=u(y(y')),\quad F_{j}(y')=f_0(y(y')),\quad y'\in
K_{j}^\varepsilon,\\
F_{j\sigma}(y')=f_i(y(y')),\quad B_{j\sigma}^u(y')=(\mathbf
B_i^2u)(y(y')),\\
\Psi_{j\sigma}(y')=F_{j\sigma}(y')-B_{j\sigma}^u(y'),\quad
y'\in\gamma_{j\sigma}^\varepsilon,
\end{gathered}
\end{equation}
where $\sigma=1$ $(\sigma=2)$ if the transformation $y\mapsto
y'(g_j)$ takes $\Gamma_i$ to the side $\gamma_{j1}$
($\gamma_{j2}$) of the angle $K_j$. Denote $y'$ by $y$ again.
Then, by virtue of Condition~\ref{condK1}, problem~\eqref{eqPinG},
\eqref{eqBinG} acquires the form
\begin{gather}
  {\bf P}_{j}U_j=F_{j}(y) \quad (y\in
  K_j^\varepsilon),\label{eqPinK}\\
  {\mathbf B}_{j\sigma}U\equiv\sum\limits_{k,s}
      b_{j\sigma ks}(y)U_k({\mathcal G}_{j\sigma ks}y)
    =\Psi_{j\sigma}(y) \quad (y\in\gamma_{j\sigma}^\varepsilon).\label{eqBinK}
\end{gather}
Here (and below unless otherwise stated) $j, k=1, \dots, N;$
$\sigma=1, 2;$ $s=0, \dots, S_{j\sigma k}$; ${\mathbf P}_j$ are
properly elliptic second-order differential operators with
variable complex-valued $C^\infty$-coefficients,
$$
\mathbf P_jv=\sum\limits_{i,k=1}^2p_{jik}(y)v_{y_iy_k}+
\sum\limits_{k=1}^2p_{jk}(y)v_{y_k}+p_{j0}(y)v;
$$
$U=(U_1,\dots,U_N)$; $b_{j\sigma ks}(y)$ are smooth functions,
$b_{j\sigma j0}(y)\equiv 1$; ${\mathcal G}_{j\sigma ks}$ is an
operator of rotation through an angle~$\omega_{j\sigma ks}$ and
homothety with a coefficient~$\chi_{j\sigma ks}>0$ in the
$y$-plane. Moreover,
$$
|(-1)^\sigma \omega_{j}+\omega_{j\sigma
ks}|<\omega_{k}\qquad\text{for}\qquad (k,s)\ne(j,0)
$$
(see Remark~\ref{remK1}) and
$$
\omega_{j\sigma j0}=0,\qquad \chi_{j\sigma j0}=1
$$
(i.e., $\mathcal G_{j\sigma j0}y\equiv y$).

Let the principal homogeneous parts of the operators $\mathbf P_j$
at the point $y=0$ have the following form in the polar
coordinates:
$$
\sum\limits_{i,k=1}^2p_{jik}(0)v_{y_iy_k}=r^{-2}\tilde{\mathcal
P_j}(\omega,\partial/\partial\omega,r\partial/\partial r)v.
$$
Consider the analytic operator-valued function $\tilde{\mathcal
L}(\lambda):\prod\limits_j
W^{2}(-\omega_j,\omega_j)\to\prod\limits_j (L_2(-\omega_j,
\omega_j)\times\mathbb C^2)$ given by
$$
\tilde{\mathcal L}(\lambda)\varphi=\big\{\tilde{\mathcal
P}_j(\omega, \partial/\partial\omega, i\lambda)\varphi_j,\
  \sum\limits_{k,s} (\chi_{j\sigma ks})^{i\lambda}
              b_{j\sigma ks}(0)\varphi_k((-1)^\sigma \omega_j+\omega_{j\sigma ks})\big\}.
$$
Main definitions and facts concerning analytic operator-valued
functions can be found in~\cite{GS}. The following assertion is of
particular importance (see~\cite[Lemmas~2.1 and 2.2]{SkDu90}).
\begin{lemma}\label{lSpectrum}
The spectrum of the operator $\tilde{\mathcal L}(\lambda)$ is
discrete. For any numbers $c_1<c_2$, the band $c_1<\Im\lambda<c_2$
contains at most finitely many eigenvalues of the operator
$\tilde{\mathcal L}(\lambda)$.
\end{lemma}

Spectral properties of the operator $\tilde{\mathcal L}(\lambda)$
play a crucial role in the study of smoothness of generalized
solutions.

\section{Preservation of Smoothness of Generalized Solutions}\label{sectNoEigen}

First, we study the case in which the following condition holds.
\begin{condition}\label{condNoEigen}
The band $-1\le\Im\lambda<0$ contains no eigenvalues of the
operator $\tilde{\mathcal L}(\lambda)$.
\end{condition}

The main result of this section is as follows.

\begin{theorem}\label{thuinW^2NoEigen}
Let Condition~$\ref{condNoEigen}$ hold, and let $u\in W^1(G)$ be a
generalized solution of problem~\eqref{eqPinG}, \eqref{eqBinG}
with right-hand side $\{f_0,f_i\}\in L_2(G)\times \mathcal
W^{3/2}(\partial G)$. Then $u\in W^2(G)$.
\end{theorem}

\begin{remark} By Theorem~\ref{thuinW^2NoEigen}, any generalized
solution of problem~\eqref{eqPinG}, \eqref{eqBinG} belongs to
$W^2(G)$ whenever Condition~\ref{condNoEigen} holds. The
right-hand sides $f_i$ in nonlocal conditions are naturally
supposed to belong to the space $W^{3/2}(\Gamma_i)$. However, no
additional assumptions (e.g., consistency conditions) are imposed
on the behavior of the functions $f_i$ and on the behavior of the
coefficients of nonlocal terms near the set $\mathcal K$. In fact,
the functions $f_i\in W^{3/2}(\Gamma_i)$ are not quite arbitrary.
For instance, if $\mathbf B_i^1=0$, $\mathbf B_i^2=0$ (i.e., we
have a ``local'' problem), and a solution $u$ belongs to $W^2(G)$,
then, by Sobolev's embedding theorem,
\begin{equation}\label{eqfi=fj}
f_i(g)=f_j(g)\qquad\text{for}\qquad
g\in\overline{\Gamma_i}\cap\overline{\Gamma_j}\ne\varnothing.
\end{equation}
Theorem~\ref{thuinW^2NoEigen} implies that, if
Condition~\ref{condNoEigen} holds, then the existence of a {\em
generalized} solution itself ensures the validity of relations of
the kind~\eqref{eqfi=fj}. In Sec.~\ref{sectProperEigen}, we will
see that, if Condition~\ref{condNoEigen} fails, then we must
impose some consistency condition on the right-hand sides $f_i$ in
order that any generalized solution be smooth.
\end{remark}

Since $\{f_0,f_i\}\in L_2(G)\times \mathcal W^{3/2}(\partial G)$
and the operators $\mathbf B^2_{i}$ satisfy
Condition~\ref{condSeparK23}, it follows
from~\cite[Lemma~2.1]{GurMatZam05} that\footnote{See
also~\cite{SkDu82}.}
\begin{equation}\label{eqSmoothOutsideK}
u\in W^{2}\bigl(G\setminus\overline{\mathcal O_\delta(\mathcal
K)}\bigr)\qquad\forall \delta>0.
\end{equation}

Let $U_j(y')=u_j(y(y'))$, $j=1,\dots,N$, be the functions
corresponding to the set (orbit) $\mathcal K$ and satisfying
problem~\eqref{eqPinK}, \eqref{eqBinK} with right-hand side
$\{F_j, \Psi_{j\sigma}\}$ (see Sec.~\ref{subsectStatementNearK}).

Set
\begin{equation}\label{eqd1d2}
D_\chi=2\max\{\chi_{j\sigma ks}\},\qquad
d_\chi=\min\{\chi_{j\sigma ks}\}/2.
\end{equation}
Let $\varepsilon>0$ be so small that $D_\chi
\varepsilon<\varepsilon_1$.

Introduce the spaces of vector-valued functions
\begin{equation}\label{eqSpacesCal1}
\mathcal W^k(K^\varepsilon)=\prod\limits_j
W^k(K_j^\varepsilon),\quad \mathcal \mathcal \mathcal
H_a^k(K^\varepsilon)=\prod\limits_j H_a^k(K_j^\varepsilon), \quad
k\ge0;
\end{equation}
\begin{equation}\label{eqSpacesCal2}
\mathcal W^{k-1/2}(\gamma^\varepsilon)=
\prod\limits_{j,\sigma}W^{k-1/2}(\gamma_{j\sigma}^\varepsilon),\quad
\mathcal H_a^{k-1/2}(\gamma^\varepsilon)=
\prod\limits_{j,\sigma}H_a^{k-1/2}(\gamma_{j\sigma}^\varepsilon),\quad
k\ge1.
\end{equation}
Similarly, one can introduce the spaces $\mathcal W^k(K)$,
$\mathcal H_a^k(K)$, $\mathcal W^{k-1/2}(\gamma)$, and $\mathcal
H_a^{k-1/2}(\gamma)$.

By virtue of relation~\eqref{eqSmoothOutsideK},
\begin{equation}\label{eqU_jW2loc}
U_j\in W^{2}(K_{j}^{\varepsilon_1}\cap\{|y|>\delta\})\quad\forall
\delta>0.
\end{equation}
Furthermore, it follows from the belonging $U\in \mathcal
W^1(K^{\varepsilon_1})$ and Lemma~\ref{lAppL4.9Kondr_5.2KovSk}
that
\begin{equation}\label{eqUa-10}
U\in \mathcal H_{a}^1(K^{\varepsilon_1})\subset \mathcal
H_{a-1}^0(K^{\varepsilon_1}),\quad a>0.
\end{equation}
Finally, we have (see~\eqref{eqPinK}, \eqref{eqBinK})
$\{F_j\}\in\mathcal W^0(K^{\varepsilon})$ and, by the belonging
$f_i\in W^{3/2}(\Gamma_i)$, by relation~\eqref{eqSmoothOutsideK},
and by estimate~\eqref{eqSeparK23'}, we have
$\{\Psi_{j\sigma}\}\in \mathcal W^{3/2}(\gamma^\varepsilon)$.
Therefore, using Lemma~\ref{lAppL4.9Kondr_5.2KovSk}, we obtain
\begin{equation}\label{eqf1+a3/2}
\{F_j\}\in \mathcal H_{1+a}^{0}(K^{\varepsilon}),\quad
\{\Psi_{j\sigma}\}\in \mathcal
H_{1+a}^{3/2}(\gamma^\varepsilon),\quad a>0.
\end{equation}

It follows from relations~\eqref{eqU_jW2loc}--\eqref{eqf1+a3/2}
and from Lemma~\ref{lAppL2.3GurMatZam05} that
\begin{equation}\label{eqU_jH1+a}
U\in\mathcal  H_{1+a}^{2}(K^{\varepsilon}),\quad a>0.
\end{equation}
To prove Theorem~\ref{thuinW^2NoEigen}, it suffices to show that
$U\in \mathcal W^{2}(K^{\varepsilon})$.

Fix a sufficiently small number $a$, $0<a<1$, such that the band
$a-1\le\Im\lambda\le a$ contains no nonreal eigenvalues of the
operator $\tilde{\mathcal L}(\lambda)$. The existence of such an
$a$ follows from Lemma~\ref{lSpectrum} and
Condition~\ref{condNoEigen}.

Denote
$$
\mathcal P_jv=\sum\limits_{i,k=1}^2p_{jik}(0)v_{y_iy_k},\quad
{\mathcal B}_{j\sigma}U=\sum\limits_{k,s}
      b_{j\sigma ks}(0)U_k({\mathcal G}_{j\sigma ks}y).
$$

\begin{lemma}\label{lU=C+}
Let $U\in\mathcal W^1(K^\varepsilon)$ be a generalized
solution\footnote{That is $U$ satisfies Eq.~\eqref{eqPinK} in the
sense of distributions and nonlocal conditions~\eqref{eqBinK} in
the sense of traces.} of problem~\eqref{eqPinK}, \eqref{eqBinK}
with right-hand side $\{F_j, \Psi_{j\sigma}\}\in \mathcal
W^0(K^{\varepsilon})\times\mathcal W^{3/2}(\gamma^\varepsilon)$.
Then
\begin{equation}\label{eqU=C+}
 U=C+U',
\end{equation}
where $U'\in\mathcal  H_{a}^{2}(K^\varepsilon)$, $a$ is the above
number, and $C=(C_1, \dots, C_N)$ is a constant vector. The
function $U'$ and the vector $C$ are uniquely defined, and the
vector $C$ satisfies the relation
\begin{equation}\label{eqBC}
\mathcal B_{j\sigma}C=\Psi_{j\sigma}(0).
\end{equation}
\end{lemma}
\begin{proof}
1. Write problem~\eqref{eqPinK}, \eqref{eqBinK} as follows:
\begin{equation}\label{eqPBinK0}
{\mathbf P}_{j}U_j=F_j(y) \quad (y\in K_j^\varepsilon),\qquad
{\mathbf B}_{j\sigma}U =\Psi_{j\sigma}(0)+\Psi_{j\sigma}^0(y)
\quad (y\in\gamma_{j\sigma}^\varepsilon),
\end{equation}
where $\Psi_{j\sigma}^0(y)=\Psi_{j\sigma}(y)-\Psi_{j\sigma}(0)$.
We claim that
\begin{equation}\label{eqFHa}
\{F_j\}\in\mathcal  H_{a}^0(K^\varepsilon),\quad
\{\Psi_{j\sigma}^0\}\in\mathcal H_{a}^{3/2}(\gamma^\varepsilon).
\end{equation}
Indeed, the first belonging follows from the relaton
$\{F_j\}\in\mathcal W^0(K^\varepsilon)$, whereas the second one
from the relations $\{\Psi_{j\sigma}^0\}\in\mathcal
W^{3/2}(\gamma^\varepsilon)$ and $\Psi_{j\sigma}^0(0)=0$ and from
Lemma~\ref{lAppL2.1GurRJMP03}.

2. By Lemma~\ref{lAppL4.3GurPetr03}, there exists a function
\begin{equation}\label{eqWH1+a}
W=\sum\limits_{l=0}^{\varkappa}\frac{\displaystyle
1}{\displaystyle l!}(i\ln r)^l w^{(l)}(\omega)\in\mathcal
H_{1+a}^2(K^\varepsilon)
\end{equation}
such that
\begin{equation}\label{eqBW}
\mathcal P_jW_j=0\quad (y\in K_j),\qquad {\mathcal
B}_{j\sigma}W=\Psi_{j\sigma}(0)\quad (y\in\gamma_{j\sigma}),
\end{equation}
Here $\varkappa=0$ if $\lambda=0$ is not an eigenvalue of
$\tilde{\mathcal L}(\lambda)$; otherwise, $\varkappa$ equals the
multiplicity of the eigenvalue $\lambda=0$; $w^{(l)}\in
\prod\limits_j W^2(-\omega_j,\omega_j)$.

As we have proved before this lemma, the function $U$
satisfies~\eqref{eqU_jH1+a}. Combining this fact with
relation~\eqref{eqWH1+a} yields
\begin{equation}\label{eqUWVH1+a}
U-W\in\mathcal  H_{1+a}^2(K^\varepsilon).
\end{equation}
On the other hand, Lemma~\ref{lAppL3.3'Kondr} implies that
\begin{equation}\label{eqUWVH1+a'}
\{{\mathbf P}_{j}U_j-{\mathcal P}_{j}U_j\}\in \mathcal
H_{a}^0(K^\varepsilon),\quad \{{\mathbf
B}_{j\sigma}U|_{\gamma_{j\sigma}^\varepsilon}-{\mathcal
B}_{j\sigma}U|_{\gamma_{j\sigma}^\varepsilon}\}\in \mathcal
H_{a}^{3/2}(\gamma^\varepsilon).
\end{equation}

It follows from~\eqref{eqPBinK0}, \eqref{eqFHa}, and
\eqref{eqUWVH1+a'} that
\begin{equation}\label{eqPBUWV}
\{{\mathcal P}_{j}(U_j-W_j)\}\in\mathcal
H_{a}^0(K^\varepsilon),\quad \{{\mathcal
B}_{j\sigma}(U-W)|_{\gamma_{j\sigma}^\varepsilon}\}\in\mathcal
H_{a}^{3/2}(\gamma^\varepsilon).
\end{equation}

3. Applying Theorem~\ref{lAppTh2.2GurPetr03} concerning the
asymptotic behavior of the function $U-W$ and using
relations~\eqref{eqUWVH1+a} and \eqref{eqPBUWV}, we obtain
\begin{equation}
U-W=\sum\limits_{\Im\lambda_n=0}
  \sum\limits_{q=1}^{J_n}\sum\limits_{m=0}^{\varkappa_{qn}-1}
  c_n^{(m,q)}W_n^{(m,q)}+U'.
\end{equation}
Here $\{\lambda_n\}$ is a finite set of eigenvalues of the
operator $\tilde{\mathcal L}(\lambda)$ lying on the line
$\Im\lambda=0$;
$$
W_n^{(m,q)}(\omega,r)=r^{i\lambda_n}\sum\limits_{l=0}^m\frac{1}{l!}(i\ln
r)^l\varphi_n^{(m-l,q)}(\omega),
$$
\begin{equation}\label{eqPBW=0}
{\mathcal B}_{j\sigma}W_n^{(m,q)}|_{\gamma_{j\sigma}}=0;
\end{equation}
$\varphi_n^{(0,q)}, \dots,
\varphi_n^{(\varkappa_{qn}-1,q)}\in\prod\limits_j
W^2(-\omega_j,\omega_j)$ are an eigenvector and associated vectors
(a  Jordan chain of length $\varkappa_{qn}\ge 1$) corresponding to
the eigenvalue $\lambda_n$; $c_n^{(m,q)}$ are constants; finally,
$U'\in\mathcal H_{a}^{2}(K^\varepsilon)$.

Set
$$
C=W+\sum\limits_{n,q,m} c_n^{(m,q)}W_n^{(m,q)}.
$$
It is clear that
$$
U=C+U'.
$$
Since $U,U'\in\mathcal W^1(K^\varepsilon)$, it follows that
$C\in\mathcal W^1(K^\varepsilon)$. This relation and
Lemma~\ref{lAppL4.20Kondr} imply that $C$ is a constant vector. By
virtue of~\eqref{eqPBW=0} and~\eqref{eqBW},
$$
{\mathcal B}_{j\sigma}C|_{\gamma_{j\sigma}^\varepsilon}={\mathcal
B}_{j\sigma}W|_{\gamma_{j\sigma}^\varepsilon}=\Psi_{j\sigma}(0).
$$
Therefore, using the relation ${\mathcal B}_{j\sigma}C=\const$ for
$C=\const$, we obtain~\eqref{eqBC}.

4. Now suppose that the equality $U=D+V'$ holds together
with~\eqref{eqU=C+}, where $V'\in\mathcal
H_{a}^{2}(K^\varepsilon)$ and $D=(D_1, \dots, D_N)$ is a constant
vector. Then we have $C-D=V'-U'\in\mathcal
H_{a}^{2}(K^\varepsilon)$, hence $C-D=0$ and $V'-U'=0$.
\end{proof}

\begin{lemma}\label{lUinW^2}
Let the conditions of Lemma~$\ref{lU=C+}$ be fulfilled, and let
Condition~$\ref{condNoEigen}$ hold. Then $U\in\mathcal
W^2(K^\varepsilon)$.
\end{lemma}
\begin{proof}
1. By Lemma~\ref{lU=C+}, it suffices to show that $U'\in\mathcal
W^2(K^\varepsilon)$. The function $U'$ belongs to $\mathcal
H_{a}^{2}(K^\varepsilon)$, and, by virtue of
relations~\eqref{eqU=C+} and~\eqref{eqPBinK0}, it is a solution of
the problem
\begin{equation}\label{eqPBU'}
{\mathbf P}_{j}U'_j=F_j-{\mathbf P}_{j}C_j \ (y\in
K_j^\varepsilon),\quad {\mathbf B}_{j\sigma}U'
=\Psi_{j\sigma}(0)+\Psi_{j\sigma}^0(y)-{\mathbf B}_{j\sigma}C \
(y\in\gamma_{j\sigma}^\varepsilon).
\end{equation}
Since $\{F_j\}\in\mathcal W^0(K^\varepsilon)$ and $C=\const$, it
follows that
\begin{equation}\label{eqF-PC}
\{F_j-{\mathbf P}_{j}C_j\}\in\mathcal H_0^0(K^\varepsilon).
\end{equation}
Further,
\begin{equation}\label{eqF-BC}
\begin{gathered}
\{\Psi_{j\sigma}(0)+\Psi_{j\sigma}^0(y)|_{\gamma_j^\varepsilon}-{\mathbf
B}_{j\sigma}C|_{\gamma_j^\varepsilon}\}\in\mathcal
W^{3/2}(\gamma^\varepsilon),\\
\big(\Psi_{j\sigma}(0)+\Psi_{j\sigma}^0(y)-{\mathbf
B}_{j\sigma}C\big)\big|_{y=0}=0.
\end{gathered}
\end{equation}
The latter relation follows from the fact that
$\Psi_{j\sigma}^0(0)=0$ and ${\mathbf
B}_{j\sigma}C|_{y=0}={\mathcal B}_{j\sigma}C=\Psi_{j\sigma}(0)$
(see Lemma~\ref{lU=C+}).

2. Since the line $\Im\lambda=-1$ has no eigenvalues of
$\tilde{\mathcal L}(\lambda)$ and relations~\eqref{eqF-BC} hold,
it follows from Lemma~\ref{lAppL2.4GurRJMP03} that there exists a
function
\begin{equation}\label{eqVHa}
V\in\mathcal  W^2(K)\cap\mathcal H_{a}^2(K)
\end{equation}
such that
\begin{equation}\label{eqPBV}
\{{\mathbf P}_{j}V_j\}\in\mathcal H_0^0(K^\varepsilon),\quad
\{{\mathbf
B}_{j\sigma}V|_{\gamma_{j\sigma}^\varepsilon}-\big(\Psi_{j\sigma}(0)+\Psi_{j\sigma}^0(y)-{\mathbf
B}_{j\sigma}C\big)|_{\gamma_{j\sigma}^\varepsilon}\}\in \mathcal
H_0^{3/2}(\gamma^\varepsilon).
\end{equation}

Therefore, $U'-V\in\mathcal H_{a}^2(K^\varepsilon)$ and, due
to~\eqref{eqPBU'}--\eqref{eqF-BC} and \eqref{eqPBV}, we have
$$
\{{\mathbf P}_{j}(U'_j-V_j)\}\in\mathcal
H_0^0(K^\varepsilon),\qquad \{{\mathbf
B}_{j\sigma}(U'-V)|_{\gamma_{j\sigma}^\varepsilon}\}\in \mathcal
H_0^{3/2}(\gamma^\varepsilon).
$$
Further, Lemma~\ref{lAppL3.3'Kondr} implies that
$$
\{{\mathcal P}_{j}(U'_j-V_j)\}\in\mathcal
H_0^0(K^\varepsilon),\qquad \{{\mathcal
B}_{j\sigma}(U'-V)|_{\gamma_{j\sigma}^\varepsilon}\}\in \mathcal
H_0^{3/2}(\gamma^\varepsilon).
$$

Since Condition~\ref{condNoEigen} holds, we can apply
Theorem~\ref{lAppTh2.2GurPetr03} concerning the asymptotic
behavior of the function $U'-V$, which yields
$$
U'-V\in\mathcal H_{0}^2(K^\varepsilon)\subset \mathcal
W^2(K^\varepsilon).
$$
Now the conclusion of the lemma follows from the latter relation,
from~\eqref{eqVHa}, and from~\eqref{eqU=C+}.
\end{proof}
Theorem~\ref{thuinW^2NoEigen} results
from~\eqref{eqSmoothOutsideK} and from Lemma~\ref{lUinW^2}.

\section{Border Case: Consistency Conditions}\label{sectProperEigen}

\subsection{Behavior of Solutions near the Conjugation Points}\label{subsectuFixed}

Let $\lambda=\lambda_0$ be an eigenvalue of the operator
$\tilde{\mathcal L}(\lambda)$.

\begin{definition}\label{defRegEigVal}
We say that $\lambda_0$ is a {\em proper eigenvalue} if none of
the corresponding eigenvectors
$\varphi(\omega)=(\varphi_{1}(\omega),\dots, \varphi_{N}(\omega))$
has an associated vector, while the functions
$r^{i\lambda_0}\varphi_{j}(\omega)$, $j=1, \dots, N$, are
polynomials in $y_1, y_2$. An eigenvalue which is not proper is
said to be {\em improper}.
\end{definition}

The notion of proper eigenvalue was originally proposed by
Kondrat'ev~\cite{KondrTMMO67} for ``local'' boundary-value
problems in angular or conical domains.

Clearly, if $\lambda_0$ is a proper eigenvalue, then
$\Im\lambda_0\le0$ and $\Re\lambda_0=0$. Therefore, the line
$\Im\lambda=\const$ can have at most one proper eigenvalue.

In this section, we suppose that the following condition holds.

\begin{condition}\label{condProperEigen}
The band $-1\le\Im\lambda<0$ contains only the eigenvalue
$\lambda=-i$ of the operator $\tilde{\mathcal L}(\lambda)$. This
eigenvalue is a proper one.
\end{condition}

The principal difference between the results of this section and
those of Sec.~\ref{sectNoEigen} is related to the behavior of
generalized solutions near the set $\mathcal K$. If
Condition~\ref{condProperEigen} holds, then Lemma~\ref{lU=C+}
remains valid. However, the conclusion of Lemma~\ref{lUinW^2} is
no longer true because Lemma~\ref{lAppL2.4GurRJMP03} (proved
in~\cite{GurRJMP03}) is inapplicable when the line $\Im\lambda=-1$
contains a proper eigenvalue of $\tilde{\mathcal L}(\lambda)$. In
this section, we make use of other results from~\cite{GurRJMP03}.
To do this, we impose certain consistency conditions on the
behavior of the functions $f_i$ and on the behavior of the
coefficients of nonlocal terms near the set (orbit) $\mathcal K$.

\smallskip

Let $\tau_{j\sigma}$ be the unit vector co-directed with the
ray~$\gamma_{j\sigma}$. Consider the operators
$$
 \frac{\partial} {\partial
  \tau_{j\sigma}}{\mathcal B}_{j\sigma}U\equiv
  \frac{\partial} {\partial
  \tau_{j\sigma}} \Big(\sum\limits_{k,s}b_{j\sigma
  ks}(0)U_k({\mathcal G}_{j\sigma ks}y)\Big).
$$
Using the chain rule, we obtain
\begin{equation}\label{eqDiffB}
 \frac{\partial}
 {\partial \tau_{j\sigma}}
 {\mathcal B}_{j\sigma}U\equiv
 \sum\limits_{k,s}(\hat B_{j\sigma ks}(D_y)U_k)({\mathcal G}_{j\sigma ks}y),
\end{equation}
where $\hat B_{j\sigma ks}(D_y)$ are first-order differential
operators with constant coefficients. In particular, $\hat
B_{j\sigma j0}(D_y)= {\partial}/
 {\partial \tau_{j\sigma}}$ because ${\mathcal G}_{j\sigma j0}y\equiv y$. Formally
replacing the nonlocal operators by the corresponding local
operators in~\eqref{eqDiffB}, we introduce the operators
\begin{equation}\label{eqSystemB}
 \hat{\mathcal B}_{j\sigma}(D_y)U\equiv
 \sum\limits_{k,s}\hat B_{j\sigma ks}(D_y)U_k(y).
\end{equation}

Let us prove that the system of operators~\eqref{eqSystemB} is
linearly dependent if Condition~\ref{condProperEigen} holds. Let
\begin{equation}\label{eqhatB12}
\hat B_{j\sigma ks}(D_y)=b_{j\sigma ks1}\frac{\partial}{\partial
y_1}+b_{j\sigma ks2}\frac{\partial}{\partial y_2},
\end{equation}
where $b_{j\sigma ks1}$ and $b_{j\sigma ks2}$ are complex
constants. It suffices to show that the following system of $2N$
equations for the $2N$ indeterminates $q_{k1},q_{k2}$,
$k=1,\dots,N$, admits a nontrivial solution:
\begin{equation}\label{eqSystemqj12}
\sum\limits_{k,s}b_{j\sigma ks1}q_{k1}+b_{j\sigma
ks2}q_{k2}=0,\quad j=1,\dots,N,\ \sigma=1,2.
\end{equation}

Let $\varphi(\omega)=(\varphi_{1}(\omega),\dots,
\varphi_{N}(\omega))$ be an eigenvector corresponding to the
eigenvalue $\lambda=-i$. By Condition~\ref{condProperEigen}, the
functions $Q_k(y)=r\varphi_k(\omega)$ are homogeneous polynomials
of order one. Set $q_{k1}=\partial Q_k/\partial y_1$,
$q_{k2}=\partial Q_k/\partial y_2$. Then, using
equalities~\eqref{eqhatB12}, the fact that the first derivative of
a polynomial of order one is a constant, and
relation~\eqref{eqDiffB}, we obtain
\begin{multline*}
\sum\limits_{k,s}b_{j\sigma ks1}q_{k1}+b_{j\sigma ks2}q_{k2}=
\sum\limits_{k,s}\hat B_{j\sigma ks}(D_y)Q_k(y)\\=
\sum\limits_{k,s}(\hat B_{j\sigma ks}(D_y)Q_k)({\mathcal
G}_{j\sigma ks}y)=\frac{\partial}
 {\partial \tau_{j\sigma}}
 {\mathcal B}_{j\sigma}Q,
\end{multline*}
where $Q=(Q_1,\dots,Q_N)$. Since $\lambda=-i$ is an eigenvalue of
$\tilde{\mathcal L}(\lambda)$ and $\varphi$ is the corresponding
eigenvector, it follows that ${\mathcal
B}_{j\sigma}Q|_{\gamma_{j\sigma}}=0$; hence,
$$
\big(\partial ({\mathcal B}_{j\sigma}Q)/\partial
\tau_{j\sigma}\big)|_{\gamma_{j\sigma}}=0.
$$
 It follows from the
latter relation and from the relation $\partial ({\mathcal
B}_{j\sigma}Q)/\partial \tau_{j\sigma}=\const$ that $\partial
({\mathcal B}_{j\sigma}Q)/\partial \tau_{j\sigma}=0$. Thus, we
have constructed a nontrivial solution of
system~\eqref{eqSystemqj12} and, therefore, proved that
system~\eqref{eqSystemB} is linearly dependent.

Let
\begin{equation}\label{eqSystemB'}
\{\hat{\mathcal B}_{j'\sigma'}(D_y)\}
\end{equation}
be a maximal linearly independent subsystem of
system~\eqref{eqSystemB}. In this case, any operator
$\hat{\mathcal B}_{j\sigma}(D_y)$ which does not enter
system~\eqref{eqSystemB'} can be represented as follows:
\begin{equation}\label{eqBviaB'}
\hat{\mathcal
B}_{j\sigma}(D_y)=\sum\limits_{j',\sigma'}\beta_{j\sigma}^{j'\sigma'}\hat{\mathcal
B}_{j'\sigma'}(D_y),
\end{equation}
where $\beta_{j\sigma}^{j'\sigma'}$ are some constants.

Let us introduce the notion of the consistency condition. Let
$\{Z_{j\sigma}\}\in\mathcal W^{3/2}(\gamma^\varepsilon)$ be
arbitrary functions, each of which is defined on its own interval
$\gamma_{j\sigma}^\varepsilon$. Consider the functions
$$
Z^0_{j\sigma}(r)=Z_{j\sigma}(y)|_{y=(r\cos\omega_j,\,
r(-1)^\sigma\sin\omega_j)}.
$$
Each of the functions $Z^0_{j\sigma}$ belongs to
$W^{3/2}(0,\varepsilon)$.

\begin{definition}
Let $\beta_{j\sigma}^{j'\sigma'}$ be the constants occurring
in~\eqref{eqBviaB'}. If the relations
\begin{equation}\label{eqConsistencyZ}
\int\limits_{0}^\varepsilon
r^{-1}\Bigg|\frac{d}{dr}\bigg(Z^0_{j\sigma}-\sum\limits_{j',\sigma'}\beta_{j\sigma}^{j'\sigma'}Z^0_{j'\sigma'}\bigg)\Bigg|^2dr<\infty
\end{equation}
hold for all indices $j,\sigma$ corresponding to the operators of
system~\eqref{eqSystemB} which do not enter
system~\eqref{eqSystemB'}, then we say that the {\em functions
$Z_{j\sigma}$ satisfy the consistency
condition~\eqref{eqConsistencyZ}}.
\end{definition}

\begin{remark}\label{remSufficCons}
The relation $\{Z_{j\sigma}\}\in\mathcal
H_0^{3/2}(\gamma^\varepsilon)$ is sufficient (but not necessary)
for the functions $Z_{j\sigma}$ to satisfy the consistency
condition~\eqref{eqConsistencyZ}. This follows from
Lemma~\ref{lAppL4.18Kondr}.
\end{remark}

\begin{remark}
In the paper~\cite{GurRJMP03}, of which results we use in the
present paper, the consistency condition has the form
\begin{equation}\label{eqConsistencyZEquiv}
\frac{\partial\mathbf
Z_{j\sigma}}{\partial\tau_{j\sigma}}-\sum\limits_{j',\sigma'}\beta_{j\sigma}^{j'\sigma'}
\frac{\partial\mathbf Z_{j'\sigma'}}{\partial\tau_{j'\sigma'}}\in
H_0^1 (\mathbb R^2),
\end{equation}
where $\mathbf Z_{j\sigma}\in W^2(\mathbb R^2)$ is a compactly
supported extension of $Z_{j\sigma}$ to $\mathbb R^2$ (appropriate
theorems concerning extensions of functions in angular domains can
be found in~\cite{Stein}). Let us show that
relations~\eqref{eqConsistencyZ} are equivalent
to~\eqref{eqConsistencyZEquiv}. Denote by $\mathcal G_{j\sigma}$
the operator of rotation through the angle $(-1)^\sigma\omega_j$;
in particular, the operator $\mathcal G_{j\sigma}$ takes the
positive half-line $Oy_1$ onto the ray $\gamma_{j\sigma}$.
Consider the functions $\mathbf Z_{j\sigma}^0(y)=\mathbf
Z_{j\sigma}(\mathcal G_{j\sigma}y)$. It is clear that $\mathbf
Z_{j\sigma}^0\in W^2(\mathbb R^2)$ and $\mathbf
Z_{j\sigma}^0(y_1,0)=Z_{j\sigma}^0(y_1)$. Suppose that
relations~\eqref{eqConsistencyZ} hold. Then, by
Lemma~\ref{lAppL4.8Kondr}, we have
\begin{equation}\label{eqConsistencyZ1}
\frac{\partial\mathbf Z^0_{j\sigma}}{\partial y_1
}-\sum\limits_{j',\sigma'}\beta_{j\sigma}^{j'\sigma'}
\frac{\partial\mathbf Z^0_{j'\sigma'}}{\partial y_1}\in
H_0^1(\mathbb R^2),
\end{equation}
which is equivalent to
\begin{equation}\label{eqConsistencyZ2}
\frac{\partial\mathbf
Z_{j\sigma}}{\partial\tau_{j\sigma}}\big(\mathcal
G_{j\sigma}y\big)-\sum\limits_{j',\sigma'}\beta_{j\sigma}^{j'\sigma'}
\frac{\partial\mathbf
Z_{j'\sigma'}}{\partial\tau_{j'\sigma'}}\big(\mathcal
G_{j'\sigma'}y\big)\in H_0^1 (\mathbb R^2)
\end{equation}
by the chain rule. However, by Lemma~\ref{lAppL2.2GurRJMP03}, we
have
\begin{equation}\label{eqConsistencyZ3}
\frac{\partial\mathbf
Z_{j\sigma}}{\partial\tau_{j\sigma}}\big(\mathcal
G_{j\sigma}y\big)-\frac{\partial\mathbf
Z_{j\sigma}}{\partial\tau_{j\sigma}}\big(y\big)\in H_0^1 (\mathbb
R^2)
\end{equation}
for all $\mathbf Z_{j\sigma}\in W^2(\mathbb R^2)$ because
$\partial\mathbf Z_{j\sigma}/\partial\tau_{j\sigma}\in W^1(\mathbb
R^2)$. It follows from~\eqref{eqConsistencyZ2}
and~\eqref{eqConsistencyZ3} that
relations~\eqref{eqConsistencyZEquiv} hold.

Conversely, suppose that relations~\eqref{eqConsistencyZEquiv}
hold. Using~\eqref{eqConsistencyZ3} again, we
obtain~\eqref{eqConsistencyZ2}, hence~\eqref{eqConsistencyZ1}. It
follows from~\eqref{eqConsistencyZ1} and from the boundedness of
the trace operator in appropriate weighted spaces that
$$
\frac{d}{dr}\bigg(Z^0_{j\sigma}-\sum\limits_{j',\sigma'}\beta_{j\sigma}^{j'\sigma'}Z^0_{j'\sigma'}\bigg)\in
H_0^{1/2}(0,\varepsilon).
$$
This relation and Lemma~\ref{lAppL4.18Kondr}
imply~\eqref{eqConsistencyZ}.
\end{remark}

Now we will show that the following condition is necessary and
sufficient for a given generalized solution $u$ to belong to
$W^2(G)$.

\begin{condition}\label{condConsistencyPsi-BC}
Let $u\in W^1(G)$ be a generalized solution of
problem~\eqref{eqPinG}, \eqref{eqBinG}, $\Psi_{j\sigma}$ the
right-hand sides in nonlocal conditions~\eqref{eqBinK}, and $C$
the constant vector appearing in Lemma~$\ref{lU=C+}$. Then the
functions $\Psi_{j\sigma}-{\mathbf B}_{j\sigma}C$ satisfy the
consistency condition~\eqref{eqConsistencyZ}.
\end{condition}

\begin{remark}
1. The validity of Condition~$\ref{condConsistencyPsi-BC}$
depends, in particular, on the behavior of the function $\mathbf
B_{i}^2u$ near the set (orbit) $\mathcal K$. Due
to~\eqref{eqSeparK23'}, the values of the function $\mathbf
B_{i}^2u$ near the set $\mathcal K$ depend on the values of the
function $u$ in $G\setminus\overline{\mathcal
O_{\varkappa_1}(\mathcal K)}$. Therefore, the smoothness of the
solution $u$ near the set $\mathcal K$ depends on the behavior of
$u$ outside $\mathcal K$.

2. Let us explain how the validity of
Condition~\ref{condConsistencyPsi-BC} depends on the behavior of
the functions $u(y), f_i(y),b_{is}(y),(\mathbf B_{i}^2u)(y)$ near
the set $\mathcal K$. On one hand, the vector $C$ appearing in
Lemma~\ref{lU=C+} is defined by the behavior of $u(y)$ near the
set $\mathcal K$. On the other hand, the values of $b_{is}(y)$,
$y\in\mathcal K$, together with the operators $\mathcal G_{j\sigma
ks}$, define the constants $\beta_{j\sigma}$ occurring
in~\eqref{eqBviaB'} and hence in~\eqref{eqConsistencyZ}. Finally,
the derivatives of $f_i(y),(\mathbf B_{i}^2u)(y)$, and $b_{is}(y)$
near the set $\mathcal K$ must be consistent with each other in
such a way that the absolute values of the corresponding linear
combinations of the first derivatives of $\Psi_{j\sigma}-{\mathbf
B}_{j\sigma}C$ be quadratically integrable, with the weight
$r^{-1}$, near the origin.
\end{remark}

Throughout this section, we suppose that the number $a$ is the
same as in Sec.~\ref{sectNoEigen}. The existence of such an $a$
follows from Lemma~\ref{lSpectrum} and
Condition~\ref{condProperEigen}.

\begin{theorem}\label{thuinW^2ProperEigen}
Let Condition~$\ref{condProperEigen}$ hold, and let $u\in W^1(G)$
be a generalized solution of problem~\eqref{eqPinG},
\eqref{eqBinG} with right-hand side $\{f_0,f_i\}\in L_2(G)\times
\mathcal W^{3/2}(\partial G)$. Then $u\in W^2(G)$ if and only if
Condition~$\ref{condConsistencyPsi-BC}$ holds.
\end{theorem}
\begin{proof}
1. {\em Necessity.} Let $u\in W^2(G)$, and let $U=(U_1,\dots,U_N)$
be a function corresponding to the set (orbit) $\mathcal K$.
Clearly, $U\in\mathcal W^2(K^\varepsilon)$. It follows from
Lemma~\ref{lU=C+} that $U=C+U'$, where $U'\in\mathcal
H_a^2(K^\varepsilon)$. Since we additionally have
$U'=U-C\in\mathcal W^2(K^\varepsilon)$, it follows from Sobolev's
embedding theorem that $U'(0)=0$. This relation and
Lemma~\ref{lAppL3.1GurRJMP03} imply that the functions
$\Psi_{j\sigma}-{\mathbf B}_{j\sigma}C={\mathbf B}_{j\sigma}U'$
satisfy the consistency condition~\eqref{eqConsistencyZ}

2. {\em Sufficiency.} Suppose that
Condition~$\ref{condConsistencyPsi-BC}$ holds. Similarly to the
proof of Lemma~\ref{lUinW^2}, we infer that the function
$U'\in\mathcal H_a^2(K^\varepsilon)$ is a solution of
problem~\eqref{eqPBU'}.

Using Condition~$\ref{condConsistencyPsi-BC}$ and
relations~\eqref{eqF-BC}, we can apply
Lemma~\ref{lAppL3.3GurRJMP03}, which ensures the existence of a
function $V$ satisfying relations~\eqref{eqVHa} and~\eqref{eqPBV}.

Further, similarly to the proof of Lemma~\ref{lUinW^2}, we obtain
$U'-V\in\mathcal H_{a}^2(K^\varepsilon)$, $\{{\mathcal
P}_{j}(U'_j-V_j)\}\in\mathcal H_0^0(K^\varepsilon)$, $\{{\mathcal
B}_{j\sigma}(U'-V)|_{\gamma_{j\sigma}^\varepsilon}\}\in \mathcal
H_0^{3/2}(\gamma^\varepsilon)$. It follows from these relations
and from Lemma~\ref{lAppL3.4GurRJMP03} that all the second
derivatives of the function $U'-V$ belong to $\mathcal
W^0(K^\varepsilon)$. Combining this fact with the relations
$$
U'-V\in\mathcal H_{a}^2(K^\varepsilon)\subset\mathcal
H_{a-1}^1(K^\varepsilon)\subset\mathcal W^1(K^\varepsilon)
$$
yields $U'-V\in\mathcal W^2(K^\varepsilon)$. Now the conclusion of
the theorem results from \eqref{eqVHa} and~\eqref{eqU=C+}.
\end{proof}

Note that Theorem~\ref{thuinW^2ProperEigen} enables us to conclude
whether or not a given solution $u$ is smooth near the set
$\mathcal K$, provided that we know the asymptotics for $u$ of the
kind~\eqref{eqU=C+} near the set $\mathcal K$ (i.e., if we know
the value of the constant\footnote{As for the calculation of the
constant $C$, see~\cite{GurPetr03, GurFDE03}.} $C$).
Theorem~\ref{thuinW^2ProperEigen} shows what affects the
smoothness of solutions in principle. Below, this will enable us
to obtain a constructive condition which is necessary and
sufficient for {\em any} generalized solution to belong to
$W^2(G)$.

\subsection{Problem with Nonhomogeneous Nonlocal Conditions}\label{subsectNonHomogProb}
If any generalized solution of problem~\eqref{eqPinG},
\eqref{eqBinG} belongs to $W^2(G)$, then we say that {\em
smoothness} of generalized solutions {\em preserves}. If there
exists a generalized solution of problem~\eqref{eqPinG},
\eqref{eqBinG} which does not belong to $W^2(G)$, then we say that
{\em smoothness} of generalized solutions {\em can be violated}.

In this subsection, we formulate necessary and sufficient
conditions for the smoothness of solutions to preserve. First of
all, we show that right-hand sides $f_i$ in nonlocal
conditions~\eqref{eqBinG} cannot be arbitrary functions from
$W^{3/2}(\Gamma_i)$, they must satisfy the consistency
condition~\eqref{eqConsistencyZ}.

Denote by $\mathcal S^{3/2}(\partial G)$ the set of functions
$\{f_i\}\in\mathcal W^{3/2}(\partial G)$ such that the functions
$F_{j\sigma}$ (see~\eqref{eqytoy'}) satisfy the consistency
condition~\eqref{eqConsistencyZ}.

It follows from~\cite[Lemma~3.2]{GurRJMP03} that the set $\mathcal
S^{3/2}(\partial G)$ is not closed in the space $\mathcal
W^{3/2}(\partial G)$.

Smoothness of generalized solutions of problem~\eqref{eqPinG},
\eqref{eqBinG} can be violated if right-hand sides in nonlocal
conditions~\eqref{eqBinG} do not satisfy the consistency
condition. The following result is valid.
\begin{theorem}\label{thUNonSmFNonConsist}
Let Condition~$\ref{condProperEigen}$ hold. Then there exist a
function $\{f_0,f_i\}\in L_2(G)\times \mathcal W^{3/2}(\partial
G)$, $\{f_i\}\notin\mathcal S^{3/2}(\partial G)$, and a function
$u\in W^1(G)$ such that $u$ is a generalized solution of
problem~\eqref{eqPinG}, \eqref{eqBinG} with the right-hand side
$\{f_0,f_i\}$ and $u\notin W^2(G)$.
\end{theorem}

To prove Theorem~\ref{thUNonSmFNonConsist}, we preliminarily
establish two auxiliary results.
\begin{lemma}\label{lDense}
Let $f\in W^2(\mathbb R^2)$ and $f(0)=0$. Then there exists a
sequence $f^n\in C_0^\infty(\mathbb R^2)$, $n=1,2,\dots$, such
that $f^n(y)=0$ in some neighborhood of the origin $($depending on
$n)$ and $f^n\to f$ in $W^{2}(\mathbb R^2)$.
\end{lemma}
\begin{proof}
As is well known, the set $C_0^\infty(\mathbb R^2)$ is dense in
$W^2(\mathbb R^2)$. On the other hand, it follows from Sobolev's
embedding theorem and Riesz' theorem on the general form of a
linear continuous functional in a Hilbert space that the set
$\{u\in W^2(\mathbb R^2): u(0)=0\}$ is a closed subspace in
$W^2(\mathbb R^2)$ of codimension one. Therefore,
by~\cite[Lemma~8.1]{Kr}, the set $C_0^\infty(\mathbb
R^2)\cap\{u\in W^2(\mathbb R^2):\ u(0)=0\}$ is dense in $\{u\in
W^2(\mathbb R^2): u(0)=0\}$. Hence, it suffices to prove the lemma
for a function $f\in C_0^\infty(\mathbb R^2)$ such that $f(0)=0$.
Introduce a function $\xi\in C_0^\infty[0,\infty)$ such that
$0\le\xi(t)\le 1$, $\xi(t)=1$ for $t<1$, and $\xi(t)=0$ for $t>2$.
Consider the sequence
$$
\xi^n(y)=\xi\Big(-\frac{\ln r}{n}\Big),
$$
where $r=|y|$. Clearly, $0\le\xi^n(y)\le1$, $\xi^n(y)=0$ for
$|y|<e^{-2n}$, $\xi^n(y)=1$ for $|y|>e^{-n}$, $|\xi^n_{y_k}|\le
c_1/(rn)$, $|\xi^n_{y_iy_k}|\le c_2/(r^2n)$, where $c_1,c_2>0$ do
not depend on $n$ and $y$.

Let us show that the sequence $\xi^n f$ converges to $f$ in
$W^2(\mathbb R^2)$ as $n\to\infty$. Clearly,
\begin{equation}\label{eqIntf}
\int\limits_{\mathbb R^2} |f-\xi^n f|^2
dy\le\int\limits_{|y|<e^{-n}}|f|^2dy\to0.
\end{equation}
Further,
\begin{equation}\label{eqIntfy}
\int\limits_{\mathbb R^2} |(f-\xi^n f)_{y_k}|^2 dy\le2\left[
\int\limits_{|y|<e^{-n}} |f_{y_k}|^2
dy+\frac{c_1^2}{n^2}\int\limits_{e^{-2n}<|y|<e^{-n}}
|f|^2\frac{1}{r^2}dy\right]\to0.
\end{equation}
Indeed, the first bracketed term tends to zero because
$e^{-n}\to0$, whereas the second term can be estimated from above
by the following expression:
$$
2\pi\max\limits_{y\in\mathbb
R^2}|f|^2\,\frac{c_1^2}{n^2}\int\limits_{e^{-2n}}^{e^{-n}}\frac{dr}{r}=
2\pi\max\limits_{y\in\mathbb R^2}|f|^2\,\frac{c_1^2}{n}\to0.
$$
Finally,
\begin{multline}\label{eqIntfyy}
\int\limits_{\mathbb R^2} |(f-\xi^n f)_{y_iy_k}|^2 dy \le4\left[
\int\limits_{|y|<e^{-n}} |f_{y_iy_k}|^2
dy\right.\\
+\left.\frac{c_1^2}{n^2}\int\limits_{e^{-2n}<|y|<e^{-n}}
(|f_{y_i}|^2+|f_{y_k}|^2)\frac{1}{r^2}dy+\frac{c_2^2}{n^2}\int\limits_{e^{-2n}<|y|<e^{-n}}
|f|^2\frac{1}{r^4}dy\right]\to0.
\end{multline}
Indeed, the first and the second bracketed terms tend to zero
because of the reasons similar to the above. To prove that the
third term tends to zero, we recall that $f\in C_0^\infty(\mathbb
R^2)$ and $f(0)=0$. Therefore, by the Taylor formula, $f(y)=O(r)$
as $r\to0$, and hence the third term can be estimated from above
similarly to the second one.
\end{proof}

Set
\begin{equation}\label{eqEpsilon'}
\varepsilon'=d_\chi\min(\varepsilon,\varkappa_2),
\end{equation}
where $d_\chi$ is defined in~\eqref{eqd1d2}.

\begin{lemma}\label{lUnonSmoothZnonConsist}
Let Condition~$\ref{condProperEigen}$ hold. Let a function
$\{Z_{j\sigma}\}\in\mathcal W^{3/2}(\gamma^\varepsilon)$ be such
that $\supp \{Z_{j\sigma}\}\subset\mathcal O_{\varepsilon/2}(0)$,
$Z_{j\sigma}(0)=0$, and the functions $Z_{j\sigma}$ do not satisfy
the consistency condition~\eqref{eqConsistencyZ}. Then there
exists a function $U\in\mathcal H_a^2(K)\cap\mathcal W^1(K)$ such
that $\supp U\subset\mathcal O_{\varepsilon'}(0)$,
$U\notin\mathcal W^2(K^\varepsilon)$, and $U$ satisfies the
relations
\begin{equation}\label{eqUnonSmoothZnonConsist}
\{{\mathbf P}_{j}U_j\}\in\mathcal W^0(K^\varepsilon),\qquad
\{{\mathbf
B}_{j\sigma}U|_{\gamma_{j\sigma}^\varepsilon}-Z_{j\sigma}\}\in\mathcal
H_0^{3/2}( \gamma^\varepsilon).
\end{equation}
\end{lemma}
\begin{proof}
By Lemma~\ref{lDense}, there exists a sequence of vector-valued
functions $\{Z_{j\sigma}^n\}\in\mathcal W^{3/2}(\gamma)$,
$n=1,2,\dots$, such that $\supp Z_{j\sigma}^n\subset\mathcal
O_{\varepsilon}(0)$, $Z_{j\sigma}^n(0)=0$, $Z_{j\sigma}^n$ satisfy
the consistency condition~\eqref{eqConsistencyZ} (because the
functions $Z_{j\sigma}^n$ vanish near the origin), and
$Z_{j\sigma}^n\to Z_{j\sigma}$ in $W^{3/2}(\gamma_j)$. Now we
apply Lemma~3.5 in~\cite{GurRJMP03}, which ensures the existence
of a sequence $V^n=(V_1^n,\dots,V_N^n)$ satisfying the following
conditions: $V^n\in\mathcal W^2(K^d)\cap \mathcal H_0^1(K^d)$ for
any $d>0$,
\begin{equation}\label{eqPBVdeltainK}
{\mathcal P}_{j}V_j^n=0 \quad (y\in K_j),\qquad {\mathcal
B}_{j\sigma}V^n =Z_{j\sigma}^n(y) \quad (y\in\gamma_{j\sigma}),
\end{equation}
and the sequence $V^n$ converges to some function $V\in\mathcal
H_0^{1}(K^d)$ in $\mathcal H_0^{1}(K^d)$ for any $d>0$. Passing to
the limit in the first equality in~\eqref{eqPBVdeltainK} in the
sense of distributions and in the second equality in
$W^{1/2}(\gamma_{j\sigma}^d)$ for any $d>0$, we obtain
\begin{equation}\label{eqPBVinK}
{\mathcal P}_{j}V_j=0 \quad (y\in K_j),\qquad {\mathcal
B}_{j\sigma}V =Z_{j\sigma}(y) \quad (y\in\gamma_{j\sigma}).
\end{equation}
In particular, it follows from these relations and from
Lemma~\ref{lU=C+} that $V=C+V'$, where $V'\in\mathcal
H_{a}^{2}(K^\varepsilon)$ and $C=(C_1, \dots, C_N)$ is a constant
vector. Therefore, $C=V-V'\in\mathcal H_0^{1}(K^\varepsilon)$, and
hence $C=0$. Thus, we have proved that
\begin{equation}\label{eqVHa2W1}
 V\in\mathcal
H_{a}^{2}(K^d)\cap\mathcal W^1(K^d)\quad\forall d>0.
\end{equation}

Consider a cut-off function $\xi\in C_0^\infty(|y|<\varepsilon')$
equal to one near the origin. Set $U=\xi V$. Clearly, $\supp
U\subset\mathcal O_{\varepsilon'}(0)$ and, by virtue
of~\eqref{eqVHa2W1},
\begin{equation}\label{eqUHa2W1}
U\in \mathcal H_a^2(K)\cap\mathcal W^1(K).
\end{equation}

2. We claim that $U$ is the desired function. Indeed, using
Leibniz' formula, relations~\eqref{eqPBVinK} and
Lemma~\ref{lAppL3.3'Kondr}, we
infer~\eqref{eqUnonSmoothZnonConsist}.

It remains to prove that $U\notin \mathcal W^2(K^\varepsilon)$.
Assume the contrary. Let $U\in \mathcal W^2(K^\varepsilon)$. In
this case, it follows from Sobolev's embedding theorem and from
the belonging $U\in\mathcal H_a^2(K^\varepsilon)$ that $U(0)=0$.
Combining this fact with Lemma~\ref{lAppL3.1GurRJMP03} implies
that the functions ${\mathbf
B}_{j\sigma}U|_{\gamma_{j\sigma}^\varepsilon}$ satisfy the
consistency condition~\eqref{eqConsistencyZ}. However, the
functions ${\mathbf
B}_{j\sigma}U|_{\gamma_{j\sigma}^\varepsilon}-Z_{j\sigma}$ do not
satisfy the consistency condition~\eqref{eqConsistencyZ} in that
case. This contradicts~\eqref{eqUnonSmoothZnonConsist} (see
Remark~\ref{remSufficCons}).
\end{proof}

\begin{proof}[Proof of Theorem~$\ref{thUNonSmFNonConsist}$] 1. We will
construct a generalized solution $u$ supported near the set
$\mathcal K$ (so that $\mathbf B^2_{i}u=0$ due
to~\eqref{eqSeparK23'}) and such that $u\notin W^2(G)$.

It was shown in the course of the proof of Lemma~3.2
in~\cite{GurRJMP03} that there exists a function
$\{Z_{j\sigma}\}\in\mathcal W^{3/2}(\gamma)$ such that $\supp
Z_{j\sigma}\subset\mathcal O_{\varepsilon/2}(0)$,
$Z_{j\sigma}(0)=0$, and the functions $Z_{j\sigma}$ do not satisfy
the consistency condition~\eqref{eqConsistencyZ}. By
Lemma~\ref{lUnonSmoothZnonConsist}, there exists a function $U\in
\mathcal H_a^2(K)\cap\mathcal W^1(K)$ such that $\supp
U\subset\mathcal O_{\varepsilon'}(0)$, $U\notin\mathcal W^2(K)$,
and $U$ satisfies relations~\eqref{eqUnonSmoothZnonConsist}.
Therefore, $\{{\mathbf P}_{j}U_j\}\in \mathcal
W^0(K^\varepsilon)$, $\{{\mathbf
B}_{j\sigma}U|_{\gamma_{j\sigma}^\varepsilon}\}\in\mathcal
W^{3/2}(\gamma^\varepsilon)$, and the functions ${\mathbf
B}_{j\sigma}U|_{\gamma_{j\sigma}^\varepsilon}$ do not satisfy the
consistency condition~\eqref{eqConsistencyZ}.

2. Introduce a function $u(y)$ such that $u(y)=U_j(y'(y))$ for
$y\in\mathcal O_{\varepsilon'}(g_j)$ and $u(y)=0$ for
$y\notin\mathcal O_{\varepsilon'}(\mathcal K)$, where $y'\mapsto
y(g_j)$ is the change of variables inverse to the change of
variables $y\mapsto y'(g_j)$ from Sec.~\ref{subsectStatement}.
Since $\supp u\subset \mathcal O_{\varkappa_1}(\mathcal K)$, it
follows that $\mathbf B_{i}^2u=0$. Therefore, $u(y)$ is the
desired generalized solution of problem~\eqref{eqPinG},
\eqref{eqBinG}.
\end{proof}

\bigskip

Theorem~\ref{thUNonSmFNonConsist} shows that, if we want that {\em
any} generalized solution of problem~\eqref{eqPinG},
\eqref{eqBinG} be smooth, then we must take right-hand sides
$\{f_0,f_i\}$ from the space $L_2(G)\times \mathcal
S^{3/2}(\partial G)$.

\medskip

Let $v$ be an arbitrary function from the space
$W^{2}(G\setminus\overline{\mathcal O_{\varkappa_1}(\mathcal
K)})$. Consider the change of variables $y\mapsto y'(g_j)$ from
Sec.~\ref{subsectStatement} again and introduce the functions
$$
B^v_{j\sigma}(y')=(\mathbf B_{i}^2v)(y(y')),\quad
y'\in\gamma_{j\sigma}^\varepsilon
$$
(cf. functions~\eqref{eqytoy'}). We prove that the following
condition is necessary and sufficient for any generalized solution
to be smooth.
\begin{condition}\label{condB2vB1CConsistency}
\begin{enumerate}
\item For any $v\in
W^{2}(G\setminus\overline{\mathcal O_{\varkappa_1}(\mathcal K)})$,
the functions $B^v_{j\sigma}$ satisfy the consistency
condition~\eqref{eqConsistencyZ}.
\item For any constant vector $C=(C_1,\dots,C_N)$,
the functions $\mathbf
B_{j\sigma}C|_{\gamma_{j\sigma}^\varepsilon}$ satisfy the
consistency condition~\eqref{eqConsistencyZ}.
\end{enumerate}
\end{condition}

Note that the validity of Condition~\ref{condB2vB1CConsistency},
unlike Condition~\ref{condConsistencyPsi-BC}, does not depend on a
generalized solution. It depends only on the operators $\mathbf
B_i^1$ and $\mathbf B_i^2$ and on the geometry of the domain $G$
near the set (orbit) $\mathcal K$. This is quite natural because
we study the smoothness of {\em all} generalized solutions in this
section (while in Sec.~\ref{subsectuFixed}, we have investigated
the smoothness of a fixed solution).

\begin{theorem}\label{thSmoothfne0}
Let Condition~$\ref{condProperEigen}$ hold. Then the following
assertions are true.
\begin{enumerate}
\item
If Condition~$\ref{condB2vB1CConsistency}$ is fulfilled and $u\in
W^1(G)$ is a generalized solution of problem~\eqref{eqPinG},
\eqref{eqBinG} with right-hand side $\{f_0,f_i\}\in L_2(G)\times
\mathcal S^{3/2}(\partial G)$, then $u\in W^2(G)$.
\item
If Condition~$\ref{condB2vB1CConsistency}$ fails, then there
exists a right-hand side $\{f_0,f_i\}\in L_2(G)\times \mathcal
S^{3/2}(\partial G)$ and a generalized solution $u\in W^1(G)$ of
problem~\eqref{eqPinG}, \eqref{eqBinG}  such that $u\notin
W^2(G)$.
\end{enumerate}
\end{theorem}
\begin{proof}
1. {\em Sufficiency.} Let Condition~\ref{condB2vB1CConsistency}
hold, and let $u\in W^1(G)$ be an arbitrary generalized solution
of problem~\eqref{eqPinG}, \eqref{eqBinG} with right-hand side
$\{f_0,f_i\}\in L_2(G)\times \mathcal S^{3/2}(\partial G)$.
By~\eqref{eqSmoothOutsideK}, we have $u\in
W^{2}(G\setminus\overline{\mathcal O_{\varkappa_1}(\mathcal K)})$.
Therefore, by Condition~\ref{condB2vB1CConsistency}, the functions
$B^u_{j\sigma}$ satisfy the consistency
condition~\eqref{eqConsistencyZ}. Let $C$ be a constant vector
defined by Lemma~\ref{lU=C+}. Using
Condition~\ref{condB2vB1CConsistency} again, we see that the
functions $\mathbf B_{j\sigma}C$ satisfy the consistency
condition~\eqref{eqConsistencyZ}. Since $\{f_i\}\in\mathcal
S^{3/2}(\partial G)$, it follows that the functions $F_{j\sigma}$
satisfy the consistency condition~\eqref{eqConsistencyZ}.
Therefore, the functions
$\Psi_{j\sigma}=F_{j\sigma}-B^u_{j\sigma}$ and $\mathbf
B_{j\sigma}C$ satisfy Condition~\ref{condConsistencyPsi-BC}.
Applying Theorem~\ref{thuinW^2ProperEigen}, we obtain $u\in
W^2(G)$.

2. {\em Necessity.} Let Condition~\ref{condB2vB1CConsistency}
fail. In this case, there exist a function $v\in
W^{2}(G\setminus\overline{\mathcal O_{\varkappa_1}(\mathcal K)})$
and a constant vector $C=(C_1,\dots,C_N)$ such that the functions
$B^v_{j\sigma}+\mathbf B_{j\sigma}C$ do not satisfy the
consistency condition~\eqref{eqConsistencyZ} (one can assume that
either $v=0,\,C\ne0$ or $v\ne0,\,C=0$). Extend the function $v$ to
the domain $G$ in such a way that $v(y)=0$ for $y\in\mathcal
O_{\varkappa_1/2}(\mathcal K)$ and $v\in W^2(G)$.

Consider functions $F_{j\sigma}'\in
C^\infty(\overline{\gamma_{j\sigma}})$ such that
$$
F_{j\sigma}'(y)=B^v_{j\sigma}(0)+(\mathbf B_{j\sigma}C)(0),\quad
|y|<\varepsilon/2,\qquad F_{j\sigma}'(y)=0,\quad |y|>\varepsilon.
$$
Since $\partial F_{j\sigma}'/\partial\tau_{j\sigma}=0$ near the
origin, it follows that the functions $F_{j\sigma}'$ satisfy the
consistency condition~\eqref{eqConsistencyZ}. By construction,
$$
\{F_{j\sigma}'-B^v_{j\sigma}-\mathbf
B_{j\sigma}C|_{\gamma_{j\sigma}^\varepsilon}\}\in\mathcal
W^{3/2}(\gamma^\varepsilon),\qquad
(F_{j\sigma}'-B^v_{j\sigma}-\mathbf B_{j\sigma}C)|_{y=0}=0,
$$
and the functions $F_{j\sigma}'-B^v_{j\sigma}-\mathbf
B_{j\sigma}C$ do not satisfy the consistency
condition~\eqref{eqConsistencyZ}. By
Lemma~\ref{lUnonSmoothZnonConsist}, there exists a function
$U'\in\mathcal H_a^2(K)\cap\mathcal W^1(K)$ such that $\supp
U'\subset\mathcal O_{\varepsilon'}(0)$, $U'\notin\mathcal
W^2(K^\varepsilon)$, and
\begin{equation}\label{eqPU'-PsiinH0}
\{\mathbf P_jU'_j\}\in\mathcal W^0(K^\varepsilon),
\end{equation}
$$
 \big\{\big(\mathbf
B_{j\sigma}U'-(F_{j\sigma}'-B^v_{j\sigma}-\mathbf
B_{j\sigma}C)\big)|_{\gamma_{j\sigma}^\varepsilon}\big\}\in\mathcal
H_0^{3/2}(\gamma^\varepsilon).
$$
One can also write the latter relation as follows:
\begin{equation}\label{eqBU'-PsiinH0}
\{\mathbf
B_{j\sigma}(U'+C)|_{\gamma_{j\sigma}^\varepsilon}+B^v_{j\sigma}-F_{j\sigma}'\}\in\mathcal
H_0^{3/2}(\gamma^\varepsilon).
\end{equation}
Introduce a function $u'(y)$ such that
$u'(y)=U'_j(y'(y))+\xi_j(y)C_j$ for $y\in\mathcal
O_{\varepsilon'}(g_j)$ and $u'(y)=0$ for $y\notin\mathcal
O_{\varepsilon'}(\mathcal K)$, where $y'\mapsto y(g_j)$ is the
change of variables inverse to the change of variables $y\mapsto
y'(g_j)$ from Sec.~\ref{subsectStatement}, while $\xi_j\in
C_0^\infty(O_{\varepsilon'}(g_j))$, $\xi_j(y)=1$ for $y\in\mathcal
O_{\varepsilon'/2}(g_j)$, and $\varepsilon'$ is given
by~\eqref{eqEpsilon'}. Let us prove that the function $u=u'+v$ is
the desired one. Clearly, $u\in W^1(G)$, $u\notin W^2(G)$, and $u$
satisfies relations~\eqref{eqSmoothOutsideK}. It follows from the
belonging $v\in W^2(G)$ and from relations~\eqref{eqPU'-PsiinH0}
that
$$
\mathbf Pu\in L_2(G).
$$
Consider the functions $f_i=u|_{\Gamma_i}+\mathbf B_{i}^1
u+\mathbf B_{i}^2 u$. It follows from the belonging $v\in W^2(G)$,
from relations~\eqref{eqSmoothOutsideK}, and from
inequality~\eqref{eqSeparK23'} that $f_i\in
W^{3/2}\bigl(\Gamma_i\setminus\overline{\mathcal O_\delta(\mathcal
K)}\bigr)$ for any $\delta>0$. Consider the behavior of $f_i$ near
the set $\mathcal K$. Note that $\mathbf B_{i}^2 u'=0$
by~\eqref{eqSeparK23'}. Furthermore, $v|_{\Gamma_i}+\mathbf
B_{i}^1 v=0$ for $y\in\mathcal O_{\varkappa_1/D_\chi }(\mathcal
K)$. Therefore,
\begin{equation}\label{eqfi}
 f_i=u'|_{\Gamma_i}+\mathbf
B_{i}^1 u'+\mathbf B_{i}^2 v\quad (y\in\mathcal
O_{\varkappa_1/D_\chi }(\mathcal K)).
\end{equation}
Introduce the functions $F_{j\sigma}(y')=f_i(y(y'))$, where
$y\mapsto y'(g_j)$ is the change of variables from
Sec.~\ref{subsectStatement}. It follows from~\eqref{eqfi} and
from~\eqref{eqBU'-PsiinH0} that
$\{F_{j\sigma}-F_{j\sigma}'\}\in\mathcal
H_0^{3/2}(\gamma^\varepsilon)$. Therefore,
$\{F_{j\sigma}\}\in\mathcal W^{3/2}(\gamma^\varepsilon)$ and the
functions $F_{j\sigma}$, together with $F_{j\sigma}'$, satisfy the
consistency condition~\eqref{eqConsistencyZ}. Hence
$\{f_i\}\in\mathcal S^{3/2}(\partial G)$, which completes the
proof.
\end{proof}

\subsection{Problem with Regular and Homogeneous Nonlocal Conditions}\label{subsectHomogProb}

\begin{definition}\label{defAdmit}
We say that a function $v\in W^{2}(G\setminus\overline{\mathcal
O_{\varkappa_1}(\mathcal K)})$ is {\em admissible} if there exists
a constant vector $C=(C_1,\dots,C_N)$ such that
\begin{equation}\label{eqvadmissible}
B_{j\sigma}^v(0)+(\mathbf B_{j\sigma}C)(0)=0,\quad j=1,\dots,N,\
\sigma=1,2.
\end{equation}
Any vector $C$ satisfying relations~\eqref{eqvadmissible} is said
to be an {\em admissible vector corresponding to the
function~$v$.}
\end{definition}

\begin{remark} The set of admissible functions is linear. It is
clear that the function $v=0$ is admissible, while the victor
$C=0$ is an admissible vector corresponding to the function $v=0$.
In fact, the set of admissible functions is much wider. In
particular, it contains all generalized solutions of
problem~\eqref{eqPinG}, \eqref{eqBinG} with homogeneous nonlocal
conditions for all $f\in L_2(G)$ (see the proof of
Theorem~\ref{thSmoothf0} below). Therefore, this set consists of
infinitely many elements due to Theorem~2.1
in~\cite{GurMatZam05}.\footnote{Theorem~2.1 in~\cite{GurMatZam05}
asserts that problem~\eqref{eqPinG}, \eqref{eqBinG} has the
Fredholm property.}

As for the set of admissible vectors corresponding to an
admissible function $v$, it is an affine space of the form
\begin{equation}\label{eqAdmissibleSet}
\{C+\tilde C:\ \tilde C=\const,\,(\mathbf B_{j\sigma}\tilde
C)(0)=0\},
\end{equation}
where $C$ is a fixed admissible vector corresponding to $v$ (if
the relations $(\mathbf B_{j\sigma}\tilde C)(0)=0$, $j=1,\dots,N$,
$\sigma=1,2$, hold for $\tilde C=0$ only, then the set of
admissible vectors corresponding to $v$ consists of a unique
vector). Indeed, if a constant vector $D$ belongs to the
set~\eqref{eqAdmissibleSet}, then $(\mathbf
B_{j\sigma}(D-C))(0)=0$ and, therefore,
$$
B_{j\sigma}^v(0)+(\mathbf
B_{j\sigma}D)(0)=B_{j\sigma}^v(0)+(\mathbf B_{j\sigma}C)(0)=0
$$
due to~\eqref{eqvadmissible}, i.e., the vector $D$ is admissible.
Conversely, if $D$ is an admissible vector corresponding to $v$,
then
$$
B_{j\sigma}^v(0)+(\mathbf B_{j\sigma}D)(0)=0.
$$
Subtracting~\eqref{eqvadmissible} from this equality yields
$(\mathbf B_{j\sigma}(D-C))(0)=0$.
\end{remark}

\medskip

\begin{definition}
Right-hand sides $f_i$ in nonlocal conditions~\eqref{eqBinG} are
said to be \textit{regular} if $\{f_i\}\in\mathcal
S^{3/2}(\partial G)$ and $f_i|_{\overline{\Gamma_i}\cap\mathcal
K}=0$, $i=1,\dots,N$.
\end{definition}
In particular, right-hand sides $\{f_i\}\in\mathcal
H_0^{3/2}(\partial G)$ are regular due to the Sobolev embedding
theorem and Remark~\ref{remSufficCons}. In this subsection, we
prove that the following condition is necessary and sufficient for
any generalized solution of problem~\eqref{eqPinG}, \eqref{eqBinG}
with regular $f_i$ to be smooth.

\begin{condition}\label{condBv+BCConsist}
For each admissible function $v$ and for each admissible vector
$C$ corresponding to $v$, the functions $B_{j\sigma}^v+\mathbf
B_{j\sigma}C$ satisfy the consistency
condition~\eqref{eqConsistencyZ}.
\end{condition}

Note that Condition~\ref{condBv+BCConsist} is, in general, weaker
than Condition~\ref{condB2vB1CConsistency}.

\begin{theorem}\label{thSmoothf0}
Let Condition~$\ref{condProperEigen}$ hold. Then the following
assertions are true.
\begin{enumerate}
\item
If Condition~$\ref{condBv+BCConsist}$ is fulfilled and $u\in
W^1(G)$ is a generalized solution of problem~\eqref{eqPinG},
\eqref{eqBinG} with right-hand side $\{f_0,f_i\}\in
L_2(G)\times\mathcal S^{3/2}(\partial G)$, where $f_i$ are
regular, then $u\in W^2(G)$.
\item
If Condition~$\ref{condBv+BCConsist}$ fails, then there exists a
right-hand side $\{f_0,f_i\}\in L_2(G)\times\mathcal
H_0^{3/2}(\partial G)$ and a generalized solution $u\in W^1(G)$ of
problem~\eqref{eqPinG}, \eqref{eqBinG} such that $u\notin W^2(G)$.
\end{enumerate}
\end{theorem}
\begin{proof}
1. {\em Sufficiency.} Let Condition~\ref{condBv+BCConsist} hold,
and let $u\in W^1(G)$ be an arbitrary generalized solution of
problem~\eqref{eqPinG}, \eqref{eqBinG} with right-hand side
$\{f_0,f_i\}\in L_2(G)\times\mathcal S^{3/2}(\partial G)$,
$f_i|_{\overline{\Gamma_i}\cap\mathcal K}=0$.
By~\eqref{eqSmoothOutsideK}, we have $u\in
W^{2}(G\setminus\overline{\mathcal O_{\varkappa_1}(\mathcal K)})$.

It follows from the properties of $f_i$ that the right-hand sides
in nonlocal conditions~\eqref{eqBinK} have the form
\begin{equation}\label{eqPsi=B^u}
\Psi_{j\sigma}=F_{j\sigma}-B_{j\sigma}^u,
\end{equation}
where $F_{j\sigma}\in\mathcal W^{3/2}(\gamma^\varepsilon)$,
$F_{j\sigma}(0)=0$, and $F_{j\sigma}$ satisfy the consistency
condition~\eqref{eqConsistencyZ}.

 Further, let $U=C+U'$, where
$U'\in\mathcal H_{a}^{2}(K^\varepsilon)$ and $C$ are the function
and the constant vector defined in Lemma~\ref{lU=C+}. It follows
from~\eqref{eqBinK} and~\eqref{eqPsi=B^u} that
$$
{\mathbf B}_{j\sigma}U'=F_{j\sigma}-(B_{j\sigma}^u+{\mathbf
B}_{j\sigma}C).
$$
Since $\{B_{j\sigma}^u+{\mathbf
B}_{j\sigma}C|_{\gamma_{j\sigma}^\varepsilon}-F_{j\sigma}\}\in\mathcal
W^{3/2}(\gamma^\varepsilon)$ and $U'\in \mathcal
H_{a}^{2}(K^\varepsilon)$, it follows that
$$
\{B_{j\sigma}^u+{\mathbf
B}_{j\sigma}C|_{\gamma_{j\sigma}^\varepsilon}-F_{j\sigma}\}=\{-{\mathbf
B}_{j\sigma}U'\}\in\mathcal
W^{3/2}(\gamma^\varepsilon)\cap\mathcal
H_a^{3/2}(\gamma^\varepsilon).
$$
Therefore, $B_{j\sigma}^u(0)+({\mathbf B}_{j\sigma}C)(0)=0$ (since
$F_{j\sigma}(0)=0$ due to the above), i.e., $u$ is an admissible
function and $C$ is an admissible vector corresponding to $u$.
Hence, by virtue of~\eqref{eqPsi=B^u} and by
Condition~\ref{condBv+BCConsist},
Condition~\ref{condConsistencyPsi-BC} holds. Combining this fact
with Theorem~\ref{thuinW^2ProperEigen} implies $u\in W^2(G)$.

2. {\em Necessity.} Let Condition~\ref{condBv+BCConsist} fail. In
this case, there exists a function $v\in
W^{2}(G\setminus\overline{\mathcal O_{\varkappa_1}(\mathcal K)})$
and a constant vector $C=(C_1,\dots,C_N)$ such that
$B^v_{j\sigma}(0)+(\mathbf B_{j\sigma}C)(0)=0$ and the functions
$B^v_{j\sigma}+\mathbf B_{j\sigma}C$ do not satisfy the
consistency condition~\eqref{eqConsistencyZ}.

We must find a function  $u\in W^1(G)$ such that $u\notin W^2(G)$
and
$$
\mathbf Pu\in L_2(G),\qquad u|_{\Gamma_i}+\mathbf B_{i}^1
u+\mathbf B_{i}^2 u\in H_0^{3/2}(\Gamma_i).
$$
To do this, one can repeat the proof of assertion~2 of
Theorem~\ref{thSmoothfne0}, assuming that $v$ is the above
function, $C$ is the above constant vector, and
$F_{j\sigma}'(y)\equiv0$ (which is possible due to the relation
$B^v_{j\sigma}(0)+(\mathbf B_{j\sigma}C)(0)=0$).
\end{proof}

\begin{corollary}\label{corSmoothf0NeighbK}
Let Condition~$\ref{condProperEigen}$ hold. If
Condition~$\ref{condBv+BCConsist}$ fails, then there exist a
right-hand side $\{f_0,f_i\}\in L_2(G)\times\mathcal
H_0^{3/2}(\partial G)$, where $f_i(y)=0$ for
$y\in\Gamma_i\cap\mathcal O_{\varkappa_2}(\mathcal K)$, and a
generalized solution $u\in W^1(G)$ of problem~\eqref{eqPinG},
\eqref{eqBinG} such that $u\notin W^2(G)$.
\end{corollary}
The proof of this corollary results from assertion~2 in
Theorem~\ref{thSmoothf0}, from the embedding $H_0^2(G)\subset
W^2(G)$, and from assertion~1 of the following lemma.

\begin{lemma}\label{lHomogeneous}
\begin{enumerate}
\item
Let $f_{i}\in H_0^{3/2}(\Gamma_i)$, $i=1, \dots, N$. Then there
exists a function $u_0\in H_0^{2}(G)$ such that
$$
\supp u_0\subset\mathcal O_{\varkappa_1}(\mathcal K),
$$
$$
u_0|_{\Gamma_i}= f_{i}(y),\quad
    y\in \Gamma_i\cap\mathcal O_{\varkappa_2}(\mathcal K),\ i=1, \dots, N,
$$
\begin{equation}\label{eqHomogeneous0}
\mathbf B_{i}^1 u_0=\mathbf B_{i}^2 u_0=0,\quad i=1, \dots, N.
\end{equation}
\item
Let $f_{i}\in H_0^{3/2}(\Gamma_i)$ and $\supp f_i\subset\mathcal
O_{\varkappa_2}(\mathcal K)$, $i=1, \dots, N$. Then there exists a
function $u_0\in H_0^{2}(G)$ such that
$$
\supp u_0\subset\mathcal O_{\varkappa_2}(\mathcal K),
$$
$$
u_0|_{\Gamma_i}= f_{i}(y),\quad
    y\in \Gamma_i,\ i=1, \dots, N,
$$
and relations~\eqref{eqHomogeneous0} are valid.
\end{enumerate}
\end{lemma}
\begin{proof}
1. Using Lemma~\ref{lAppL8.1GurRJMP04} and a partition of unity,
one can construct a function $u_0\in H_0^{2}(G)$ such that
\begin{equation}\label{eqHomogeneous1}
\supp u_0\subset \mathcal O_{\varkappa_1}(\mathcal K),
\end{equation}
\begin{equation}\label{eqHomogeneous1'}
u_0|_{\Gamma_i}=f_{i}(y),\quad
  y\in \Gamma_i\cap\mathcal O_{\varkappa_2}(\mathcal K),\ i=1, \dots,
  N,
\end{equation}
$$
\mathbf B_{i}^1 u_0=0.
$$
By~\eqref{eqHomogeneous1} and~\eqref{eqSeparK23'}, we have
$\mathbf B_i^2u_0=0$. Therefore, $u_0$ is the desired function.

2. If $\supp f_i\subset\mathcal O_{\varkappa_2}(\mathcal K)$, we
can assume that $\supp u_0\subset\mathcal O_{\varkappa_2}(\mathcal
K)$. In this case, the equality in~\eqref{eqHomogeneous1'} holds
for $y\in\Gamma_i$.
\end{proof}

Now we find sufficient conditions for the violation of smoothness
of generalized solutions in the case of \textit{homogeneous}
nonlocal conditions. The following corollary results from
assertion~2 of Theorem~\ref{thSmoothf0}.

\begin{corollary}\label{corSmoothf0Homog}
Suppose that Condition~$\ref{condProperEigen}$ holds and
Condition~$\ref{condBv+BCConsist}$ fails. Let $\{f_0,f_i\}\in
L_2(G)\times\mathcal H_0^{3/2}(\partial G)$ be a function
constructed in assertion~$2$ of Theorem~$\ref{thSmoothf0}$, and
let there exist a function $u_0\in W^2(G)$ such that
\begin{equation}\label{eqSmoothf0Homog}
u_0|_{\Gamma_i}+\mathbf B_{i}^1 u_0+\mathbf B_{i}^2 u_0=
   f_{i}(y),\quad
    y\in \Gamma_i,\quad i=1, \dots, N.
\end{equation}
Then there is a right-hand side $\{f_0,0\}$, where $f_0\in
L_2(G)$, and a generalized solution $u\in W^1(G)$ of
problem~\eqref{eqPinG}, \eqref{eqBinG} such that $u\notin W^2(G)$.
\end{corollary}

We do not have an algorithm allowing one to construct a function
$u_0$ satisfying relations~\eqref{eqSmoothf0Homog} in the general
case of abstract operators $\mathbf B_i^2$. However, one can
guarantee the existence of $u_0$ in some particular cases which
are described in Corollaries~\ref{corSmoothf0B^2vInsideG}
and~\ref{corSmoothf0B^2vComp} below (see also
Sec.~\ref{subsecEx2}).

\begin{corollary}\label{corSmoothf0B^2vInsideG}
Suppose that the operators $\mathbf B_i^2$ satisfy the following
condition for some $\rho>0${\rm :}
\begin{equation}\label{eqSeparK23'''}
\|\mathbf B^2_{i}v\|_{W^{3/2}(\Gamma_i)}\le c
  \|v\|_{W^{2}(G_\rho)}\qquad\text{for all}\qquad v\in
  W^{2}(G_\rho).
\end{equation}
Let Condition~$\ref{condProperEigen}$ hold, and let
Condition~$\ref{condBv+BCConsist}$ fail. Then the conclusion of
Corollary~{\rm \ref{corSmoothf0Homog}} is true.
\end{corollary}
The proof of this corollary results from Corollary~{\rm
\ref{corSmoothf0Homog}}, from the embedding $H_0^2(G)\subset
W^2(G)$, and from the following lemma.
\begin{lemma}\label{lHomogeneousPartialG}
Let $f_{i}\in H_0^{3/2}(\Gamma_i)$, and let the operators $\mathbf
B_i^2$ satisfy condition~\eqref{eqSeparK23'''}. Then there exists
a function $u_0\in H_0^{2}(G)$ satisfying~\eqref{eqSmoothf0Homog}.
\end{lemma}
\begin{proof}
Using Lemma~\ref{lAppL8.1GurRJMP04} and a partition of unity, one
can construct a function $u_0\in H_0^{2}(G)$ such that
\begin{equation}\label{eqHomogeneousPartialG1}
\supp u_0\subset \overline{G}\setminus \overline{G_\rho},
\end{equation}
$$
u_0|_{\Gamma_i}=
   f_{i}(y),\quad
  y\in \Gamma_i;\ i=1, \dots, N.
$$
$$
\mathbf B_{i}^1 u_0=0.
$$
By~\eqref{eqHomogeneousPartialG1} and~\eqref{eqSeparK23'''}, we
have $\mathbf B_i^2u_0=0$. Therefore, $u_0$
satisfies~\eqref{eqSmoothf0Homog}.
\end{proof}

\begin{remark}
Condition~\eqref{eqSeparK23'''}, which is stronger than
Condition~\ref{condSeparK23}, means that the operators $\mathbf
B_{i}^2$ \textit{correspond to nonlocal terms supported inside the
domain $G$}.
\end{remark}

\begin{corollary}\label{corSmoothf0B^2vComp}
Let Condition~$\ref{condProperEigen}$ hold. Suppose that
Condition~$\ref{condBv+BCConsist}$ fails for an admissible
function $v$ such that
\begin{equation}\label{eqSmoothf0B^2vComp}
\supp(v|_{\Gamma_i}+\mathbf B_i^1v+\mathbf
B_i^2v)\subset\Gamma_i\cap\mathcal O_{\varkappa_2}(\mathcal K).
\end{equation}
Then the conclusion of Corollary~$\ref{corSmoothf0Homog}$ is true.
\end{corollary}
\begin{proof}
If $\supp(v|_{\Gamma_i}+\mathbf B_i^1v+\mathbf
B_i^2v)\subset\Gamma_i\cap\mathcal O_{\varkappa_2}(\mathcal K)$,
then the function $\{f_i\}\in\mathcal H_0^{3/2}(\partial G)$
constructed in the proof of assertion~2 of
Theorem~\ref{thSmoothf0} is also supported in $\mathcal
O_{\varkappa_2}(\mathcal K)$. Therefore, applying assertion~2 of
Lemma~\ref{lHomogeneous}, we obtain a function $u_0$
satisfying~\eqref{eqSmoothf0Homog}. Using
Corollary~\ref{corSmoothf0Homog}, we complete the proof.
\end{proof}

\section{Violation of Smoothness of Generalized Solutions}\label{sectImproperEigen}

It remains to study the case in which the following condition
holds.
\begin{condition}\label{condImproperEigen}
The band $-1\le\Im\lambda<0$ contains an improper eigenvalue of
the operator $\tilde{\mathcal L}(\lambda)$.
\end{condition}
In this section, we show that the smoothness of generalized
solutions can be violated for any operators $\mathbf B_i^2$ even
if nonlocal conditions~\eqref{eqBinG} are homogeneous.

\begin{theorem}\label{thNonSmoothu}
Let Condition~$\ref{condImproperEigen}$ hold. Then there exists a
right-hand side $\{f_0,0\}$, where $f_0\in L_2(G)$, and a
generalized solution $u\in W^1(G)$ of problem~\eqref{eqPinG},
\eqref{eqBinG} such that $u\notin W^2(G)$.
\end{theorem}
\begin{proof}
1. By assertion~2 of Lemma~\ref{lHomogeneous}, it suffices to find
a function $u\in W^1(G)$ such that $u\notin W^2(G)$ and
\begin{equation}\label{eqNonSmoothu1}
\begin{aligned}
\mathbf Pu\in L_2(G),\qquad u|_{\Gamma_i}+\mathbf B_{i}^1
u+\mathbf B_{i}^2 u\in H_0^{3/2}(\Gamma_i),\\
\supp(u|_{\Gamma_i}+\mathbf B_i^1u+\mathbf
B_i^2u)\subset\Gamma_i\cap\mathcal O_{\varkappa_2}(\mathcal K).
\end{aligned}
\end{equation}

Let $\lambda=\lambda_0$ be an improper eigenvalue of the operator
$\tilde{\mathcal L}(\lambda)$, $-1\le\Im\lambda_0<0$. Consider the
function
$$
 W=r^{i\lambda_0}\sum\limits_{l=0}^m\frac{1}{l!}(i\ln
 r)^l\varphi^{(m-l)}(\omega),
$$
where $\varphi^{(0)}, \dots, \varphi^{(\varkappa-1)}$ are an
eigenvector and associated vectors (a Jordan chain of length
$\varkappa\ge 1$) of the operator $\tilde{\mathcal L}(\lambda)$
corresponding to the eigenvalue $\lambda_0$. The number $m$ ($0\le
m\le\varkappa-1$) occurring in the definition of $W$ is such that
the function $W$ is not a polynomial vector in $y_1, y_2$. Such an
$m$ does exist because $\lambda_0$ is not a proper eigenvalue (if
$\Im\lambda\ne -1$ or $\Im\lambda=-1,\ \Re\lambda\ne0$, then we
can take $m=0$). It follows from Lemma~\ref{lAppL2.1GurPetr03}
that
\begin{equation}\label{eqNonSmoothu2}
 {\mathcal P}_jW_j=0,\qquad {\mathcal B}_{j\sigma}W|_{\gamma_{j\sigma}}=0.
\end{equation}

Consider a cut-off function $\xi\in C_0^\infty(\mathcal
O_{\varepsilon'}(0))$ equal to one near the origin, where
$\varepsilon'$ is given by~\eqref{eqEpsilon'}. Set $U=\xi W$.
Clearly, $\supp U\subset\mathcal O_{\varepsilon'}(0)$ and
$$
U\in \mathcal H_1^2(K)\cap\mathcal W^1(K).
$$
It follows from this relation, from~\eqref{eqNonSmoothu2}, from
Leibniz' formula, and from Lemma~\ref{lAppL3.3'Kondr} that
\begin{equation}\label{eqNonSmoothu3}
\{{\mathbf P}_jU_j\}\in\mathcal W^0(K^\varepsilon),\qquad
\{{\mathbf
B}_{j\sigma}U|_{\gamma_{j\sigma}^\varepsilon}\}\in\mathcal
H_0^{3/2}(\gamma^\varepsilon),
\end{equation}
while the relation $\supp U\subset\mathcal O_{\varepsilon'}(0)$
implies
\begin{equation}\label{eqNonSmoothu3'}
\supp{\mathbf
B}_{j\sigma}U|_{\gamma_{j\sigma}}\subset\gamma_{j\sigma}\cap\mathcal
O_{\varkappa_2}(0).
\end{equation}

 Moreover, we claim that
\begin{equation}\label{eqNonSmoothu4}
U\notin \mathcal W^2(K).
\end{equation}
Indeed, if $-1<\Im\lambda_0<0$, then one can directly verify the
validity of~\eqref{eqNonSmoothu4}; if $\Im\lambda_0=-1$,
then~\eqref{eqNonSmoothu4} follows from Lemma~\ref{lAppL4.20Kondr}
and from the fact that $W$ is not a polynomial vector.

2. Consider the function $u(y)$ given by $u(y)=U_j(y'(y))$ for
$y\in\mathcal O_{\varepsilon'}(g_j)$ and $u(y)=0$ for
$y\notin\mathcal O_{\varepsilon'}(\mathcal K)$, where $y'\mapsto
y(g_j)$ is the change of variables inverse to the change of
variables $y\mapsto y'(g_j)$ from Sec.~\ref{subsectStatement}. The
function $u$ is the desired one. Indeed, $u\notin W^2(G)$ due
to~\eqref{eqNonSmoothu4}. Furthermore, $\mathbf B_{i}^2u=0$ due to
inequality~\eqref{eqSeparK23'} because $\supp u\subset \mathcal
O_{\varkappa_1}(\mathcal K)$. It follows from the equality
$\mathbf B_{i}^2u=0$ and from relations~\eqref{eqNonSmoothu3}
and~\eqref{eqNonSmoothu3'} that the function $u$
satisfies~\eqref{eqNonSmoothu1}.
\end{proof}

\section{The Case of Several Orbits}\label{sectGener}

\subsection{Model Problems and Preservation of Smoothness}
In this section, we generalize the results of
Secs.~\ref{sectStatement}--\ref{sectImproperEigen} to the case
where the set $\mathcal K$ consists of more than one orbit. Let
$$
\mathcal K=\bigcup\limits_{t=1}^T \mathcal K_t,
$$
where $\mathcal K_1,\dots,\mathcal K_T$ are disjoint orbits
forming the set $\mathcal K$ of conjugation points. Let the orbit
$\mathcal K_t$ consists of points $g_{t,1},\dots,g_{t,N_t}$.

Take a sufficiently small number $\varepsilon$ such that there
exist neighborhoods $\mathcal O_{\varepsilon_1}(g_{t,j})$, $
\mathcal O_{\varepsilon_1}(g_{t,j})\supset\mathcal
O_{\varepsilon}(g_{t,j}) $, satisfying the following conditions:
\begin{enumerate}
\item The domain $G$ is a plane angle in the neighborhood $\mathcal O_{\varepsilon_1}(g_{t,j})$;
\item
$\overline{\mathcal
O_{\varepsilon_1}(g_{t,j})}\cap\overline{\mathcal
O_{\varepsilon_1}(g_{\tau,k})}=\varnothing$ for any
$g_{t,j},g_{\tau,k}\in\mathcal K$, $(t,j)\ne (\tau,k)$;
\item If $g_{t,j}\in\overline{\Gamma_i}$ and
$\Omega_{is}(g_{t,j})=g_{t,k},$ then ${\mathcal
O}_{\varepsilon}(g_{t,j})\subset\mathcal
 O_i$,
 $\Omega_{is}\big({\mathcal
O}_{\varepsilon}(g_{t,j})\big)\subset{\mathcal
O}_{\varepsilon_1}(g_{t,k}).$
\end{enumerate}

For each point $g_{t,j}\in\overline{\Gamma_i}\cap\mathcal K$, we
fix a transformation $Y_{t,j}: y\mapsto y'(g_{t,j})$ which is a
composition of the shift by the vector
$-\overrightarrow{Og_{t,j}}$ and the rotation through some angle
so that
$$
Y_{t,j}({\mathcal O}_{\varepsilon_1}(g_{t,j}))={\mathcal
O}_{\varepsilon_1}(0),\qquad Y_{t,j}(G\cap{\mathcal
O}_{\varepsilon_1}(g_{t,j}))=K_{t,j}\cap{\mathcal
O}_{\varepsilon_1}(0),
$$
$$
Y_{t,j}(\Gamma_i\cap{\mathcal
O}_{\varepsilon_1}(g_{t,j}))=\gamma_{t,j\sigma}\cap{\mathcal
O}_{\varepsilon_1}(0)\quad (\sigma=1\ \text{or}\ \sigma=2),
$$
where
\begin{gather*}
K_{t,j}=\{y\in{\mathbb R}^2:\ r>0,\ |\omega|<\omega_{t,j}\},\\
\gamma_{t,j\sigma}=\{y\in\mathbb R^2:\ r>0,\ \omega=(-1)^\sigma
\omega_{t,j}\}.
\end{gather*}
Here $(\omega,r)$ are the polar coordinates and
$0<\omega_{t,j}<\pi$.

Consider the following condition instead of
Condition~\ref{condK1}.
\begin{condition_a}
Let $g_{t,j}\in\overline{\Gamma_i}\cap\mathcal K$ and
$\Omega_{is}(g_{t,j})=g_{t,k}\in\mathcal K_t;$ then the
transformation
$$
Y_{t,k}\circ\Omega_{is}\circ Y_{t,j}^{-1}:{\mathcal
O}_{\varepsilon}(0)\to{\mathcal O}_{\varepsilon_1}(0)
$$
is the composition of rotation and homothety.
\end{condition_a}

We assume throughout this section that Conidtions~\ref{condK1}$'$
and~\ref{condSeparK23} are fulfilled.

Let $y\mapsto y'(g_{t,j})$ be the above change of variables. Set
$$
K_{t,j}^\varepsilon=K_{t,j}\cap\mathcal O_\varepsilon(0),\qquad
\gamma_{t,j\sigma}^\varepsilon=\gamma_{t,j\sigma}\cap\mathcal
O_\varepsilon(0)
$$
and introduce the functions
\begin{equation}\label{eqytoy'Gener}
\begin{gathered}
U_{t,j}(y')=u(y(y')),\quad F_{t,j}(y')=f_0(y(y')),\quad y'\in
K_{t,j}^\varepsilon,\\
F_{t,j\sigma}(y')=f_i(y(y')),\quad B_{t,j\sigma}^u(y')=(\mathbf
B_i^2u)(y(y')),\\
\Psi_{j\sigma}(y')=F_{t,j\sigma}(y')-B_{t,j\sigma}^u(y'),\quad
y'\in\gamma_{t,j\sigma}^\varepsilon,
\end{gathered}
\end{equation}
where $\sigma=1$ $(\sigma=2)$ if the transformation $y\mapsto
y'(g_{t,j})$ takes $\Gamma_i$ to the side $\gamma_{t,j1}$
($\gamma_{t,j2}$) of the angle $K_{t,j}$. Similarly
to~\eqref{eqPinK}, \eqref{eqBinK}, using
Condition~\ref{condK1}$'$, we obtain the following model nonlocal
problem for each $t=1,\dots,T$:
\begin{equation}\label{eqPinKGener}
  {\bf P}_{t,j}U_{t,j}=F_{t,j}(y) \quad (y\in
  K_{t,j}^\varepsilon,\ j=1,\dots N_t),
\end{equation}
\begin{equation}\label{eqBinKGener}
\begin{gathered}
  {\mathbf
  B}_{t,j\sigma}U_t\equiv\sum\limits_{k=1}^{N_t}\sum\limits_{s=1}^{S_{t,j\sigma
  k}}  b_{t,j\sigma ks}(y)U_k({\mathcal G}_{t,j\sigma ks}y)
=\Psi_{t,j\sigma}(y)\\ (y\in\gamma_{t,j\sigma}^\varepsilon,\
j=1,\dots N_t,\ \sigma=1,2).
\end{gathered}
\end{equation}
Here ${\mathbf P}_{t,j}$ are properly elliptic second-order
differential operators with variable complex-valued
$C^\infty$-coefficients, $U_t=(U_{t,1},\dots,U_{t,N_t})$,
$b_{t,j\sigma ks}(y)$ are smooth functions, $b_{t,j\sigma
j0}(y)\equiv 1$; ${\mathcal G}_{t,j\sigma ks}$ is an operator of
rotation through an angle~$\omega_{t,j\sigma ks}$ and homothety
with a coefficient~$\chi_{t,j\sigma ks}>0$ in the $y$-plane.
Moreover,
$$
|(-1)^\sigma \omega_{t,j}+\omega_{t,j\sigma
ks}|<\omega_{t,k}\qquad\text{for}\qquad (k,s)\ne(j,0)
$$
and
$$
\omega_{t,j\sigma j0}=0,\qquad \chi_{t,j\sigma j0}=1
$$
(i.e., $\mathcal G_{t,j\sigma j0}y\equiv y$).

Let the principal homogeneous parts of the operators $\mathbf
P_{t,j}$ at the point $y=0$ have the following form in the polar
coordinates:
$$
r^{-2}\tilde{\mathcal
P}_{t,j}(\omega,\partial/\partial\omega,r\partial/\partial r)v.
$$
Consider the analytic operator-valued function $\tilde{\mathcal
L}_t(\lambda):\prod\limits_{j=1}^{N_t}
W^{2}(-\omega_{t,j},\omega_{t,j})\to\prod\limits_{j=1}^{N_t}
(L_2(-\omega_{t,j}, \omega_{t,j})\times\mathbb C^2)$ given by
\begin{multline*}
\tilde{\mathcal L}_t(\lambda)\varphi=\big\{\tilde{\mathcal
P}_{t,j}(\omega, \partial/\partial\omega, i\lambda)\varphi_j,\\
  \sum\limits_{k,s} (\chi_{t,j\sigma ks})^{i\lambda}
              b_{t,j\sigma ks}(0)\varphi_k((-1)^\sigma \omega_{t,j}+\omega_{t,j\sigma ks})\big\}.
\end{multline*}

First, we study the case in which the following condition holds.
\begin{condition}\label{condNoEigenGener}
The band $-1\le\Im\lambda<0$ contains no eigenvalues of the
operators $\tilde{\mathcal L}_t(\lambda)$, $t=1,\dots,T$.
\end{condition}

The following result can be proved similarly to
Theorem~\ref{thuinW^2NoEigen}.

\begin{theorem}\label{thuinW^2NoEigenGener}
Let Condition~$\ref{condNoEigenGener}$ hold, and let $u\in W^1(G)$
be a generalized solution of problem~\eqref{eqPinG},
\eqref{eqBinG} with right-hand side $\{f_0,f_i\}\in L_2(G)\times
\mathcal W^{3/2}(\partial G)$. Then $u\in W^2(G)$.
\end{theorem}

\subsection{Border Case and Violation of Smoothness}

Now we assume that the border case occurs for some of the orbits.
Let the following condition hold.

\begin{condition}\label{condProperEigenGener}
The band $-1\le\Im\lambda<0$ contains only the eigenvalue
$\lambda=-i$ of the operators $\tilde{\mathcal L}_t(\lambda)$,
$t=1,\dots,t_1$, $t_1\le T$, and this eigenvalue is a proper one.
If $t_1<T$, then the operators $\tilde{\mathcal L}_t(\lambda)$,
$t=t_1+1,\dots,T$, have no eigenvalues in the band
$-1\le\Im\lambda<0$.
\end{condition}

Analogously to Sec.~\ref{subsectuFixed}, we will introduce the
notion of the consistency condition for each orbit $\mathcal K_t$,
$t=1,\dots,t_1$. For each $t=1,\dots,t_1$, we denote by
\begin{equation}\label{eqSystemBGener}
 \{\hat{\mathcal B}_{t,j\sigma}(D_y)\},\ j=1,\dots,N_t,\
 \sigma=1,2,
\end{equation}
the system of operators~\eqref{eqSystemB} corresponding to the
orbit $\mathcal K_t$. It has been proved in
Sec.~\ref{subsectuFixed} that this system is linearly dependent.
Let
\begin{equation}\label{eqSystemB'Gener}
\{\hat{\mathcal B}_{t,j'\sigma'}(D_y)\}
\end{equation}
be a maximal linearly independent subsystem of
system~\eqref{eqSystemBGener}. In this case, any operator
$\hat{\mathcal B}_{t,j\sigma}(D_y)$ which does not enter
system~\eqref{eqSystemB'Gener} can be represented as follows:
\begin{equation}\label{eqBviaB'Gener}
\hat{\mathcal
B}_{t,j\sigma}(D_y)=\sum\limits_{j',\sigma'}\beta_{t,j\sigma}^{j'\sigma'}\hat{\mathcal
B}_{t,j'\sigma'}(D_y),
\end{equation}
where $\beta_{t,j\sigma}^{j'\sigma'}$ are some constants.

To introduce the notion of the consistency condition, we consider
arbitrary functions $\{Z_{j\sigma}\}\in\mathcal
W^{3/2}(\gamma_t^\varepsilon)$, each of which is defined on its
own interval $\gamma_{t,j\sigma}^\varepsilon$. Consider the
functions
$$
Z^0_{j\sigma}(r)=Z_{j\sigma}(y)|_{y=(r\cos\omega_{t,j},\,
r(-1)^\sigma\sin\omega_{t,j})}.
$$
Each of the functions $Z^0_{j\sigma}$ belongs to
$W^{3/2}(0,\varepsilon)$.

\begin{definition}
Let $\beta_{t,j\sigma}^{j'\sigma'}$ be the constants occurring
in~\eqref{eqBviaB'Gener}. If the relations
\begin{equation}\label{eqConsistencyZGener}
\int\limits_{0}^\varepsilon
r^{-1}\Bigg|\frac{d}{dr}\bigg(Z^0_{j\sigma}-\sum\limits_{j',\sigma'}\beta_{t,j\sigma}^{j'\sigma'}Z^0_{j'\sigma'}\bigg)\Bigg|^2dr<\infty
\end{equation}
hold for all indices $j,\sigma$ corresponding to the operators of
system~\eqref{eqSystemBGener} which do not enter
system~\eqref{eqSystemB'Gener}, then we say that the {\em
functions $Z_{j\sigma}$ satisfy the consistency
condition~\eqref{eqConsistencyZGener}}.
\end{definition}

Denote by $\mathcal S^{3/2}(\partial G)$ the set of functions
$\{f_i\}\in\mathcal W^{3/2}(\partial G)$ such that the functions
$F_{t,j\sigma}$ (see~\eqref{eqytoy'Gener}) satisfy the consistency
condition~\eqref{eqConsistencyZGener} for each $t=1,\dots,t_1$.

The following result can be proved similarly to
Theorem~\ref{thUNonSmFNonConsist}

\begin{theorem}\label{thUNonSmFNonConsistGener}
Let Condition~$\ref{condProperEigenGener}$ hold. Then there exist
a function $\{f_0,f_i\}\in L_2(G)\times \mathcal W^{3/2}(\partial
G)$, $\{f_i\}\notin\mathcal S^{3/2}(\partial G)$, and a function
$u\in W^1(G)$ such that $u$ is a generalized solution of
problem~\eqref{eqPinG}, \eqref{eqBinG} with the right-hand side
$\{f_0,f_i\}$ and $u\notin W^2(G)$.
\end{theorem}

Now we assume that $\{f_i\}\in\mathcal S^{3/2}(\partial G)$ and
prove that the following condition is necessary and sufficient for
any generalized solution to be smooth.
\begin{condition}\label{condB2vB1CConsistencyGener}
\begin{enumerate}
\item For any $v\in
W^{2}(G\setminus\overline{\mathcal O_{\varkappa_1}(\mathcal K)})$,
the functions $B^v_{t,j\sigma}$ satisfy the consistency
condition~\eqref{eqConsistencyZGener}, where $t=1,\dots,t_1$.
\item For any vector $C_t=(C_{t,1},\dots,C_{t,N_t})$ with constant elements,
the functions $\mathbf
B_{t,j\sigma}C_t|_{\gamma_{t,j\sigma}^\varepsilon}$ satisfy the
consistency condition~\eqref{eqConsistencyZGener}, where
$t=1,\dots,t_1$.
\end{enumerate}
\end{condition}

Set $\varepsilon'=d'_\chi\min(\varepsilon,\varkappa_2),$ where
$d_\chi'=\min\{\chi_{t,j\sigma ks}\}/2$.

\begin{theorem}\label{thSmoothfne0Gener}
Let Condition~$\ref{condProperEigenGener}$ hold. Then the
following assertions are true.
\begin{enumerate}
\item
If Condition~$\ref{condB2vB1CConsistencyGener}$ is fulfilled and
$u\in W^1(G)$ is a generalized solution of problem~\eqref{eqPinG},
\eqref{eqBinG} with right-hand side $\{f_0,f_i\}\in L_2(G)\times
\mathcal S^{3/2}(\partial G)$, then $u\in W^2(G)$.
\item
If Condition~$\ref{condB2vB1CConsistencyGener}$ fails, then there
exists a right-hand side $\{f_0,f_i\}\in L_2(G)\times \mathcal
S^{3/2}(\partial G)$ and a generalized solution $u\in W^1(G)$ of
problem~\eqref{eqPinG}, \eqref{eqBinG}  such that $u\notin
W^2(G)$.
\end{enumerate}
\end{theorem}
\begin{proof} The proof of this theorem is similar to the proof
of Theorem~\ref{thSmoothfne0}. For instance, we prove assertion~2.
Let Condition~\ref{condB2vB1CConsistencyGener} be violated, e.g.,
for the orbit $\mathcal K_1$. In this case, there exist a function
$v\in W^{2}(G\setminus\overline{\mathcal O_{\varkappa_1}(\mathcal
K)})$ and a constant vector $C_1=(C_{1,1},\dots,C_{1,N_1})$ such
that the functions $B^v_{1,j\sigma}+\mathbf B_{1,j\sigma}C_1$ do
not satisfy the consistency condition~\eqref{eqConsistencyZGener}
for $t=1$ (one can assume that either $v=0,\,C_1\ne0$ or
$v\ne0,\,C_1=0$). Extend the function $v$ to the domain $G$ in
such a way that $v(y)=0$ for $y\in\mathcal
O_{\varkappa_1/2}(\mathcal K)$ and $v\in W^2(G)$.

Consider functions $F_{t,j\sigma}'\in
C^\infty(\overline{\gamma_{t,j\sigma}})$ such that
$$
F_{t,j\sigma}'(y)=B^v_{t,j\sigma}(0)+(\mathbf
B_{t,j\sigma}C_t)(0),\quad |y|<\varepsilon/2,\qquad
F_{t,j\sigma}'(y)=0,\quad |y|>\varepsilon,
$$
where $t=1,\dots,T$, $C_1$ is the above vector, and
$C_2,\dots,C_T$ are arbitrary (but fixed) constant vectors.

Since $F_{t,j\sigma}$ are constant near the origin, it follows
that they satisfy the consistency
condition~\eqref{eqConsistencyZGener} for each $t=1,\dots,t_1$. By
construction,
$$
\{F_{t,j\sigma}'-B^v_{t,j\sigma}-\mathbf
B_{t,j\sigma}C_t|_{\gamma_{t,j\sigma}^\varepsilon}\}\in\mathcal
W^{3/2}(\gamma^\varepsilon),\quad
(F_{t,j\sigma}'-B^v_{t,j\sigma}-\mathbf
B_{t,j\sigma}C_t)|_{y=0}=0,
$$
where $t=1,\dots,T$. Moreover, the functions
$F_{1,j\sigma}'-B^v_{1,j\sigma}-\mathbf B_{1,j\sigma}C_1$ do not
satisfy the consistency condition~\eqref{eqConsistencyZGener} for
$t=1$.

By Lemma~\ref{lUnonSmoothZnonConsist}, there exists a function
$U_1'\in\mathcal H_a^2(K_1)\cap\mathcal W^1(K_1)$ such that $\supp
U_1'\subset\mathcal O_{\varepsilon'}(0)$, $U_1'\notin\mathcal
W^2(K_1^\varepsilon)$, and
\begin{equation}\label{eqPU'-PsiinH0Gener}
\{\mathbf P_{1,j}U'_{1,j}\}\in\mathcal W^0(K_1^\varepsilon),
\end{equation}
$$
 \big\{\big(\mathbf
B_{1,j\sigma}U_1'-(F_{1,j\sigma}'-B^v_{1,j\sigma}-\mathbf
B_{1,j\sigma}C_1)\big)|_{\gamma_{1,j\sigma}^\varepsilon}\big\}\in\mathcal
H_0^{3/2}(\gamma_1^\varepsilon).
$$
One can also write the latter relation as follows:
\begin{equation}\label{eqBU'-PsiinH0Gener}
\{\mathbf
B_{1,j\sigma}(U_1'+C_1)|_{\gamma_{1,j\sigma}^\varepsilon}+B^v_{1,j\sigma}-F_{1,j\sigma}'\}\in\mathcal
H_0^{3/2}(\gamma_1^\varepsilon).
\end{equation}

Let $t=2,\dots,t_1$. It follows from
Lemma~\ref{lUnonSmoothZnonConsist} (if the functions
\begin{equation}\label{eqBU'-PsiinH0Gener'}
F_{t,j\sigma}'-B^v_{t,j\sigma}-\mathbf B_{t,j\sigma}C_t
\end{equation}
do not satisfy the consistency
condition~\eqref{eqConsistencyZGener}) or from
Lemma~\ref{lAppL3.3GurRJMP03} (if the
functions~\eqref{eqBU'-PsiinH0Gener'} satisfy the consistency
condition~\eqref{eqConsistencyZGener}) that there exists a
function $U_t'\in\mathcal H_a^2(K_t)\cap\mathcal W^1(K_t)$ such
that $\supp U_t'\subset\mathcal O_{\varepsilon'}(0)$ and
\begin{equation}\label{eqPU'-PsiinH0Gener''}
\{\mathbf P_{t,j}U'_{t,j}\}\in\mathcal W^0(K_t^\varepsilon),
\end{equation}
\begin{equation}\label{eqBU'-PsiinH0Gener'''}
\{\mathbf
B_{t,j\sigma}(U_t'+C_t)|_{\gamma_{t,j\sigma}^\varepsilon}+B^v_{t,j\sigma}-F_{t,j\sigma}'\}\in\mathcal
H_0^{3/2}(\gamma_t^\varepsilon).
\end{equation}
(If Lemma~\ref{lAppL3.3GurRJMP03} has been applied, then
$U_t'\in\mathcal W^2(K_t^\varepsilon)$.)

Finally, let $t=t_1+1,\dots, T$. In this case, by
Lemma~\ref{lAppL2.4GurRJMP03}, there exists a function
$U_t'\in\mathcal H_a^2(K_t)\cap\mathcal W^2(K_t)$ such that $\supp
U_t'\subset\mathcal O_{\varepsilon'}(0)$ and
relations~\eqref{eqPU'-PsiinH0Gener''}
and~\eqref{eqBU'-PsiinH0Gener'''} hold.

Introduce a function $u'(y)$ such that
$u'(y)=U'_{t,j}(y'(y))+\xi_{t,j}(y)C_{t,j}$ for $y\in\mathcal
O_{\varepsilon'}(g_{t,j})$ and $u'(y)=0$ for $y\notin\mathcal
O_{\varepsilon'}(\mathcal K)$, where $y'\mapsto y(g_{t,j})$ is the
change of variables inverse to the change of variables $y\mapsto
y'(g_{t,j})$, while $\xi_{t,j}\in
C_0^\infty(O_{\varepsilon'}(g_{t,j}))$, $\xi_{t,j}(y)=1$ for
$y\in\mathcal O_{\varepsilon'/2}(g_{t,j})$. Similarly to the proof
of assertion~2 in Theorem~\ref{thSmoothfne0}, using
relations~\eqref{eqPU'-PsiinH0Gener}, \eqref{eqBU'-PsiinH0Gener},
\eqref{eqPU'-PsiinH0Gener''}, and \eqref{eqBU'-PsiinH0Gener'''},
one can verify that the function $u=u'+v$ is the desired one.
\end{proof}

Now we consider problem~\eqref{eqPinG}, \eqref{eqBinG} with
regular and homogeneous nonlocal conditions.

\begin{definition}\label{defAdmitGener}
We say that a function $v\in W^{2}(G\setminus\overline{\mathcal
O_{\varkappa_1}(\mathcal K)})$ is {\em admissible} if there exist
constant vectors $C_t=(C_{t,1},\dots,C_{t,N_t})$, $t=1,\dots,T$,
such that
\begin{equation}\label{eqvadmissibleGener}
B_{t,j\sigma}^v(0)+(\mathbf B_{t,j\sigma}C_t)(0)=0,\quad
j=1,\dots,N,\ \sigma=1,2,\ t=1,\dots T.
\end{equation}
Vectors $C_t$, $t=1,\dots T$, satisfying
relations~\eqref{eqvadmissibleGener} are said to be {\em
admissible vectors corresponding to the function~$v$.}
\end{definition}

\begin{definition}
Right-hand sides $f_i$ in nonlocal conditions~\eqref{eqBinG} are
said to be \textit{regular} if $\{f_i\}\in\mathcal
S^{3/2}(\partial G)$ and $f_i|_{\overline{\Gamma_i}\cap\mathcal
K_t}=0$, $t=1,\dots T$ (i.e.,
$f_i|_{\overline{\Gamma_i}\cap\mathcal K}=0$).
\end{definition}

We prove that the following condition is necessary and sufficient
for any generalized solution of problem~\eqref{eqPinG},
\eqref{eqBinG} with regular $f_i$ to be smooth.

\begin{condition}\label{condBv+BCConsistGener}
For each admissible function $v$ and for each admissible vector
$C_t$, $t=1,\dots,t_1$, corresponding to $v$, the functions
$B_{t,j\sigma}^v+\mathbf B_{t,j\sigma}C_t$ satisfy the consistency
condition~\eqref{eqConsistencyZGener}.
\end{condition}

\begin{theorem}\label{thSmoothf0Gener}
Let Condition~$\ref{condProperEigenGener}$ hold. Then the
following assertions are true.
\begin{enumerate}
\item
If Condition~$\ref{condBv+BCConsistGener}$ is fulfilled and $u\in
W^1(G)$ is a generalized solution of problem~\eqref{eqPinG},
\eqref{eqBinG} with right-hand side $\{f_0,f_i\}\in
L_2(G)\times\mathcal S^{3/2}(\partial G)$, where $f_i$ are
regular, then $u\in W^2(G)$.
\item
If Condition~$\ref{condBv+BCConsistGener}$ fails, then there
exists a right-hand side $\{f_0,f_i\}\in L_2(G)\times\mathcal
H_0^{3/2}(\partial G)$ and a generalized solution $u\in W^1(G)$ of
problem~\eqref{eqPinG}, \eqref{eqBinG} such that $u\notin W^2(G)$.
\end{enumerate}
\end{theorem}
\begin{proof}
The proof of the theorem is similar to the proof of
Theorem~\ref{thSmoothf0}. For instance, let us prove assertion~2.
If Condition~\ref{condBv+BCConsistGener} fails, there exists a
function $v\in W^{2}(G\setminus\overline{\mathcal
O_{\varkappa_1}(\mathcal K)})$ and constant vectors
$C_t=(C_{t,1},\dots,C_{t,N_t})$, $t=1,\dots,T$, such that
$B^v_{t,j\sigma}(0)+(\mathbf B_{t,j\sigma}C_t)(0)=0$ and, e.g.,
the functions $B^v_{1,j\sigma}+\mathbf B_{1,j\sigma}C_1$ do not
satisfy the consistency condition~\eqref{eqConsistencyZGener}.

We must find a function  $u\in W^1(G)$ such that $u\notin W^2(G)$
and
$$
\mathbf Pu\in L_2(G),\qquad u|_{\Gamma_i}+\mathbf B_{i}^1
u+\mathbf B_{i}^2 u\in H_0^{3/2}(\Gamma_i).
$$
To do this, one can repeat the proof of assertion~2 in
Theorem~\ref{thSmoothfne0Gener}, assuming that $v$ is the above
function, $C_t$, $t=1,\dots,T$, are the above vectors, and
$F_{t,j\sigma}'(y)\equiv0$, $t=1,\dots,T$ (which is possible due
to the relations $B^v_{t,j\sigma}(0)+(\mathbf
B_{t,j\sigma}C_t)(0)=0$).
\end{proof}

\begin{remark}\label{remCorGener}
It is easy to see that
Corollaries~\ref{corSmoothf0NeighbK}--\ref{corSmoothf0B^2vComp}
(in which Conditions~\ref{condProperEigen}
and~\ref{condBv+BCConsist} and Theorem~\ref{thSmoothf0} must be
replaced by Conditions~\ref{condProperEigenGener}
and~\ref{condBv+BCConsistGener} and Theorem~\ref{thSmoothf0Gener},
respectively) are true in the case of several orbits.
\end{remark}

It remains to study the case in which the following condition
holds.
\begin{condition}\label{condImproperEigenGener}
There is a number $t\in\{1,\dots,T\}$ such that the band
$-1\le\Im\lambda<0$ contains an improper eigenvalue of the
operator $\tilde{\mathcal L}_t(\lambda)$.
\end{condition}

The proof of the following result is similar to the proof of
Theorem~\ref{thNonSmoothu}.

\begin{theorem}\label{thNonSmoothuGener}
Let Condition~$\ref{condImproperEigenGener}$ hold. Then there
exists a right-hand side $\{f_0,0\}$, where $f_0\in L_2(G)$, and a
generalized solution $u\in W^1(G)$ of problem~\eqref{eqPinG},
\eqref{eqBinG} such that $u\notin W^2(G)$.
\end{theorem}

\section{Example}\label{sectExample}

\subsection{Problem with Nonhomogeneous Nonlocal Conditions}

In this section, we apply the results of
Secs.~\ref{sectStatement}--\ref{sectGener} to the study of
smoothness of generalized solutions for
problem~\eqref{eqIntroPinG}, \eqref{eqIntroBinG}. We recall the
setting of this problem.

Let $\partial G\setminus{\mathcal K}=\Gamma_1\cup\Gamma_2$, where
$\Gamma_i$ are open (in the topology of $\partial G$)
$C^\infty$-curves and $\mathcal
K=\overline{\Gamma_1}\cap\overline{\Gamma_2}=\{g, h\}$, where
$g,h$ are the end points of the curves $\overline{\Gamma_1}$ and
$\overline{\Gamma_2}$. Suppose that the domain $G$ is the plane
angle of opening $\pi$ in a neighborhood of each of the points $g,
h$. Thus, the boundary of $G$ is infinitely smooth. We consider
the following nonlocal problem in $G$:
\begin{equation}\label{eqExPinG}
 \Delta u=f_0(y)\quad (y\in G),
\end{equation}
\begin{equation}\label{eqExBinG}
\begin{aligned}
&u|_{\Gamma_1}+b_1(y)
u\bigl(\Omega_{1}(y)\bigr)\big|_{\Gamma_1}+a(y)
u\bigl(\Omega(y)\bigr)\big|_{\Gamma_1}=f_1(y)\quad
(y\in\Gamma_1),\\
&u|_{\Gamma_2}+b_2(y)
u\bigl(\Omega_{2}(y)\bigr)\big|_{\Gamma_2}=f_2(y)\quad
 (y\in\Gamma_2).
\end{aligned}
\end{equation}
Here $b_1$, $b_2$, and $a$ are real-valued $C^\infty$-functions;
$\Omega_i$ and $\Omega$ are $C^\infty$-diffeo\-mor\-phisms
described in Introduction (see Fig.~\ref{figEx1}).

\smallskip

Let us show that nonlocal conditions~\eqref{eqExBinG} can be
represented in the form~\eqref{eqBinG}. To do this, we take a
small ${\varepsilon}$ such that the sets $\overline{\mathcal
O_{{\varepsilon}}(g)}$ and $\overline{\mathcal
O_{{\varepsilon}}(h)}$ do not intersect with the curve
$\overline{\Omega(\Gamma_1)}$.

Consider a function $\zeta\in C_0^\infty(\mathbb R^2)$ such that
$\zeta(y)=1$ for $y\in\mathcal O_{\varepsilon/2}(\mathcal K)$ and
$\supp\zeta\subset\mathcal O_{\varepsilon}(\mathcal K)$. Introduce
the operators
\begin{gather*}
\mathbf B_{i}^1u=\zeta(y)b_{i}(y)
u(\Omega_{i}(y))|_{\Gamma_i},\\
\mathbf B_{1}^2u=(1-\zeta(y))b_{1}(y)u(\Omega_{1}(y))|_{\Gamma_1}+
a(y)u(\Omega(y))|_{\Gamma_1},\\ \mathbf
B_{2}^2u=(1-\zeta(y))b_{2}(y)u(\Omega_{2}(y))|_{\Gamma_2}.
\end{gather*}
In this example, the set $\mathcal K$ is formed by two orbits; the
first orbit consists of the point $g$ and the second orbit of the
point $h$. Since the support of $\zeta$ is contained in a
neighborhood of the set $\mathcal K$, one can assume that the
transformations $\Omega_{i}$ occurring in the definition of the
operators $\mathbf B_{i}^1$ are also defined in a neighborhood of
the set $\mathcal K$ and satisfy Condition~\ref{condK1}$'$.
Furthermore, due to the arguments of~\cite[Sec.~1.2]{GurRJMP03},
the operators $\mathbf B_{i}^2$ satisfy
Condition~\ref{condSeparK23} with $\varkappa_1=\varepsilon/2$ and
some $\varkappa_2<\varkappa_1$ and $\rho$. Therefore, nonlocal
conditions~\eqref{eqExBinG} can be represented in the
form~\eqref{eqBinG}.

\smallskip

 Write a model problem corresponding to the point $g$
(one can similarly write a model problem corresponding to the
point $h$). To be definite, we assume that the point $g$ coincides
with the origin, $g=0$, while the axis $Oy_1$ is directed inside
the domain $G$, perpendicularly to the boundary. Consider the sets
\begin{gather*}
K^\varepsilon=\{y\in{\mathbb R}^2:\ 0<r<\varepsilon,\,
|\omega|<\pi/2\},\\ \gamma_\sigma^\varepsilon=\{y\in{\mathbb
R}^2:\ 0<r<\varepsilon,\, \omega=(-1)^\sigma\pi/2\}.
\end{gather*}
Take a small number $\varepsilon$ such that $\mathcal
O_{\varepsilon}(0)\cap G=K^\varepsilon$. The model problem
acquires the form
\begin{align}
 \Delta U&=F(y)\quad (y\in K^\varepsilon),\label{eqExPinK}\\
 U(y)+b_\sigma(y) U(\mathcal G_\sigma y)&=\Psi_\sigma(y)\quad
 (y\in\gamma_\sigma^\varepsilon,\ \sigma=1, 2).\label{eqExBinK}
\end{align}
Here $
 \mathcal G_\sigma=
 \begin{pmatrix}
  0 & (-1)^\sigma\\
  (-1)^{\sigma+1} & 0
 \end{pmatrix}
$ is the operator of rotation through the angle
$(-1)^{\sigma+1}\pi/2$,
$$
F(y)=f_0(y),\quad y\in K^\varepsilon,\qquad
\Psi_\sigma(y)=f_\sigma(y)-B_\sigma^u(y),\quad
y\in\gamma_\sigma^\varepsilon;
$$
moreover,
$$
B_1^u(y)=a(y)u\big(\Omega(y)\big),\quad
y\in\gamma_1^{\varepsilon/2},\qquad B_2^u(y)=0,\quad
y\in\gamma_2^{\varepsilon/2},
$$
because $(1-\zeta(y))b_{\sigma}(y)u(\Omega_{\sigma}(y))=0$ for
$y\in\gamma_\sigma^{\varepsilon/2}$, $\sigma=1,2$.

The eigenvalue problem has the form
\begin{gather}
\varphi''(\omega)-\lambda^2\varphi(\omega)=0\quad
(|\omega|<\pi/2),\label{eqExPEigen}\\
\varphi(-\pi/2)+b_1(0)\varphi(0)=0,\quad
\varphi(\pi/2)+b_2(0)\varphi(0)=0.\label{eqExBEigen}
\end{gather}

Set $I_1=(-\infty,-2]\cup(0,\infty)$ and $I_2=(-2,0)$. Simple
calculations~\cite[\S~9]{GurRJMP04} show that the eigenvalues of
problem~\eqref{eqExPEigen}, \eqref{eqExBEigen} are distributed in
the band $-1\le\Im\lambda<0$ as follows.
\begin{description}
\item[Case 1 ($b_1(0)+b_2(0)\in
I_1$)] the band $-1\le\Im\lambda<0$ contains no eigenvalues.
\item[Case 2 ($b_1(0)+b_2(0)=0$)] the band $-1\le\Im\lambda<0$
contains the unique eigenvalue $\lambda=-i$, and this eigenvalue
is proper.
\item[Case 3 ($b_1(0)+b_2(0)\in I_2$)] the band $-1\le\Im\lambda<0$
contains the improper eigenvalue
$\lambda=2\pi^{-1}i\arctan\big(\sqrt{4-(b_1(0)+b_2(0))^2}/(b_1(0)+b_2(0))\big)$.
\end{description}

Consider Case 1.

\begin{theorem}\label{thExuinW^2NoEigen}
Let $b_1(0)+b_2(0)\in I_1$ and $b_1(h)+b_2(h)\in I_1$. Let $u\in
W^1(G)$ be a generalized solution of problem~\eqref{eqExPinG},
\eqref{eqExBinG} with right-hand side $\{f_0,f_i\}\in L_2(G)\times
\mathcal W^{3/2}(\partial G)$. Then $u\in W^2(G)$.
\end{theorem}
\begin{proof}
In the case under consideration, the band $-1\le\Im\lambda<0$
contains no eigenvalues of problem~\eqref{eqExPEigen},
\eqref{eqExBEigen} (and no eigenvalues of the analogous problem
corresponding to the point $h$). Therefore, this theorem follows
from Theorem~\ref{thuinW^2NoEigenGener}.
\end{proof}
Note that we impose no consistency conditions on the coefficients
$b_i$ and $a$ and on the right-hand sides $f_i$ in Case~1.

\medskip

Consider Case~2. To be definite, we assume that
$$
b_1(h)+b_2(h)\in I_1.
$$
In this case, the consistency
condition~\eqref{eqConsistencyZGener} is considered only near the
origin. Let us find out the form of this condition in terms of
problem~\eqref{eqExPinG}, \eqref{eqExBinG}. Let $\tau_\sigma$
denote the vector with the coordinates $(0,(-1)^\sigma)$. Then
$\partial/\partial\tau_\sigma=(-1)^\sigma\partial/\partial y_2$
and
$$
\frac{\partial} {\partial
 \tau_1}\big(U(y)+b_{1}(0)U({\mathcal G}_{1}y)\big)=-U_{y_2}(y)+b_1(0)U_{y_1}({\mathcal
  G}_{1}y),
$$
$$
\frac{\partial} {\partial
  \tau_2}\big(U(y)+b_{2}(0)U({\mathcal G}_{2}y)\big)=U_{y_2}(y)+b_2(0)U_{y_1}({\mathcal G}_{2}y).
$$
Therefore,
$$
\hat{\mathcal B}_{\sigma}(D_y)U=(-1)^{\sigma}U_{y_2}+
b_\sigma(0)U_{y_1},\quad \sigma=1,2.
$$
Since $b_1(0)+b_2(0)=0$, it follows that the operators
$\hat{\mathcal B}_{1}(D_y)$ and $\hat{\mathcal B}_{2}(D_y)$ are
linearly dependent,
$$
\hat{\mathcal B}_{1}(D_y)+\hat{\mathcal B}_{2}(D_y)=0.
$$
Thus, the consistency condition~\eqref{eqConsistencyZGener} for
functions $Z_\sigma\in W^{3/2}(\gamma_\sigma^\varepsilon)$
acquires the form
\begin{equation}\label{eqExConsistencyZ}
\int\limits_{0}^\varepsilon r^{-1}\Bigg|\frac{\partial
Z_1}{\partial y_2}\bigg|_{y=(0,-r)}-\frac{d
Z_2}{dy_2}\bigg|_{y=(0,r)}\Bigg|^2dr<\infty.
\end{equation}

Due to~\eqref{eqExConsistencyZ}, the space $\mathcal
S^{3/2}(\partial G)$ consists of the functions $\{f_i\}\in\mathcal
W^{3/2}(\partial G)$ such that
\begin{equation}\label{eqExConsistencyf}
\int\limits_{0}^\varepsilon r^{-1}\Bigg|\frac{\partial
f_1}{\partial y_2}\bigg|_{y=(0,-r)}-\frac{\partial f_2}{\partial
y_2}\bigg|_{y=(0,r)}\Bigg|^2dr<\infty.
\end{equation}
By Theorem~\ref{thUNonSmFNonConsistGener}, the validity of the
condition $\{f_i\}\in\mathcal S^{3/2}(\partial G)$ is necessary
for any generalized solution of problem~\eqref{eqExPinG},
\eqref{eqExBinG} to belong to $W^2(G)$.

\begin{theorem}\label{thExSmoothfne0}
Let $b_1(0)+b_2(0)=0$ and $b_1(h)+b_2(h)\in I_1$. Then the
following assertions are true.
\begin{enumerate}
\item
If
\begin{gather}
a(0)=0,\quad \frac{\partial
a}{\partial y_2}\bigg|_{y=0}=0,\label{eqExB2vB1CConsistency1}\\
\int\limits_{0}^\varepsilon r^{-1}\Bigg|\frac{\partial
b_1}{\partial y_2}\bigg|_{y=(0,-r)}-\frac{\partial b_2}{\partial
y_2}\bigg|_{y=(0,r)}\Bigg|^2dr<\infty,\label{eqExB2vB1CConsistency2}
\end{gather}
and $u\in W^1(G)$ is a generalized solution of
problem~\eqref{eqExPinG}, \eqref{eqExBinG} with right-hand side
$\{f_0,f_i\}\in L_2(G)\times \mathcal S^{3/2}(\partial G)$, then
$u\in W^2(G)$.
\item If
condition~\eqref{eqExB2vB1CConsistency1}--\eqref{eqExB2vB1CConsistency2}
fails, then there exists a right-hand side $\{f_0,f_i\}\in
L_2(G)\times \mathcal S^{3/2}(\partial G)$ and a generalized
solution $u\in W^1(G)$ of problem~\eqref{eqExPinG},
\eqref{eqExBinG} such that $u\notin W^2(G)$.
\end{enumerate}
\end{theorem}
\begin{proof}
1. By Theorem~\ref{thSmoothfne0Gener}, it suffices to prove that
condition~\eqref{eqExB2vB1CConsistency1}--\eqref{eqExB2vB1CConsistency2}
is equivalent to Condition~\ref{condB2vB1CConsistencyGener}.

For any function $v\in W^{2}(G\setminus\overline{\mathcal
O_{\varkappa_1}(\mathcal K)})$, set
$v_\Omega(y)=v\big(\Omega(y)\big)$, $y\in\Gamma_1$. In this case,
we have
$$
B_1^v(y)=a(y)v_\Omega(y),\quad y\in\gamma_1^{\varepsilon/2},\qquad
B_2^v(y)=0,\quad y\in\gamma_2^{\varepsilon/2}.
$$
Therefore, the functions $B_\sigma^v$ satisfy the consistency
condition~\eqref{eqExConsistencyZ} if and only if
\begin{equation}\label{eqExSmoothfne0-1}
\int\limits_{0}^{\varepsilon/2} r^{-1}\Bigg|\frac{\partial
(av_\Omega)}{\partial y_2}\bigg|_{y=(0,-r)}\Bigg|^2dr=
\int\limits_{0}^{\varepsilon/2} r^{-1}\Bigg|\Big(\frac{\partial
a}{\partial y_2}v_\Omega+a\frac{\partial v_\Omega}{\partial
y_2}\Big)\bigg|_{y=(0,-r)}\Bigg|^2dr<\infty.
\end{equation}
We take $\varepsilon/2$ instead of $\varepsilon$ as the upper
limit of integration because the functions $B_\sigma^v$ look
simpler in this case; clearly, this change does not affect the
convergence of the integral.

Let us prove that condition~\eqref{eqExSmoothfne0-1} is equivalent
to~\eqref{eqExB2vB1CConsistency1}. Suppose that
\eqref{eqExSmoothfne0-1} holds. Take a function $v$ such that
$v_\Omega(y)=y_2$ near the origin; then we have
$$
\frac{\partial (av_\Omega)}{\partial y_2}\bigg|_{y=0}=a(0).
$$
Since the function $\partial (av_\Omega)/{\partial y_2}$ is
continuous near the origin, it follows from the latter relation
and from~\eqref{eqExSmoothfne0-1} that $a(0)=0$. In the similar
way, substituting a function $v$ such that $v_\Omega(y)=1$ near
the origin into~\eqref{eqExSmoothfne0-1}, we obtain $({\partial
a}/{\partial y_2})|_{y=0}=0$.

Conversely, suppose that~\eqref{eqExB2vB1CConsistency1} holds. By
virtue of smoothness of the transformation $\Omega$, we have
$$
v_\Omega,\frac{\partial v_\Omega}{\partial y_2} \in
W^{1/2}(\gamma_1^\varepsilon)\subset
H_1^{1/2}(\gamma_1^\varepsilon)
$$
for any $v\in W^{2}(G\setminus\overline{\mathcal
O_{\varkappa_1}(\mathcal K)})$. It follows from this relation,
from~\eqref{eqExB2vB1CConsistency1}, and from
Lemma~\ref{lAppL3.3'Kondr} that $\partial (av_\Omega)/{\partial
y_2}\in H_0^{1/2}(\gamma_1^\varepsilon)$. Therefore, by
Lemma~\ref{lAppL4.18Kondr}, relation~\eqref{eqExSmoothfne0-1}
follows. Thus, we have proved that part~1 of
Condition~\ref{condB2vB1CConsistencyGener} is equivalent to
condition~\eqref{eqExB2vB1CConsistency1}.

2. Part~2 of Condition~\ref{condB2vB1CConsistencyGener} is
fulfilled if and only if the functions $C+b_1(y)C$ and $C+b_2(y)C$
satisfy the consistency condition~\eqref{eqExConsistencyZ} for any
constant $C$. The latter is equivalent
to~\eqref{eqExB2vB1CConsistency2}.
\end{proof}

Thus, we see that, in Case~2, the smoothness of generalized
solutions depends on the values of the first derivatives of the
coefficients $b_1,b_2$ {\em near} the origin as well as on the
values of the coefficient $a$ and its first derivative {\it at}
the origin.


Consider Case 3.
\begin{theorem}\label{thExNonSmoothfne0}
Let $b_1(0)+b_2(0)\in I_2$ or $b_1(h)+b_2(h)\in I_2$. Then there
exists a right-hand side $\{f_0,0\}$, where $f_0\in L_2(G)$, and a
generalized solution $u\in W^1(G)$ of problem~\eqref{eqExPinG},
\eqref{eqExBinG} such that $u\notin W^2(G)$.
\end{theorem}
\begin{proof}
The band $-1\le\Im\lambda<0$ contains an improper eigenvalue of
problem~\eqref{eqExPEigen}, \eqref{eqExBEigen} (or an improper
eigenvalue of the analogous problem corresponding to the point
$h$). Therefore, this theorem follows from
Theorem~\ref{thNonSmoothuGener}.
\end{proof}

Thus, in Case~3, the smoothness of generalized solutions can be
violated irrespective of the behavior of the coefficient $a$ and
of the derivatives of the coefficients $b_1,b_2$ near the point
$g$.

\subsection{Problem with Regular and Homogeneous Nonlocal
Conditions}\label{subsecEx2} Consider problem~\eqref{eqExPinG},
\eqref{eqExBinG} with regular and homogeneous nonlocal conditions.
By Theorems~\ref{thExuinW^2NoEigen} and~\ref{thExNonSmoothfne0},
the smoothness of generalized solutions preserves in Case~1 and
can be violated in Case~3. Case~2 (the border case) is of
particular interest.

First, we study the case of regular right-hand sides. To be
definite, we again assume that
$$
b_1(h)+b_2(h)\in I_1.
$$

\begin{theorem}\label{thExSmoothf0}
Let $b_1(0)+b_2(0)=0$ and $b_1(h)+b_2(h)\in I_1$. Then the
following assertions are true.
\begin{enumerate}
\item
If
\begin{equation}\label{eqExBv+BCConsist}
a(0)=0,\quad \frac{\partial a}{\partial y_2}\bigg|_{y=0}=0
\end{equation}
and $u\in W^1(G)$ is a generalized solution of
problem~\eqref{eqExPinG}, \eqref{eqExBinG} with right-hand side
$\{f_0,f_i\}\in L_2(G)\times\mathcal S^{3/2}(\partial G)$, where
$f_i(0)=0$, then $u\in W^2(G)$.
\item If condition~\eqref{eqExBv+BCConsist} fails, then there exists
a right-hand side $\{f_0,f_i\}\in L_2(G)\times\mathcal
H_0^{3/2}(\partial G)$, where $f_i(y)=0$ in a neighborhood of the
origin, and a generalized solution $u\in W^1(G)$ of
problem~\eqref{eqExPinG}, \eqref{eqExBinG}  such that $u\notin
W^2(G)$.
\end{enumerate}
\end{theorem}
\begin{proof}
1. By virtue of Theorem~\ref{thSmoothf0Gener} and
Corollary~\ref{corSmoothf0NeighbK}, it suffices to prove that
condition~\eqref{eqExBv+BCConsist} is equivalent to
Condition~\ref{condBv+BCConsistGener}.

By Definition~\ref{defAdmitGener}, a function $v\in
W^{2}(G\setminus\overline{\mathcal O_{\varkappa_1}(\mathcal K)})$
is admissible if there exist constants $C$ and $C_h$ such that
\begin{equation}\label{eqExSmoothf0-1'}
\begin{aligned}
 a(0)v_\Omega(0)+C+b_1(0)C=0,\quad
C+b_2(0)C=0,\\
 a(h)v_\Omega(h)+C_h+b_1(h)C_h=0,\quad
C_h+b_2(h)C_h=0,
\end{aligned}
\end{equation}
where $v_\Omega(y)=v\big(\Omega(y)\big)$, $y\in\Gamma_1$.

Let $\xi\in C^\infty(\mathbb R^2)$ be a cut-off function such that
$$
\supp\xi\subset \mathcal O_{\delta}(\Omega(0)),\qquad
\xi(y)=1,\quad y\in\mathcal O_{\delta/2}(\Omega(0)),
$$
where $\delta>0$ is so small that $\Omega(h)\notin\mathcal
O_{\delta}(\Omega(0))$. Since $b_1(h)+b_2(h)\in I_1$, the
consistency condition~\eqref{eqConsistencyZGener} is considered
only near the origin. Therefore, if $v$ is an admissible function,
$C,C_h$ are admissible constants corresponding to $v$, and
Condition~\ref{condBv+BCConsistGener} holds (fails) for $v$ and
$C$, then the function $\xi v$ is admissible, $C,0$ are admissible
constants corresponding to $\xi v$, and
Condition~\ref{condBv+BCConsistGener} holds (respectively, fails)
for $\xi v$ and $C$. Thus, it suffices to consider only functions
$v$ supported in $\mathcal O_{\delta}(\Omega(0))$ (i.e., functions
$v_\Omega$ supported near the origin) and assume that $C_h=0$.

First, we study the situation in which $b_2(0)\ne-1$. In this
case, according to~\eqref{eqExSmoothf0-1'}, a function $v$
supported in $\mathcal O_{\delta}(\Omega(0))$ is admissible if and
only if
\begin{equation}\label{eqExSmoothf0-1}
a(0)v_\Omega(0)=0,
\end{equation}
while the corresponding set of admissible vectors (constants in
our case) consists of the unique constant $C=0$ (recall that $C_h$
is supposed to equal zero). Therefore,
Condition~\ref{condBv+BCConsistGener} holds if and only if the
relation
\begin{equation}\label{eqExSmoothf0-2}
\int\limits_{0}^{\varepsilon/2} r^{-1}\Bigg|\frac{\partial
(av_\Omega)}{\partial y_2}\bigg|_{y=(0,-r)}\Bigg|^2dr=
\int\limits_{0}^{\varepsilon/2} r^{-1}\Bigg|\Big(\frac{\partial
a}{\partial y_2}v_\Omega+a\frac{\partial v_\Omega}{\partial
y_2}\Big)\bigg|_{y=(0,-r)}\Bigg|^2dr<\infty
\end{equation}
holds for any $v_\Omega$ satisfying~\eqref{eqExSmoothf0-1}.
Suppose that~\eqref{eqExBv+BCConsist} is fulfilled. Then any
function $v$ supported in $\mathcal O_{\delta}(\Omega(0))$ is
admissible (because $a(0)=0$), and repeating the arguments of the
proof of Theorem~\ref{thExSmoothfne0}
yields~\eqref{eqExSmoothf0-2}.

Conversely, suppose that~\eqref{eqExSmoothf0-2} holds for any
function $v_\Omega$ satisfying~\eqref{eqExSmoothf0-1}. Clearly, a
function $v$ such that $v_\Omega(y)=y_2$ near the origin
satisfies~\eqref{eqExSmoothf0-1}. Substituting the function
$v_\Omega$ into~\eqref{eqExSmoothf0-2}, we obtain $a(0)=0$ (cf.
the proof of Theorem~\ref{thExSmoothfne0}). Therefore, any
function $v$  supported in $\mathcal O_{\delta}(\Omega(0))$ is
admissible. Substituting $v_\Omega(y)=1$
into~\eqref{eqExSmoothf0-2}, we obtain $({\partial a}/{\partial
y_2 })|_{y=0}=0$.

2. It remains to study the situation in which $b_2(0)=-1$. This
implies $b_1(0)=1$. In this case, according
to~\eqref{eqExSmoothf0-1'}, any function $v$ supported in
$\mathcal O_{\delta}(\Omega(0))$ is admissible, while the
corresponding set of admissible vectors (constants in our case)
consists of the unique constant $C=-a(0)v_\Omega(0)/2$ (while
$C_h$ is supposed to equal zero). Therefore,
Condition~\ref{condBv+BCConsistGener} holds if and only if the
relation
\begin{multline}\label{eqExSmoothf0-3}
\int\limits_{0}^{\varepsilon/2} r^{-1}\Bigg|\frac{\partial
(av_\Omega)}{\partial y_2}\bigg|_{y=(0,-r)}+C\bigg(\frac{\partial
b_1}{\partial y_2}\bigg|_{y=(0,-r)}-\frac{\partial b_2}{\partial
y_2}\bigg|_{y=(0,r)}\bigg)\Bigg|^2dr\\
=\int\limits_{0}^{\varepsilon/2} r^{-1}\Bigg|\Big(\frac{\partial
a}{\partial y_2}v_\Omega+a\frac{\partial v_\Omega}{\partial
y_2}\Big)\bigg|_{y=(0,-r)}\\-\frac{a(0)v_\Omega(0)}{2}\bigg(\frac{\partial
b_1}{\partial y_2}\bigg|_{y=(0,-r)}-\frac{\partial b_2}{\partial
y_2}\bigg|_{y=(0,r)}\bigg)\Bigg|^2dr<\infty
\end{multline}
holds for any  $v$ supported in $\mathcal O_{\delta}(\Omega(0))$.
Suppose that condition~\eqref{eqExBv+BCConsist} is fulfilled.
Then, similarly to the above, we see that~\eqref{eqExSmoothf0-2}
holds for any function $v_\Omega$; hence, \eqref{eqExSmoothf0-3}
also holds for any $v_\Omega$ (because $a(0)=0$).

Conversely, suppose that~\eqref{eqExSmoothf0-3} is fulfilled.
Consider a function $v$ such that $v_\Omega(y)=y_2$ near the
origin and substitute it into~\eqref{eqExSmoothf0-3}. Since
$v_\Omega(0)=0$ and $({\partial v_\Omega}/{\partial y_2
})|_{y=0}=1$, we infer from~\eqref{eqExSmoothf0-3} that $a(0)=0$
similarly to the above. Therefore, relation~\eqref{eqExSmoothf0-3}
coincides with~\eqref{eqExSmoothf0-2}. Now, repeating the above
arguments, we obtain $({\partial a}/{\partial y_2 })|_{y=0}=0$,
which completes the proof.
\end{proof}

Clearly, condition~\eqref{eqExBv+BCConsist} is weaker than
condition~\eqref{eqExB2vB1CConsistency1}--\eqref{eqExB2vB1CConsistency2}:
we impose no restrictions on the behavior of the coefficients
$b_1, b_2$ in condition~\eqref{eqExBv+BCConsist}. The absence of
those restrictions is ``compensated'' by the fact that nonlocal
conditions are {\em regular}, i.e., $\{f_i\}\in\mathcal
S^{3/2}(\partial G)$ and $f_i(0)=0$.

\medskip

Finally, we consider the case of homogeneous nonlocal conditions.
In this case, assertion~1 of Theorem~\ref{thExSmoothf0} implies
that the validity of condition~\eqref{eqExBv+BCConsist} is
sufficient for any generalized solution to be smooth.  We prove
that this condition is also necessary in the following cases (see
Figs.~\ref{figCaseA}, \ref{figCaseB}, and~\ref{figCaseC}):
\begin{description}
\item[Case~A]  $\supp a(\Omega^{-1}(y))|_{\Omega(\Gamma_1)}\subset G$.
\begin{figure}[h]
{ \hfill\epsfbox{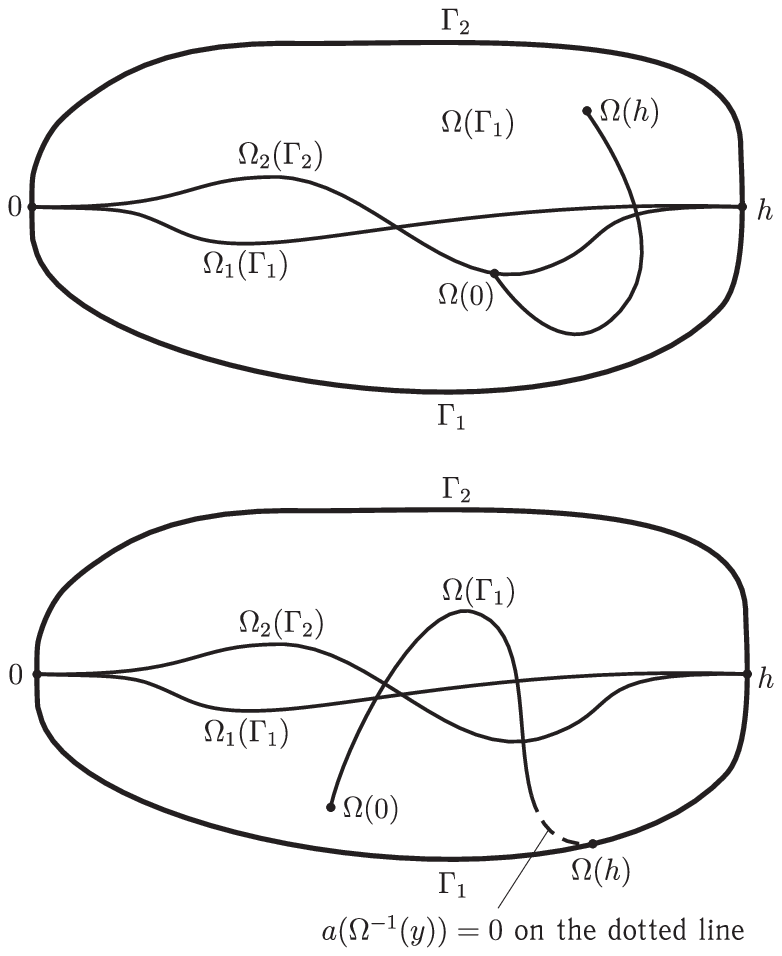}\hfill\ } \caption{Case~A.}
   \label{figCaseA}
\end{figure}
\item[Case~B] $\Omega(0)\in G$ and $\Omega(0)\notin\Omega_1(\Gamma_1)\cup\Omega_2(\Gamma_2)$.
\begin{figure}[h]
{ \hfill\epsfbox{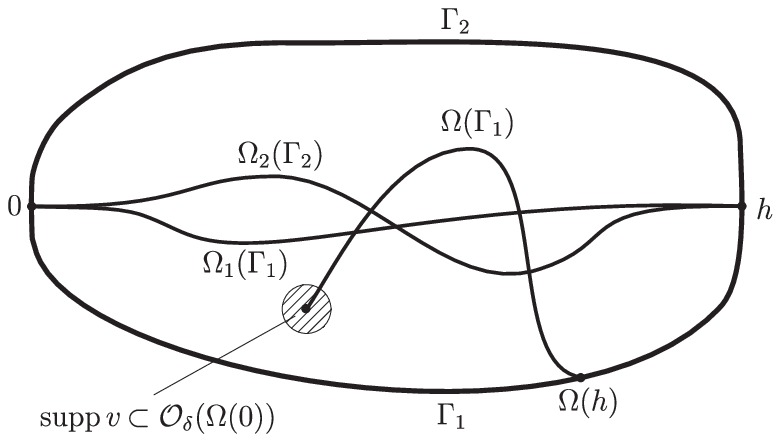}\hfill\ } \caption{Case~B.}
   \label{figCaseB}
\end{figure}
\item[Case~C] We have
\begin{equation}\label{eqExOmegaginGamma1}
\Omega(0)\in\Gamma_1,\qquad
\Omega(\Omega(0))\notin\Omega_1(\Gamma_1)\cup\Omega_2(\Gamma_2).
\end{equation}
\begin{equation}\label{eqExaOmegagne0}
a(\Omega(0))\ne0.
\end{equation}
\end{description}
\begin{figure}[h]
{ \hfill\epsfbox{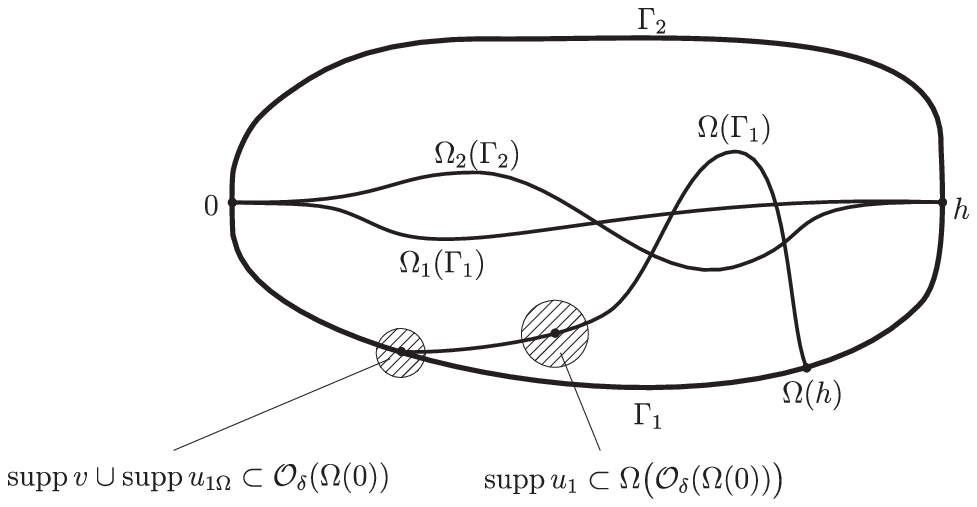}\hfill\ } \caption{Case~C.}
   \label{figCaseC}
\end{figure}

\begin{corollary}\label{corExSmoothf0Homog}
Let $b_1(0)+b_2(0)=0$ and $b_1(h)+b_2(h)\in I_1$. Suppose that
either Case~{\rm A}, or Case~{\rm B}, or Case~{\rm C} takes place.
If condition~\eqref{eqExBv+BCConsist} fails, then there exists a
right-hand side $\{f_0,0\}$, where $f_0\in L_2(G)$, and a
generalized solution $u\in W^1(G)$ of problem~\eqref{eqExPinG},
\eqref{eqExBinG} such that $u\notin W^2(G)$.
\end{corollary}
\begin{proof}
1. First, we assume that Case~A takes place. It follows from the
continuity of the transformations $\Omega_i$ and $\Omega$ that the
operators $\mathbf B_i^2$ satisfy condition~\eqref{eqSeparK23'''}
with any $\rho$ such that $0<\rho<\dist(\supp
a(\Omega^{-1}(y))|_{\Omega(\Gamma_1)},\,\partial G)$. Therefore,
the conclusion of this corollary follows from
Corollary~\ref{corSmoothf0B^2vInsideG} and
Remark~\ref{remCorGener}.

2. Now we assume that Case~B takes place. As before, we can
suppose that Condition~\ref{condBv+BCConsistGener} is violated for
an admissible function $v$ supported in an arbitrarily small
$\delta$-neighborhood $\mathcal O_\delta(\Omega(0))$ of the point
$\Omega(0)$. The number $\delta$ can be chosen so small that
$$
v(y)|_{\Gamma_i}\equiv0,\qquad v(\Omega_i(y))|_{\Gamma_i}=0,\qquad
\supp v(\Omega(y))|_{\Gamma_1}\subset\Gamma_1\cap\mathcal
O_{\varkappa_2}(0).
$$
Therefore, the function $v$ satisfies
relations~\eqref{eqSmoothf0B^2vComp}, and the conclusion of this
corollary follows from Corollary~\ref{corSmoothf0B^2vComp} and
Remark~\ref{remCorGener}.

3. Finally, we assume that Case~C takes place. Again we can
suppose that Condition~\ref{condBv+BCConsistGener} is violated for
an admissible function $v$ supported in $\mathcal
O_\delta(\Omega(0))$. By virtue of
relations~\eqref{eqExOmegaginGamma1}, the number $\delta$ can be
chosen so small that
\begin{equation}\label{eqExSmoothf0Homog1}
v(\Omega_i(y))|_{\Gamma_i}\equiv0,
\end{equation}
\begin{equation}\label{eqExSmoothf0Homog2}
\supp v(\Omega(y))|_{\Gamma_1}\subset\Gamma_1\cap\mathcal
O_{\varkappa_2}(0).
\end{equation}
Let $f_i$ be the functions from assertion~$2$ of
Theorem~$\ref{thSmoothf0}$, constructed accordingly to the scheme
suggested in the proof of Theorem~\ref{thSmoothf0Gener} . It
follows from~\eqref{eqExSmoothf0Homog1}
and~\eqref{eqExSmoothf0Homog2} that
$$
\supp f_1\subset\Gamma_1\cap\big(\mathcal O_{\varkappa_2}(0)\cup
\mathcal O_\delta(\Omega(0))\big),\qquad \supp
f_2\subset\Gamma_2\cap\mathcal O_{\varkappa_2}(0).
$$

If we construct a function $u_1\in H_0^2(G)$ such that
\begin{equation}\label{eqExSmoothf0Homog3}
u_1|_{\Gamma_i}+\mathbf B_{i}^1 u_1+\mathbf B_{i}^2 u_1=
   f_{i}(y),\quad
    y\in \Gamma_i\setminus\mathcal O_{\varkappa_2}(\mathcal K),\ i=1, \dots,
    N,
\end{equation}
\begin{equation}\label{eqExSmoothf0Homog4}
u_1|_{\Gamma_i}+\mathbf B_{i}^1 u_1+\mathbf B_{i}^2 u_1=0,\quad
y\in \Gamma_i\cap\mathcal O_{\varkappa_2}(\mathcal K),\ i=1,
\dots,
    N,
\end{equation}
then the conclusion of this corollary will follow from
Lemma~\ref{lHomogeneous}, Corollary~\ref{corSmoothf0Homog}, and
Remark~\ref{remCorGener}.

Let us construct the function $u_1$. To do this, we consider a
function $u_{1\Omega}\in W^2\big(\mathcal
O_\delta(\Omega(0))\big)$ supported in $\mathcal
O_\delta(\Omega(0))$ (see Fig.~\ref{figCaseC}) such that
$$
u_{1\Omega}(y)=f_1(y)/a(y),\quad y\in \Gamma_1\cap \mathcal
O_\delta(\Omega(0)),
$$
where $\delta$ is so small that $a(y)\ne0$ for
$y\in\overline{\mathcal O_\delta(\Omega(0))}$ (the existence of
such a $\delta$ follows from~\eqref{eqExaOmegagne0} and from the
continuity of $a(y)$).

Now we set $u_1(y)=u_{1\Omega}(\Omega^{-1}(y))$ for
$y\in\Omega\big(\mathcal O_\delta(\Omega(0))\big)$ and $u_1(y)=0$
for $y\notin\Omega\big(\mathcal O_\delta(\Omega(0))\big)$. Suppose
that $\delta$ is so small that
$$
\Gamma_i\cap\Omega\big(\mathcal
O_\delta(\Omega(0))\big)=\varnothing,\quad\!
\Omega_i(\Gamma_i)\cap\Omega\big(\mathcal
O_\delta(\Omega(0))\big)=\varnothing,\quad\! \mathcal
O_\delta(\Omega(0))\cap\mathcal O_{\varkappa_2}(0)=\varnothing
$$
(the existence of such a $\delta$ follows
from~\eqref{eqExOmegaginGamma1} and from the continuity of the
transformation $\Omega$). Then we have
$$
u_1|_{\Gamma_i}=0,\qquad u_1(\Omega_i(y))|_{\Gamma_i}=0,
$$
$$
a(y)u_1(\Omega(y))=f_1(y),\quad y\in \Gamma_1\setminus\mathcal
O_{\varkappa_2}(0),
$$
$$
u_1(\Omega(y))=0,\quad y\in \Gamma_1\cap\mathcal
O_{\varkappa_2}(0).
$$
Therefore, the function $u_1$ satisfies
relations~\eqref{eqExSmoothf0Homog3}
and~\eqref{eqExSmoothf0Homog4}, and the theorem is proved.
\end{proof}

\appendix
\section{}

This appendix is included for the reader's convenience. Here we
have collected some known results on weighted spaces and on
properties of nonlocal operators, which are most frequently
referred to in the main part of the paper.

\subsection{Properties of weighted spaces}
In this subsection, we formulate some results concerning
properties of weighted spaces introduced in
Sec.~\ref{subsectStatement}. Set
$$
K=\{y\in{\mathbb R}^2:\ r>0,\ |\omega|<\omega_0\},
$$
$$
\gamma_{\sigma}=\{y\in\mathbb R^2:\ r>0,\ \omega=(-1)^\sigma
\omega_0\}\qquad (\sigma=1,2).
$$
\begin{lemma}[see Lemma~4.9 in~\cite{KondrTMMO67}]\label{lAppL4.9Kondr_5.2KovSk}
Let a function $u\in W^k(K)$, where $k\ge1$, be compactly
supported. Then $u\in H^k_{b}(K)$ for any $b>k-1$.
\end{lemma}

\begin{lemma}[see Lemma~2.1 in~\cite{GurRJMP03}]\label{lAppL2.1GurRJMP03}
Let a function $u\in W^2(K)$ be compactly supported, and let
$u(0)=0$. Then $u\in H^2_{b}(K)$ for any $b>0$.
\end{lemma}

\begin{lemma}[see Lemma~3.3$'$ in~\cite{KondrTMMO67}]\label{lAppL3.3'Kondr}
Let a function $u\in H_b^k(K)$, where $k\ge0$ and $b\in\mathbb R$,
be compactly supported. Suppose that $p\in C^k(\overline{K})$ and
$p(0)=0$. Then $pu\in H^k_{b-1}(K)$.
\end{lemma}

\begin{lemma}[see Lemma~4.8 in~\cite{KondrTMMO67}]\label{lAppL4.8Kondr}
Let a function $u\in W^{1}(K)$ be compactly supported. Suppose
that
$$
\int\limits_{\gamma_\sigma} r^{-1}|u|^2dr<\infty,
$$
where $\sigma=1$ or $\sigma=2$. Then $u\in H_0^1(K)$.
\end{lemma}

\begin{lemma}[see Lemma~4.18 in~\cite{KondrTMMO67}]\label{lAppL4.18Kondr}
Let a function $\varphi\in H_0^{1/2}(\gamma_\sigma)$, where
$\sigma=1$ or $\sigma=2$, be compactly supported. Then
$$
\int\limits_{\gamma_\sigma} r^{-1}|\varphi|^2dr<\infty.
$$
\end{lemma}

\begin{lemma}[see Lemma~4.20 in~\cite{KondrTMMO67}]\label{lAppL4.20Kondr}
The function $r^{i\lambda_0}\Phi(\omega)\ln^s r$, where
$\Im\lambda_0=-(k-1)$, belongs to $W^k(K\cap\{|y|<1\})$ if and
only if it is a homogeneous polynomial in $y_1,y_2$ of order
$k-1$.
\end{lemma}

Denote by $\mathcal G$ the operator which is the composition of
rotation about the origin and homothety.

\begin{lemma}[see Lemma~2.2 in~\cite{GurRJMP03}]\label{lAppL2.2GurRJMP03}
Let a function $u\in W^{1}(\mathbb R^2)$ be compactly supported.
Then $u(\mathcal Gy)-u(y)\in H_0^1(\mathbb R^2)$.
\end{lemma}

\subsection{Nonlocal Problems in Plane Angles in Weighted Spaces}

In this subsection and in the next one, we formulate some
properties of solutions of problem~\eqref{eqPinK}, \eqref{eqBinK}
in the spaces~\eqref{eqSpacesCal1} and~\eqref{eqSpacesCal2}.
First, we consider the case of weighted spaces.

For convenience, we rewrite the problem:
\begin{equation}\label{eqAppPBinK}
\begin{gathered}
\mathbf P_jU_j=F_j(y)\quad (y\in K_j^\varepsilon),\\
 \mathbf
B_{j\sigma}U\equiv\sum\limits_{k,s}
      b_{j\sigma ks}(y)U_k({\mathcal G}_{j\sigma ks}y)=\Phi_{j\sigma}(y)\quad \quad (y\in
\gamma_{j\sigma}^\varepsilon),
\end{gathered}
\end{equation}
where
$$
\mathbf P_jv=\sum\limits_{i,k=1}^2p_{jik}(y)v_{y_iy_k}+
\sum\limits_{k=1}^2p_{jk}(y)v_{y_k}+p_{j0}(y)v
$$
(see Sec.~\ref{subsectStatementNearK}). Along with
problem~\eqref{eqAppPBinK}, we consider the following model
problem in the unbounded angles:
\begin{equation}\label{eqAppCalPBinK}
\begin{gathered}
\mathcal P_jU_j=F_j(y)\quad (y\in K_j),\\
 \mathcal
B_{j\sigma}U\equiv\sum\limits_{k,s}
      b_{j\sigma ks}(0)U_k({\mathcal G}_{j\sigma ks}y)=\Phi_{j\sigma}(y)\quad (y\in
\gamma_{j\sigma}),
\end{gathered}
\end{equation}
where
$$
\mathcal P_jv=\sum\limits_{i,k=1}^2p_{jik}(0)v_{y_iy_k}.
$$

\begin{lemma}[see Lemma~2.3 in~\cite{GurMatZam05}]\label{lAppL2.3GurMatZam05}
Let $U$ be a solution of problem~\eqref{eqAppPBinK}
{\rm(}or~\eqref{eqAppCalPBinK}{\rm)} such that
$$U_j\in W^{2}(K_{j}^{D_\chi
\varepsilon}\cap\{|y|>\delta\})\qquad\forall \delta>0,
$$
$$
U\in \mathcal H_{b-2}^0(K^{D_\chi \varepsilon}),
$$
where $D_\chi$ is given by~\eqref{eqd1d2} and $b\in\mathbb R$.
Suppose that
$$
\{F_j\}\in \mathcal H_{b}^{0}(K^{\varepsilon}),\qquad
\{\Phi_{j\sigma}\}\in \mathcal H_{b}^{3/2}(\gamma^\varepsilon).
$$
Then $ U\in\mathcal  H_{b}^{2}(K^{\varepsilon}). $
\end{lemma}

Consider the asymptotics of solutions of
problem~\eqref{eqAppCalPBinK}.
\begin{lemma}[see Lemma~2.1 in~\cite{GurPetr03}]\label{lAppL2.1GurPetr03}
The function
\begin{equation}\label{eqAppL2.1GurPetr03}
U=r^{i\lambda_0}\sum\limits_{l=0}^m\frac{\displaystyle
1}{\displaystyle l!}(i\ln r)^l\varphi^{(m-l)}(\omega),
\end{equation}
is a solution of homogeneous problem~\eqref{eqAppCalPBinK} if and
only if~$\lambda_0$ is an eigenvalue of the
operator~$\tilde{\mathcal L}(\lambda)$ and $\varphi^{(0)}, \dots,
\varphi^{(\varkappa-1)}$ is a Jordan chain corresponding to the
eigenvalue~$\lambda_0;$ here $m\le\varkappa-1$.
\end{lemma}

Any solution of the kind~\eqref{eqAppL2.1GurPetr03} is called a
\textit{power solution}.

\begin{theorem}[see Theorem~2.2 and Remark~2.2 in~\cite{GurPetr03}]\label{lAppTh2.2GurPetr03}
Let $\{F_j\}\in \mathcal H_b^0(K)\cap \mathcal H_{b'}^0(K)$ and
$\{\Phi_{j\sigma}\}\in \mathcal H_b^{3/2}(\gamma)\cap \mathcal
H_{b'}^{3/2}(\gamma)$, where $b>b'$. Suppose that the line
$\Im\lambda=b'-1$ contains no eigenvalues of the
operator~$\tilde{\mathcal L}(\lambda)$. If $U$ is a solution of
problem~\eqref{eqAppCalPBinK} belonging to the space~$\mathcal
H_b^{2}(K)$, then
$$
U=
\sum\limits_{n=1}^{n_0}\sum\limits_{q=1}^{J_n}\sum\limits_{m=0}^{\varkappa_{qn}-1}
c_n^{(m,q)}W_n^{(m,q)}(\omega, r)+U'.
$$
Here $\lambda_1, \dots, \lambda_{n_0}$ are eigenvalues of
$\tilde{\mathcal L}(\lambda)$ located in the band
$b'-1<\Im\lambda<b-1;$
$$
W_n^{(m,q)}(\omega, r)=r^{i\lambda_n}\sum\limits_{l=0}^m
   \frac{\displaystyle 1}{\displaystyle l!}(i\ln r)^l\varphi_n^{(m-l,q)}(\omega)
$$
are power solutions of homogeneous
problem~\eqref{eqAppCalPBinK}$;$
$$
 \{\varphi_n^{(0,q)}, \dots, \varphi_n^{(\varkappa_{qn}-1,q)}:\ q=1,\dots,J_n\}
$$
is a canonical system of Jordan chains of the
operator~$\tilde{\mathcal L}(\lambda)$ corresponding to the
eigenvalue~$\lambda_n;$ $c_n^{(m,q)}$ are some complex
constants$;$ finally, $U'$ is a solution of
problem~\eqref{eqAppCalPBinK} belonging to the space $\mathcal
H_{b'}^{2}(K)$.
\end{theorem}

If the right-hand sides of problem~\eqref{eqAppCalPBinK} are of
particular form, then there exist solutions of particular form.
Let
\begin{equation}\label{eqAppL4.3GurPetr03}
  F_j(\omega, r)=r^{i\lambda_0-2}\sum\limits_{l=0}^M\frac{\displaystyle 1}{\displaystyle l!}(i\ln
  r)^l
  f_j^{(l)}(\omega),\qquad
\Phi_{j\sigma}(r)=r^{i\lambda_0}
\sum\limits_{l=0}^M\frac{\displaystyle 1}{\displaystyle l!}(i\ln
r)^l \psi_{j\sigma}^{(l)},
\end{equation}
where
$$
  f_j^{(l)}\in L^2(-\omega_j,\omega_j),\qquad \psi_{j\sigma}^{(l)}\in{\mathbb
  C},\qquad \lambda_0\in\mathbb C.
$$
If~$\lambda_0$ is an eigenvalue of the operator~$\tilde{\mathcal
L}(\lambda)$, then denote by~$\varkappa(\lambda_0)$ the greatest
of the partial multiplicities (see~\cite{GS}) of this eigenvalue;
otherwise, set $\varkappa(\lambda_0)=0$.
\begin{lemma}[see Lemma~4.3 in~\cite{GurPetr03}]\label{lAppL4.3GurPetr03}
For problem~\eqref{eqAppCalPBinK} with right-hand
side~$\{F_j,\Phi_{j\sigma}\}$ given by~\eqref{eqAppL4.3GurPetr03},
there exists a solution
 \begin{equation}\label{eqAppL4.3GurPetr03'}
U= r^{i\lambda_0}\sum\limits_{l=0}^{M+\varkappa(\lambda_0)}
\frac{\displaystyle 1}{\displaystyle l!}(i\ln r)^l
u^{(l)}(\omega),
 \end{equation}
where $u^{(l)}\in \prod\limits_j W^{2}(-\omega_j,\omega_j)$. A
solution of such a form is unique if~$\varkappa(\lambda_0)=0$
{\rm(}i.e., $\lambda_0$ is not an eigenvalue of~$\tilde{\mathcal
L}(\lambda)${\rm)}. If~$\varkappa(\lambda_0)>0$, then the
solution~\eqref{eqAppL4.3GurPetr03'} is defined accurate to an
arbitrary linear combination of power
solutions~\eqref{eqAppL2.1GurPetr03} corresponding to the
eigenvalue~$\lambda_0$.
\end{lemma}

Note that Theorem~\ref{lAppTh2.2GurPetr03} and
Lemma~\ref{lAppL4.3GurPetr03} were earlier proved in~\cite{SkMs86}
for the case in which the operators $\mathcal G_{j\sigma ks}$ are
rotations only (but not homothety).

The following result is a modification of
Theorem~\ref{lAppTh2.2GurPetr03} for the case in which the line
$\Im\lambda=-1$ contains the unique eigenvalue $\lambda_0=-i$ of
$\tilde{\mathcal L}(\lambda)$ and this eigenvalue is proper (see
Definition~\ref{defRegEigVal}).

\begin{lemma}[see Lemma~3.4 in~\cite{GurRJMP03}]\label{lAppL3.4GurRJMP03}
Let $U\in \mathcal H_b^{2}(K)$, where $b>0$, be a solution of
problem~\eqref{eqAppCalPBinK} with right-hand side $\{F_j\}\in
\mathcal H_b^0(K)\cap \mathcal H_0^0(K)$,
$\{\Phi_{j\sigma}\}\in\mathcal H_b^{3/2}(\gamma)\cap\mathcal
H_0^{3/2}(\gamma)$. Suppose that the closed band
$-1\le\Im\lambda\le b-1$ contains only the eigenvalue
$\lambda_0=-i$ of $\tilde{\mathcal L}(\lambda)$ and this
eigenvalue is proper. Then $D^\alpha U\in \mathcal H_0^0(K)$ for
$|\alpha|=2$.
\end{lemma}

Finally, we formulate the result that allows one to reduce
nonlocal problems with nonhomogeneous boundary conditions to those
with homogeneous boundary conditions.

\begin{lemma}[see Lemma~8.1 in~\cite{GurRJMP04}]\label{lAppL8.1GurRJMP04}
For any  function $\{\Phi_{j\sigma}\}\in \mathcal
H_b^{3/2}(\gamma)$, there exists a  function $U\in \mathcal
H_b^{2}(K)$ such that
$$
U_j(y)|_{\gamma_{j\sigma}}=f_{j\sigma}(y),\quad
\sum\limits_{(k,s)\ne(j,\sigma)}b_{j\sigma ks}(y)U_k(\mathcal
G_{j\sigma ks}y)|_{\gamma_{j\sigma}}=0 \quad
(y\in\gamma_{j\sigma}).
$$
\end{lemma}

\subsection{Nonlocal Problems in Plane Angles in Sobolev Spaces}
In this subsection, we formulate properties of solutions of
problems~\eqref{eqAppPBinK} and~\eqref{eqAppCalPBinK} with
right-hand sides from Sobolev spaces.

The following lemma deals with the case in which the line
$\Im\lambda=-1$ is free of eigenvalues of $\tilde{\mathcal
L}(\lambda)$.

\begin{lemma}[see Lemma~2.4 and Corollary~2.1 in~\cite{GurRJMP03}]\label{lAppL2.4GurRJMP03}
Let the line $\Im\lambda=-1$ contains no eigenvalues of
$\tilde{\mathcal L}(\lambda)$. Suppose that
$$
\{\Phi_{j\sigma}\}\in\mathcal W^{3/2}(\gamma^\varepsilon),\qquad
\Phi_{j\sigma}(0)=0.
$$
Then there exists a compactly supported function
$$
V\in\mathcal  W^2(K)\cap\mathcal H_{b}^2(K),
$$
where $b$ is an arbitrary positive number, such that
$$
\{{\mathbf P}_{j}V_j\}\in\mathcal H_0^0(K^\varepsilon),\qquad
\{{\mathbf
B}_{j\sigma}V|_{\gamma_{j\sigma}^\varepsilon}-\Phi_{j\sigma}\}\in
\mathcal H_0^{3/2}(\gamma^\varepsilon).
$$
\end{lemma}

Now we consider the situation where the line $\Im\lambda=-1$
contains the unique eigenvalue $\lambda_0=-i$ of $\tilde{\mathcal
L}(\lambda)$ and it is proper (see Definition~\ref{defRegEigVal}).
In this case, we will use the following result instead of
Lemma~\ref{lAppL2.4GurRJMP03}.

\begin{lemma}[see Lemma~3.3 and Corollary~3.1 in~\cite{GurRJMP03}]\label{lAppL3.3GurRJMP03}
Let the line $\Im\lambda=-1$ contains the unique eigenvalue
$\lambda_0=-i$ of $\tilde{\mathcal L}(\lambda)$ and it is proper.
Suppose that
$$
\{\Phi_{j\sigma}\}\in\mathcal W^{3/2}(\gamma^\varepsilon),\qquad
\Phi_{j\sigma}(0)=0,
$$
and the functions $\Phi_{j\sigma}$ satisfy the consistency
condition~\eqref{eqConsistencyZ}. Then there exists a compactly
supported function
$$
V\in\mathcal  W^2(K)\cap\mathcal H_{b}^2(K),
$$
where $b$ is an arbitrary positive number, such that
$$
\{{\mathbf P}_{j}V_j\}\in\mathcal H_0^0(K^\varepsilon),\qquad
\{{\mathbf
B}_{j\sigma}V|_{\gamma_{j\sigma}^\varepsilon}-\Phi_{j\sigma}\}\in
\mathcal H_0^{3/2}(\gamma^\varepsilon).
$$
\end{lemma}

\begin{lemma}[see Lemma~3.1 in~\cite{GurRJMP03}]\label{lAppL3.1GurRJMP03}
Let the line $\Im\lambda=-1$ contains the unique eigenvalue
$\lambda_0=-i$ of $\tilde{\mathcal L}(\lambda)$ and it is proper.
Suppose that $U\in\mathcal W^2(K)$ is a compactly supported
solution of problem~\eqref{eqAppPBinK}
{\rm(}or~\eqref{eqAppCalPBinK}{\rm)} and $U(0)=0$. Then the
functions $\Phi_{j\sigma}$ satisfy the consistency
condition~\eqref{eqConsistencyZ}.
\end{lemma}

\medskip

{\bf Acknowledgements.}  This research has been supported by
Russian Foundation for Basic Research (grant no~04-01-00256). A
part of the results has been obtained during the author's visit to
Professor H.~Amann at the Institute for Mathematics, Zurich
University, in the framework of INTAS grant~YSF~2002-008. The
author also expresses his gratitude to Professor
A.~L.~Skubachevskii for attention to this work.

\end{document}